\documentclass{amsart}
\usepackage{tikz}
\usepackage{xcolor}
\usepackage{amssymb,latexsym,amsmath,extarrows}
\usepackage{graphicx,mathrsfs,comment}
\usepackage{hyperref,url}
\usepackage[margin=1in, centering]{geometry}

\usepackage{amstext}
\usepackage{bbm}

\numberwithin{equation}{section}

\renewcommand{\theequation}{\arabic{section}.\arabic{equation}}

\newtheorem{theorem}{Theorem}[section]

\newtheorem{lemma}[theorem]{Lemma}

\newtheorem{proposition}[theorem]{Proposition}
\newtheorem{remark}[theorem]{Remark}

\newtheorem{definition}[theorem]{Definition}
\newtheorem{corollary}[theorem]{Corollary}

\newcommand{\ga}{\gamma}
\newcommand{\Ga}{\Gamma}
\newcommand{\de}{\delta}

\newcommand{\De}{\Delta}
\newcommand{\e}{\epsilon}

\newcommand{\la}{\lambda}

\newcommand{\si}{\sigma}

\newcommand{\om}{\omega}

\newcommand{\mj}{\mathbf{j}}

\newcommand{\p}{\partial}

\newcommand{\wt}{\widetilde}
\newcommand{\wh}{\widehat}

\newcommand{\ZZ}{\mathbb{Z}}

\newcommand{\ZN}{\mathbb{N}}

\newcommand{\Id}{{\mathbf{1}}}

\newcommand{\bI}{{\bf I}}

\newcommand{\cC}{{\mathcal C}}

\newcommand{\R}{\mathbb{R}}

\newcommand{\dist}{{\rm dist}}

\newcommand{\bnu}{\boldsymbol{\nu}}

\newcommand{\bv}{\mathbf{v}}
\newcommand{\bde}{\mathbf{e}}
\newcommand{\bla}{\boldsymbol{\lambda}}
\newcommand{\cI}{\mathcal I}

\newcommand{\laprod}{\lambda_{\mathrm{prod}}}

\newcommand{\tc}{\tau_\circ}

\begin{document}

\title[]{Sharp Rational points counting near nondegenerate curves in $\mathbb{R}^n$}

\author{Shengwen Gan} \address{ Shengwen Gan\\  Department of Mathematics\\ Sun Yat-sen University\\ Guangzhou, 510275, P.R. China} \email{shengwengan2018@gmail.com}

\author{Shaoming Guo} \address{ Shaoming Guo\\ Chern Institute of Mathematics, LPMC and New Cornerstone Science Laboratory, Nankai University, Tianjin 300071, P.R. China} \email{shaomingguo2018@gmail.com}

\author{Changkeun Oh}\address{ Changkeun Oh\\ Department of Mathematical Sciences and RIM, Seoul National University, Republic of Korea} \email{changkeun@snu.ac.kr}

\begin{abstract}
We establish sharp bounds for rational points near compact smooth nondegenerate curves in $\mathbb R^n$, for every $n\geq4$.
For a nondegenerate curve $\mathcal C$ in $\mathbb R^n$, let 
\[ \mathcal R_{\mathcal C}(\delta,Q)
 :=\#\left\{(\mathbf p,q)\in\mathbb Z^n\times\mathbb N:
 1\leq q\leq Q,\ \operatorname{dist}(\mathbf p/q,\mathcal C)\leq\delta/q\right\}.\]
We prove that
\[\mathcal R_{\mathcal C}(\delta,Q)
\lesssim_{\varepsilon,\mathcal C}
\delta^{n-1}Q^2+Q^\varepsilon
\left(Q+\sum_{k=1}^{n-1}
\delta^{k^2/(k+1)}Q^{(2k+1)/(k+1)}\right).\]
A key new ingredient is a wave packet method developed by Gan--Maldague--Oh.
\end{abstract}

\maketitle




\section{Introduction}
\label{sec:introduction}

Let $n \geq 2$ and $I\subset\mathbb R$ be a compact interval. Let
\[
 \boldsymbol\gamma(t)=\bigl(\gamma_1(t),\gamma_2(t),\ldots,\gamma_n(t)\bigr),
 \qquad t\in I,
\]
be a $C^\infty$ parametrization of a compact curve
$\mathcal C\subset\mathbb R^n$.  We say that $\mathcal C$ is
\emph{nondegenerate} if
\begin{equation}
 \left|\det\bigl(\boldsymbol\gamma'(t),\boldsymbol\gamma''(t),\ldots,
 \boldsymbol\gamma^{(n)}(t)\bigr)\right|>0
 \label{eq:intro-nondegeneracy}
\end{equation}
for all $t \in I$.
For $Q\geq1$ and $0<\delta\leq1$, define the counting function
\begin{equation}
 \mathcal R_{\mathcal C}(\delta,Q)
 :=\#\left\{(\mathbf p,q)\in\mathbb Z^n\times\mathbb N:
 1\leq q\leq Q,\ \operatorname{dist}(\mathbf p/q,\mathcal C)\leq\delta/q\right\}.
 \label{eq:intro-cumulative-count}
\end{equation}
We count
representations $(\mathbf p,q)$, so no coprimality condition is imposed.

\begin{theorem}
\label{thm:main}
Let $n\ge3$ and let $\mathcal C\subset\mathbb R^n$ be a compact
$C^\infty$ nondegenerate curve. For every $\varepsilon>0$, $Q\ge2$,
and $0<\delta\le1$,
\begin{equation}\label{eq:intro-refined-bound}
\mathcal R_{\mathcal C}(\delta,Q)
\lesssim_{\varepsilon,\mathcal C}
\delta^{n-1}Q^2+Q^\varepsilon
\left(Q+\sum_{k=1}^{n-1}
\delta^{k^2/(k+1)}Q^{(2k+1)/(k+1)}\right).
\end{equation}

\end{theorem}

In dimension three, Chen, Seeger, Srivastava and Technau~\cite[Corollary~1.3]{CSST2026}
established the bound
$\mathcal R_{\mathcal C}(\delta,Q)\lesssim_{\varepsilon,\mathcal C}\delta^2Q^2+Q^{4/3+\varepsilon}$,
which yields the optimal volume bound for $\delta\ge Q^{-1/3+\varepsilon}$.
Theorem~\ref{thm:main} extends the optimal volume bound to higher dimensions
and provides refined estimates for small $\delta$, including in dimension three.

The first term in \eqref{eq:intro-refined-bound}
is the main term predicted by volume. For fixed $q$, the condition
$\operatorname{dist}(\mathbf p/q,\mathcal C)\leq\delta/q$ places $\mathbf p$
in the $\delta$-neighborhood of $q\mathcal C$. Since $q\mathcal C$ has length
comparable to $q$,
\[
\operatorname{vol}\{\mathbf x\in\mathbb R^n:
\operatorname{dist}(\mathbf x,q\mathcal C)\leq\delta\}
\asymp_{\mathcal C}q\delta^{n-1}.
\]
Since the integer lattice has one point per unit volume on average,
the volume heuristic predicts about $q\delta^{n-1}$ admissible numerators
$\mathbf p$ for each $q$. Summing over $1\leq q\leq Q$ gives
\[
\sum_{1\leq q\leq Q}q\delta^{n-1}
=\delta^{n-1}\sum_{1\leq q\leq Q}q
\asymp\delta^{n-1}Q^2.
\]
This motivates the \emph{volume bound}
$\mathcal R_{\mathcal C}(\delta,Q)\lesssim_{\mathcal C}\delta^{n-1}Q^2$.

In the range of $\delta$ stated below, the matching lower bound
$\mathcal R_{\mathcal C}(\delta,Q)\gtrsim_{\mathcal C}\delta^{n-1}Q^2$
follows from Beresnevich, Vaughan, Velani and Zorin
\cite[Corollary~1.2]{BVVZ2021}. Combining this with the upper bound supplied
by Theorem~\ref{thm:main} yields the following corollary. Its range of $\delta$
is optimal  up to the factor $Q^\varepsilon$,
as shown by Proposition~\ref{cor:intro-moment-count} below.

\begin{corollary}
Let $n\ge3$ and let $\mathcal C$ be as in Theorem~\ref{thm:main}.
Set $D_n=n-1+\left\lfloor\frac{(n-1)^2}{4}\right\rfloor$.
For every $\varepsilon>0$, if $Q\ge2$ is sufficiently large depending on
$\varepsilon$ and $\mathcal C$, and
$Q^{-1/D_n+\varepsilon}\leq\delta\leq1$, then
\begin{equation}
 c\delta^{n-1}Q^2
 \leq \mathcal R_{\mathcal C}(\delta,Q) \leq C\delta^{n-1}Q^2,
 \label{eq:intro-volume-range}
\end{equation}
where $c,C>0$ depend only on $\varepsilon$ and $\mathcal C$.
\end{corollary}

To explain the additional terms in \eqref{eq:intro-refined-bound}, consider the moment curve
\[
 \Gamma_n=\{(t,t^2,\ldots,t^n):0\leq t\leq1\}.
\]
For this curve, the term $Q$ corresponds to $k=0$, namely the pairs $(\mathbf 0,q)$ with
$q\asymp Q$. For $1\leq k\leq n-1$, we restrict to denominators $Q/2<q\leq Q$ and construct integer points
$\mathbf p=(p_1,\ldots,p_k,0,\ldots,0)$ within distance $\delta$ of $q\Gamma_n$.
Closeness in the $(k+1)$st coordinate requires $qt^{k+1}\leq\delta$, which implies
\[
0\leq t\leq\left(\frac{\delta}{q}\right)^{1/(k+1)}
\asymp\left(\frac{\delta}{Q}\right)^{1/(k+1)}.
\]
This is therefore a rational point problem for the $k$-dimensional curve
$(t,t^2,\ldots,t^k)$ on this short interval near the origin.
The volume of the corresponding $\delta$-neighborhood predicts the count
\[
\underbrace{Q}_{\substack{\text{number of}\\\text{denominators}}}
\underbrace{Q\left(\frac{\delta}{Q}\right)^{1/(k+1)}}_{
\substack{\text{length of the}\\\text{dilated curve segment}}}
\underbrace{\delta^{k-1}}_{\substack{\text{cross-sectional}\\\text{volume}}}
=\delta^{k^2/(k+1)}Q^{(2k+1)/(k+1)}.
\]
The following proposition establishes these predicted lower bounds for the moment curve $\Gamma_n$, showing that Theorem~\ref{thm:main} is sharp up to the factor $Q^\varepsilon$.

\begin{proposition}
\label{cor:intro-moment-count}
Let $n\ge3$ and let $\Gamma_n$ be the moment curve defined above.
For every $Q\ge2$ and $0<\delta\le1$,
\[
\delta^{n-1}Q^2+
\sum_{k=0}^{n-1}\delta^{k^2/(k+1)}Q^{(2k+1)/(k+1)}
\lesssim_n \mathcal R_{\Gamma_n}(\delta,Q).
\]
\end{proposition}

The planar case is not covered by Theorem~\ref{thm:main}.
For a compact $C^\infty$ curve $\mathcal C\subset\mathbb R^2$ with nonvanishing curvature, the theorem of Vaughan and Velani~\cite[Theorem~3 and the following remark]{VV2006} implies
\begin{equation}
\mathcal R_{\mathcal C}(\delta,Q)
\lesssim_{\varepsilon,\mathcal C}\delta Q^2+Q^{1+\varepsilon},
\qquad Q\ge2,\quad 0<\delta\le1,
\label{eq:intro-planar-bound}
\end{equation}
for every $\varepsilon>0$.
Their proof combines Fourier analysis with Huxley's estimate obtained by the determinant method~\cite{Huxley1994}.
We do not know how to prove \eqref{eq:intro-planar-bound} using Fourier analysis alone.

Chen, Seeger, Srivastava and Technau~\cite{CSST2026} recently established essentially sharp bounds for rational points near nondegenerate curves in $\mathbb R^3$.
We found it surprising that their proof of these essentially sharp upper bounds uses the planar counting estimate \eqref{eq:intro-planar-bound} as its only arithmetic ingredient; see \cite[Section~3.7]{CSST2026}.
Once this planar estimate is taken as an input, the remaining argument proceeds entirely by Fourier analysis, including Poisson summation, stationary phase, and estimates for the Fourier transform of the dual surface.

A difficulty in extending the argument of Chen, Seeger, Srivastava and Technau is to obtain pointwise bounds for the oscillatory integrals arising in the proof. The size of the oscillatory integral involves $n$ variables, and the correct form of a bound in terms of all these variables is difficult to conjecture. Moreover, even if one conjectures such a bound, it is unclear whether it can be proved by extending the argument based on van der Corput's lemma to higher dimensions. To overcome this difficulty, we use the \textit{wave packet method} in Gan, Maldague and Oh \cite{GMO2025}. Wave packet method is a way to characterize the Fourier decay of nondegenerate curves using $\bla$-planks. A key geometric object---the
$\boldsymbol{\lambda}$-plank---was first defined by
Gan, Maldague and Oh~\cite{GMO2025}. It originates from the high--low method introduced by
Guth, Solomon and Wang~\cite{GSW2019}. The multiscale anisotropic geometry
underlying the $\boldsymbol{\lambda}$-plank appeared in dimension three in
the work of
Gan, Guo, Guth, Harris, Maldague and Wang~\cite{GGGHMW2026} and was subsequently
developed in higher dimensions by Gan, Guo and Wang~\cite{GGW2024}. With the help of $\bla$-planks, we are able to define the anisotropic counting problems which plays a key role in the study of the original counting problems.

At a broad level, our proof of Theorem~\ref{thm:main} uses Poisson summation and stationary phase,
as in the approach of Chen, Seeger, Srivastava and Technau. We also use an observation by
Chen, Seeger, Srivastava and Technau~\cite{CSST2026} when choosing the threshold of cutoff radius; in particular, it guides
the choice of the cutoff $\tau_\circ$ in \eqref{tc}, which separates the low- and
high-frequency contributions in Section~\ref{sec:algebraic-computation}. However, we offer a new perspective on their three-dimensional argument by
reinterpreting it in the language of wave packets.  By adapting the wave packet
method of Gan, Maldague and Oh~\cite{GMO2025} to this setting, we believe that
our work provides a new tool for rational point counting.
Our approach yields essentially sharp bounds for compact smooth
nondegenerate curves in $\mathbb R^n$ for every $n\ge4$, throughout the full
range $0<\delta\le1$.


\subsection{General manifolds}

Counting rational points of bounded height is a classical problem in Diophantine geometry; see, for example, Schanuel~\cite{Schanuel1979}.
For rational points lying near smooth manifolds, to the best of our knowledge, the first general estimates in the planar case were obtained by Huxley~\cite{Huxley1994}.
A general formulation of the counting problem for smooth manifolds appears in Beresnevich's work~\cite{Beresnevich2012}.

Let $1\le d<n$, put $m=n-d$, and let $f_1,\ldots,f_m$ be smooth functions on  $[0,1]^d$. Consider the manifold \[ \mathcal M:=\left\{\bigl(x_1,\ldots,x_d,f_1(\mathbf x),\ldots,f_m(\mathbf x)\bigr): \mathbf x=(x_1,\ldots,x_d)\in[0,1]^d\right\}\subset\mathbb R^n. \] Define $\mathcal R_{\mathcal M}(\delta,Q)$ by \eqref{eq:intro-cumulative-count}, with $\mathcal C$ replaced by $\mathcal M$. Beresnevich~\cite[Problem~1.1]{Beresnevich2012} asks for counting estimates near a generic smooth manifold.  If $f_1,\ldots,f_m$ are real analytic on $[0,1]^d$ and $\mathcal M$ is not contained in any proper affine subspace, Beresnevich~\cite[Corollary~1.5]{Beresnevich2012} proves a lower bound which gives \[ \mathcal R_{\mathcal M}(\delta,Q)\gtrsim_{\mathcal M}\delta^mQ^{d+1}, \qquad C_{\mathcal M}Q^{-1/m}\le\delta\le1, \] for sufficiently large $Q$, where $C_{\mathcal M}>0$ depends on $\mathcal M$. He then posed a question of estimating $R_{\mathcal M}(\delta,Q)$. \par Substantial progress on this problem has been made under various geometric assumptions. For planar curves with nonvanishing curvature, Vaughan and Velani~\cite{VV2006} proved the upper bound \eqref{eq:intro-planar-bound}. Their work also provided a Fourier-analytic framework, based on Poisson summation and duality, that has played a central role in subsequent advances in rational point counting near manifolds. Huang~\cite{Huang2020} established the corresponding bound for hypersurfaces with nonvanishing Gaussian curvature. These results give the volume bound for $\delta\ge Q^{-1+\varepsilon}$ for every $\varepsilon>0$, a range that is optimal up to the $Q^\varepsilon$ loss. Allowing curvature to vanish, Chen~\cite{Chen2025} obtained sharp counting estimates for planar curves of finite type, while Srivastava and Technau~\cite{SrivastavaTechnau2023} treated certain homogeneous hypersurfaces. In higher codimension, Schindler and Yamagishi~\cite{SchindlerYamagishi2022} obtained volume bounds under a strong curvature condition. Subsequent work of Srivastava~\cite{Srivastava2025} and Hickman, Srivastava and Wright~\cite{HSW2025} enlarged the range of $\delta$ in which these bounds hold. For sufficiently smooth nondegenerate manifolds, Schindler, Srivastava and Technau~\cite[Theorem~1.5]{SST2023} combined quantitative nondivergence estimates of Bernik, Kleinbock and Margulis~\cite[Theorem~1.4]{BKM2001} with Fourier analysis to prove the volume bound $\mathcal R_{\mathcal M}(\delta,Q)\lesssim\delta^mQ^{d+1}$ when $\delta$ is at least an explicit power of $Q^{-1}$ determined by the dimension and the order of nondegeneracy. For the higher-codimension and nondegenerate classes considered in these works, determining the optimal range of $\delta$ remains open in general. \par For compact smooth nondegenerate curves in $\mathbb R^3$, Chen, Seeger, Srivastava and Technau~\cite[Corollary~1.4]{CSST2026} proved the volume bound in the range $\delta\ge Q^{-1/3+\varepsilon}$, which is sharp up to the $Q^\varepsilon$ loss. 
For curves in higher dimensions, Hickman and Srivastava~\cite{HickmanSrivastava2025} obtained partial results. More precisely, they obtained the volume bound in the range $\delta\ge Q^{-1/\alpha_n+\varepsilon}$, where\[
\qquad
\alpha_n=
\begin{cases}\dfrac{n(n+4)}{4},&n\ \text{even},\\[6pt]\dfrac{n(n+4)-1}{4},&n\ \text{odd}.
\end{cases}
\]
To the best of our knowledge, this gives the best known range for the volume upper bound for rational points near smooth nondegenerate curves in $\mathbb R^n$ for $n\ge4$. Theorem~\ref{thm:main} provides essentially sharp counting estimates throughout the full range $0<\delta\le1$ for compact smooth nondegenerate curves in $\mathbb R^n$, $n\ge4$. In particular, it yields the volume bound in the sharp range $\delta\ge Q^{-1/D_n+\varepsilon}$, as stated in \eqref{eq:intro-volume-range}. Our wave packet method does not rely heavily on the specific structure of the moment curve. We are therefore hopeful that it can be adapted to study rational points near more general manifolds.

\subsection{Acknowledgements}
\label{sec:acknowledgements}

We would like to thank Mingfeng Chen, Rajula Srivastava, Andreas Seeger, Niclas Technau, and Wooyeon Kim for valuable discussions. Shaoming Guo’s research is partly supported by the Nankai Zhide Foundation, NSFC Grant No. 12426204, the New Cornerstone Science Foundation, and the Fundamental Research Funds for the Central Universities No. 100-63263092. 
Changkeun Oh was supported by the POSCO Science Fellowship of POSCO TJ Park Foundation, and the National Research Foundation of Korea (NRF) grant funded by the Korea government (MSIT) RS-2024-00341891.

\subsection{Use of AI}
\label{sec:ai-use}

After reading the work of Chen, Seeger, Srivastava and Technau \cite{CSST2026},
we observed that their three-dimensional result could be interpreted in the language of wave packets developed in \cite{GMO2025}.
In developing the proof in higher dimensions, we used GPT-6 for numerology checks; in particular, it assisted in writing Subsection \ref{conclude}

\subsection{Note added}
\label{sec:independent-work}After completing this manuscript, we learned that Chen, Srivastava and Technau~\cite{ChenSrivastavaTechnau2026Counterexamples} had independently obtained the lower bounds in Proposition~\ref{cor:intro-moment-count}.

\section{Proof sketch and computation of numerology}
\label{sec:algebraic-computation}

In this section, we explain the
geometric intuition behind the proof.
The arguments in this section are not rigorous. For rigorous proof, we refer to later sections.

Recall that $\ga(t)=(\ga_1(t),\dots,\ga_n(t))$ and $\cC=\{\ga(t): t\in I\}$ is a nondegenerate curve in $\R^n$. The discussion below is easier to follow with the moment curve
$\gamma(t)=(t,t^2,\ldots,t^n)$ in mind. Define the counting function
\[ N_n(\de,Q):=\#\{(\mathbf p,q)\in \ZZ^n\times \ZN:Q\le q\le 2Q, \dist(\mathbf p/q,\cC) \le \de/q \}. \]
The goal of the paper is to bound $N_n(\de,Q)$ by the right hand side of \eqref{eq:intro-refined-bound}.

Define the cone associated with $\cC$ to be
\[ \Ga:=\{s(1,\ga(t)): s\in[1,2],t\in I\}. \]
Morally speaking, we have
\begin{equation}\label{Nn}
    N_n(\de,Q)\sim \#\{(\mathbf p,q)\in \ZZ^{n+1}: \dist((\frac qQ,\frac{\mathbf p}{Q}),\Ga)\le \de/Q  \}. 
\end{equation} 
Indeed, for $Q\le q\le 2Q$,
\[\dist(\mathbf p/q,\cC)\le \de/q\Leftrightarrow \dist(\frac{\mathbf p}{Q},\frac qQ\cC)\le \de/Q  \Leftrightarrow \dist((\frac qQ,\frac{\mathbf p}{Q}),\Ga)\le \de/Q.\]
So morally speaking, \textit{counting rationals} near $\ga$ is equivalent to \textit{counting lattice points} near $\Ga$. 
There are two parameters $\de,Q$ in the counting function $N_n(\de,Q)$. We call $\de$ the \textit{width} and $Q$ the \textit{height}; $R=\de^{-1}Q$ is called the \textit{scale} of the counting problem.

\subsection{Write the counting function as an integral}
Let $R:=\de^{-1}Q$.
Define
\[ X:=N_{R^{-1}}(B(0,1)\cap Q^{-1}\ZZ^{n+1}). \]
This is the union of $R^{-1}$-balls centered at lattice points $Q^{-1}\ZZ^{n+1}$ inside the unit ball.
Define $\mu_\Ga$ to be the surface measure on $\Ga$. Then \eqref{Nn} can be interpreted as the following integral:
\[ N_n(\de,Q)\sim R^{2}\int \Id_X \mu_\Ga. \]
This is because if a point  $v\in Q^{-1}\ZZ^{n+1}$ lies $R^{-1}$-close to $\Ga$, then it contributes $|\Ga\cap B(v,R^{-1})|\sim R^{-2}$ to the integral. That is why we multiply the factor $R^2$.

\bigskip

\subsection{Transfer to a dual counting problem using Plancherel} By Plancherel, one has
\begin{equation}\label{planch}
    N_n(\de,Q)\sim R^2 \int \overline{\wh{\Id_X}(\xi)}\,\wh{\mu_\Ga}(\xi)\,d\xi. 
\end{equation} 
Let us compute $\wh \Id_X$. Recall that $X$ is the union of $R^{-1}$-balls centered at points of the lattice $Q^{-1}\ZZ^{n+1}$ inside $B(0,1)$. Define the dual set
\begin{equation}\label{eq:sketch-dual-lattice-set} X':=N_1( B(0,R)\cap Q\ZZ^{n+1} ). \end{equation}
We claim that
\begin{equation}\label{whFX}
    \wh \Id_X\approx |X| \Id_{X'}\sim \de^{n+1}\Id_{X'}. 
\end{equation} 
To be rigorous, one needs to use  smooth approximations of $X,X'$, instead of the indicator functions. But we  omit the technical detail here.
Since $X$ is near a union of $Q^{-1}$-spaced lattice points, this suggests that the support of $\wh \Id_{X}$ is near a union of $Q$-spaced lattice points.
Since $X$ is supported in $B(0,1)$, $\wh \Id_X$ is locally constant at scale $1$. Since $X$ is a union of $R^{-1}$-balls, $\wh \Id_X$ decays rapidly outside $B(0,R)$. At the points where there is constructive interference, one has $\wh \Id_X \sim |X|\sim \de^{n+1}$. This yields \eqref{whFX}. The rigorous proof can be given by Poisson summation formula.

Hence, the goal is to understand the behavior of $\wh \mu_\Ga$. 
\begin{remark}
    {\rm
    The idea of transferring the original counting problem to a dual counting problem was first used by Huang \cite{Huang2020} to count rationals near hypersurfaces. 
    }
\end{remark}

\bigskip

\subsection{Gan--Maldague--Oh's characterization of \texorpdfstring{$\wh\mu_\Ga$}{}}
A full characterization of the decay rate of $\wh \mu_\Ga$ was given in \cite{GMO2025}. We briefly discuss the results in this subsection. For a complete discussion, we refer to Section 3 and Theorem~3.22 in \cite{GMO2025}. 

We first define the Frenet frame and the dual frame associated with $\Ga$. Let $\bv_1(t):=\frac{1}{\sqrt{|\ga(t)|^2+1}}(1,\ga(t))$. Since $\ga$ is nondegenerate, $\bv_1(t),\bv_1'(t),\dots,\bv_1^{(n)}(t)$ are linearly independent. Applying the Gram--Schmidt process to them gives an orthonormal frame
\[ \bv_1(t),\dots,\bv_{n+1}(t). \]
Let $\bde_i(t):=\bv_{n+2-i}(t)$ for $1\le i\le n+1$.
We call 
\[ \bde_1(t),\dots,\bde_{n+1}(t) \]
the \textit{dual Frenet frame} associated to the cone $\Ga$.

Next, we introduce $\bla$-planks. Here, $\bla=(\la_1,\dots,\la_{n-1})$ with each $\la_i\in (0,1]$. Our notation may differ slightly from that used in \cite{GMO2025}. In \cite{GMO2025}, these are called $\boldsymbol{\de}$-planks. Since $\de$ denotes another quantity, namely the width, in this paper, we use $\bla$ instead.

\begin{definition}[Admissible]
    Fix $r\ge 1$. Let $\bla=(\la_1,\dots,\la_{n-1})$ with $\la_i\in(0,1]$ and let $A\subset \{1,\dots,n-1\}$. We say $(\bla,A)$ is admissible at scale $r$ in dimension $n$ if $r^{-1}=\la_1^n\la_2^{n-1}\dots\la_{n-1}^2$, and $i\in A$ if and only if $\la_i<1$.
\end{definition}

\begin{definition}[$\rho$-sequence]
    Given $\bla=(\la_1,\dots,\la_{n-1})$, set $\rho_0=1$ and, for $1\le k\le n-1$, define
    \begin{equation}\label{defrho}
        \rho_k=\la_1^k\la_2^{k-1}\dots\la_k.
    \end{equation}
    We also set $\rho_n=\la_1^n\dots\la_{n-1}^2=r^{-1}$.
\end{definition}

\begin{remark}
    {\rm
    It is easy to see that
\begin{equation}\label{rho n-1}
    \rho_n\rho_{n-2}=\rho_{n-1}^2 \;\;\;\; \mathrm{and}  \qquad \rho_{i}\rho_{i+2}\le \rho_{i+1}^2 \quad (0\le i\le n-3). 
\end{equation} 
    }
\end{remark}

\begin{definition}[$(\bla,A)$-plank]
    Let $(\bla,A)$ be admissible at scale $r\ge 1$. Define the $(\bla,A)$-plank as follows
    \begin{equation}
        P[\bla,A]:=\Biggl\{(x_1,\dots,x_{n+1})\in\mathbb R^{n+1}:
\begin{array}{l}
|x_i|\lesssim r\rho_{i-1}
 \quad (1\le i\le n-1,\ i\notin A),\\[1mm]
|x_i|\sim r\rho_{i-1}
 \quad (i\in A),\\[1mm]
|x_n|\lesssim r\rho_{n-1},\quad |x_{n+1}|\lesssim r\rho_n=1
\end{array}\Biggr\}.
    \end{equation} 
\end{definition}
We also define the $(\bla,A)$-plank associated with the dual Frenet frame. For $t\in I$, define
\[ P[\bla,A](t):=\left\{ \sum_{i=1}^{n+1}x_i\bde_i(t): (x_1,\dots,x_{n+1})\in P[\bla,A] \right\}. \]

\begin{remark}
    {\rm
    For $i\in A$, the condition $|x_i|\sim r\rho_{i-1}$ reads as either $x_i\in [-Cr\rho_{i-1},-cr\rho_{i-1}]$ or $x_i\in [cr\rho_{i-1},Cr\rho_{i-1}]$ for some constants $0<c<C$.
    Thus, $P[\bla,A](t)$ consists of $2^{|A|}$ boxes, each with dimensions $r\times r\rho_1\times \dots\times r\rho_{n-1}\times 1$. For our purpose, it is harmless to consider a single piece of $P[\bla,A]$ which satisfies $x_i>0$ for $i\in A$. The other pieces can be handled similarly.
    When only the dimensions of $P[\bla,A](t)$ are relevant, we refer to it as a $\bla$-plank.
    }
\end{remark}

Next, we define the \textit{aperture} of a $\bla$-plank:
\[ h:=\prod_{i=1}^{n-1}\la_i. \]
(In the later sections, we will also use $\laprod$ to denote this quantity.)
This is the scale at which one can distinguish two different $\bla$-planks. The precise statement is given below. It is from \cite[Lemma 3.10, Lemma 3.12]{GMO2025}.
\begin{lemma}\label{lemcompare}
Let $h=\prod_{i=1}^{n-1}\la_i$.
    If $|s-s'|\lesssim h$, then $P[\bla,A](s)$ and $P[\bla,A](s')$ are comparable. If $|s-s'|\ge C h$, then $P[\bla,A](s)$ and $P[\bla,A](s')$ are disjoint.
\end{lemma}

Finally, we state the result concerning the decay rate estimate of $\wh \mu_\Ga$. See \cite[Theorem 3.22]{GMO2025} and the discussion in \cite[Section 3.4]{GMO2025}.

\begin{lemma}
Fix $r\ge 1$. Let $\cI(h)$ be a $h$-net of $I$.
For $x\in \R^{n+1}$ with $|x|\sim r$, we have 
    \begin{equation}\label{gmo0}
        |\wh \mu_\Ga(x)|\lesssim \sum_{\begin{subarray}{c}
          (\bla,A) \\
          \text{~admissible at~}r
\end{subarray}}h\sum_{t\in \cI(h)}\mathbf{1}_{P[\bla,A](t)}(x).
    \end{equation}
Here, the first sum runs over all pairs $(\bla,A)$ admissible at scale $r$, with each $\la_i$ dyadic. 
\end{lemma}

Summing over dyadic $r$ in $[2,R]$, we obtain
\begin{equation}\label{gmo}
\begin{split}
        |\wh \mu_\Ga(x)|\Id_{B(0,R)}&\lesssim\Id_{B(0,1)}(x)\\
        +\sum_{\begin{subarray}{c}
          2\le r\le R \\
          \text{dyadic}
\end{subarray}} &\Id_{B(0,r)\setminus B(0,r/2)}(x)\sum_{\begin{subarray}{c}
          (\bla,A) \\
          \text{~admissible at~}r
\end{subarray}}h\sum_{t\in \cI(h)}\mathbf{1}_{P[\bla,A](t)}(x).
        \end{split}
    \end{equation}

Recall $R=\de^{-1}Q$.
Plug \eqref{whFX} and \eqref{gmo} into \eqref{planch}. Note that there are only $O_n(1)$ unit balls in $X'\cap B(0,Q)$, so we only need to consider $r=1$ and $r\ge Q$. Hence, we obtain
\begin{equation}\label{anisotropic1}
    N_n(\de,Q)\lesssim \de^{n-1}Q^2+ \de^{n-1}Q^2  \sum_{\begin{subarray}{c}
          Q\le r\le R \\
          \text{dyadic}
\end{subarray}}\sum_{\begin{subarray}{c}
          (\bla,A) \\
          \text{~admissible at~}r
\end{subarray}}h\int \Id_{X'} \sum_{t\in \cI(h)}\mathbf{1}_{P[\bla,A](t)}.
\end{equation}
By Lemma~\ref{lemcompare}, the planks have bounded overlap for each fixed $(\bla,A)$, so the sum of their indicators is comparable to the indicator of their union. The right hand side is the dual counting problem of the original one. Here, for each $(\bla,A)$, we need to count the number of lattice points of $X'$ in 
\[ \bigcup_{t\in \cI(h)} P[\bla,A](t), \] which is an anisotropic neighborhood of a certain developable manifold. 

\bigskip

\subsection{Anisotropic counting problems}
Let us digest the integral on the right hand side of \eqref{anisotropic1}. For $(\bla,A)$ admissible at scale $r$, set $p[\bla,A](t):=r^{-1}P[\bla,A](t)$, the $r^{-1}$-dilation of $P[\bla,A](t)$. Then $p[\bla,A](t)$ is of dimensions
\[ 1\times \rho_1\times \dots\times \rho_{n-1}\times r^{-1}, \]
where $\rho_i$ is given by \eqref{defrho}. By \eqref{eq:sketch-dual-lattice-set}, we have \[r^{-1} X'=N_{r^{-1}}(B(0,R/r)\cap Q/r\ZZ^{n+1}).  \]
Then morally speaking,
\[\begin{aligned} \int \Id_{X'} \sum_{t\in \cI(h)}\mathbf{1}_{P[\bla,A](t)} &\sim \#\left\{ x\in Q\ZZ^{n+1}:x\in \bigcup_{t\in \cI(h)} P[\bla,A](t) \right\} \\ &= \#\left\{ x\in \frac{Q}{r}\ZZ^{n+1}:x\in \bigcup_{t\in \cI(h)} p[\bla,A](t) \right\}. \end{aligned}\]
This is the counting problem with \textit{height} parameter $Q_*=r/Q$ and \textit{width} parameter $\de_*=r^{-1}(r/Q)=Q^{-1}$. Here, \textit{width} means the ratio of the shortest length of the plank and the spacing of the lattice.

\begin{definition}
Let $1\le Q_*$ and $0<\de_*\le Q_*$.
    Suppose $(\bla,A)$ is admissible at scale $\de_*^{-1}Q_*$. Define the anisotropic counting function
\begin{equation}\label{defanisotropic}
       N_n(\de_*,Q_*;\bla):=\#\left\{ x\in Q_*^{-1}\ZZ^{n+1}:x\in \bigcup_{t\in \cI(h)} p[\bla,A](t)\right\}.
    \end{equation}
\end{definition}

\begin{remark} 

    {\rm We give two remarks.

    \begin{itemize}
        \item When $n=2$ and $\de_*\le 1\le Q_*$, let $\la_1=(\de_* Q_*^{-1})^{1/2}$. Then the anisotropic counting function $N_2(\de_*,Q_*;\la_1)$ agrees with the isotropic counting function $N_2(\de_*,Q_*)$.

        \item It may happen that $\de_*\ge 1$. In this case, the counting problem is trivial, since the shortest edge of the $\bla$-plank is larger than $Q_*^{-1}$, the spacing of the lattice $Q_*^{-1}\ZZ^{n+1}$. We have the trivial volume bound:
    \begin{equation}\label{volumebound}
        N_n(\de_*,Q_*;\bla)\lesssim Q_*^{n+1} \left|\bigcup_{t\in \cI(h)} p[\bla,A](t)\right|\lesssim Q_*^{n+1}h^{-1}\prod_{i=1}^n\rho_i. 
    \end{equation}
    \end{itemize}}
\end{remark}

Using the anisotropic counting function defined in \eqref{defanisotropic}, we can bound the dual counting term in \eqref{anisotropic1} by $N_n(Q^{-1},r/Q;\bla)$. This gives the following proposition.
\begin{proposition}\label{input1}Let $0<\de\le 1\le Q$. Then,
    \begin{equation}
        N_n(\de,Q)\lesssim \de^{n-1}Q^2+ \de^{n-1}Q^2  \sum_{\begin{subarray}{c}
          Q\le r\le \de^{-1}Q \\
          \text{dyadic}
\end{subarray}}\sum_{\begin{subarray}{c}
          (\bla,A) \\
          \text{~admissible at~}r
\end{subarray}}h\cdot N_n(Q^{-1},\frac rQ;\bla).
    \end{equation}
\end{proposition}

\bigskip

\subsection{Thickening and projection arguments}
$N_n(\de,Q)$ and $N_n(\de,Q;\bla)$ concern $n$-dimensional counting problems. We call them $n$-dimensional problems, since they are related to the nondegenerate curve $\cC$ in $\R^n$. In practice, we also need to consider counting problems with intermediate dimensions.

Let $3\le m\le n$. We will study the counting function $N_m(\de_*,Q_*;\la_1,\dots,\la_{m-1})$, where the tuple $(\la_1,\dots,\la_{m-1})$ is admissible at scale $\de_*^{-1}Q_*$ in dimension $m$. It is related to the the nondegenerate curve in $\R^m$ determined by $\cC$, for which we do not discuss in detail. In what follows, we suppress $A$ in the notation and focus on the scale parameters $\la_i$. In this subsection, we outline the proof of Proposition~\ref{input2}. We use two methods, namely \textit{thickening} and \textit{projection}, to bound $N_m$ in terms of $N_{m-1}$ and $N_2$.

Write $\bla=(\la_1,\dots,\la_{m-1})$. Set 
\[ r_*:=\de_*^{-1}Q_*=\Big(\prod_{i=1}^{m-1}\la_i^{m+1-i}\Big)^{-1},\qquad h_m:=\prod_{i=1}^{m-1}\la_i. \]  Here, the second equality follows from the admissibility of $\bla$. Let $(\rho_i)_{i=0}^{m}$ be the $\rho$-sequence associated with $(\la_i)_{i=1}^{m-1}$ as in \eqref{defrho}.
By definition, $N_m(\de_*,Q_*;\bla)$ counts the number of points of the lattice $Q_*^{-1}\ZZ^{m+1}$ in a union of $1\times \rho_1\times \rho_2\times \dots\times \rho_{m-1}\times r_*^{-1}$-boxes denoted by $\{p[\bla](t)\}_{t\in \cI(h_m)}$. To estimate $N_m(\de_*,Q_*;\bla)$, we thicken each box $p[\bla](t)$ to $p[\wt\bla](t)$, where $\wt\bla$ is obtained by increasing only $\la_{m-1}$. We give  details below.  

Note that the last two side lengths of $p[\bla](t)$, which are the two shortest, are $\la_1^{m-1}\la_2^{m-2}\dots\la_{m-1}$ and $\la_1^m\dots\la_{m-1}^2$, respectively, while the other side lengths are independent of $\la_{m-1}$. We will enlarge the parameter $\la_{m-1}$ to increase the two shortest side lengths of $p[\bla](t)$ while keeping the other side lengths fixed. Suppose we increase the shortest side length from $r_*^{-1}$ to $\tau$, and denote the enlarged value of $\la_{m-1}$ by $\wt\la_{m-1}$. The original and enlarged shortest side lengths satisfy
\[
r_*^{-1}=\left(\prod_{i=1}^{m-2}\la_i^{m+1-i}\right)\la_{m-1}^{2},
\qquad
\tau=\left(\prod_{i=1}^{m-2}\la_i^{m+1-i}\right)\wt\la_{m-1}^{2}.
\]
Admissibility requires $\wt\la_{m-1}\le1$, which is equivalent to $\tau\le\prod_{i=1}^{m-2}\la_i^{m+1-i}$.  We also impose $\tau\le Q_*^{-1}$, since when the thickness exceeds the lattice spacing, the trivial volume bound \eqref{volumebound} applies. Therefore, we require $\tau$ to satisfy
\begin{equation}
    r_*^{-1}\le \tau\le \min\{ \prod_{i=1}^{m-2}\la_i^{m+1-i},Q_*^{-1}\}=:\tc.\label{tc}
\end{equation}
We will establish the following estimate.

\begin{proposition}\label{input2} Fix $m\ge 3$. Let $0<\de_*\le 1\le Q_*$. 
Let $\bla=(\la_1,\dots,\la_{m-1})$ be admissible at scale $\de_*^{-1}Q_*$. Let $\tc=\min\{Q_*^{-1}, \prod_{i=1}^{m-2}\la_i^{m+1-i}\}$. Then for any $\tau\in [\de_*Q_*^{-1},\tc]$, we have
    \begin{equation}
\begin{split}
N_m&(\de_*,Q_*;\bla)\lesssim \frac{\tau}{\tc}N_{m-1}(\prod_{i=1}^{m-2}\la_i^{m-i}Q_*,Q_*;\la_1,\dots,\la_{m-2})\\
&+(\prod_{i=1}^{m-2}\rho_i)\cdot(\rho_{m-2}\tau)^{1/2}\tau Q_*^{m+1}N_2(\rho_{m-2}^{-1}Q_*^{-1},\tau^{-1}Q_*^{-1})\prod_{i=1}^{m-3}(1+\rho_{i}^{-1}Q_*^{-1}).
\end{split}
    \end{equation}
\end{proposition}

We now outline the proof. Fix $\tau\in[r_*^{-1},\tc]$. Keeping $\la_1,\dots,\la_{m-2}$ fixed, choose $\wt\la_{m-1}:=\la_{m-1}(r_*\tau)^{1/2}$ and set
\[ \wt\bla=(\la_1,\dots,\la_{m-2},\wt\la_{m-1}). \]
The associated $\rho$-sequence is
\[ (1,\rho_1,\dots,\rho_{m-2},\wt\rho_{m-1},\tau), \]
where
\[
\wt\rho_{m-1}:=\left(\prod_{i=1}^{m-2}\la_i^{m-i}\right)\wt\la_{m-1}
=(\rho_{m-2}\tau)^{1/2}.
\]
Applying \eqref{rho n-1} to the enlarged tuple gives the last equality. Since $\tau\le\tc$, we have
\begin{equation}\label{upperboundla}
    \wt\la_{m-1}\le \la_{m-1}(r_*\tc)^{1/2}\le \la_{m-1}(r_* \prod_{i=1}^{m-2}\la_i^{m+1-i})^{1/2}=1.
\end{equation}
Thus, the enlarged tuple $\wt\bla$ is admissible at scale $\tau^{-1}$. Let 
\[\wt h_m:=\la_1\dots\la_{m-2}\wt\la_{m-1},\qquad \wt \de_*:=\tau Q_*. \] 
Let $p[\wt \bla](t)$ be the $\wt\bla$-planks. Then they are of dimensions
\[ 1\times \rho_1\times \dots\times \rho_{m-2}\times \wt\rho_{m-1}\times \tau. \]
The \textit{aperture} of the $\wt\bla$-plank is $\wt h_m$. The new \textit{width} is $\wt\de_*$. \medskip

Since $p[\bla](t)\subset p[\wt\bla](t)$, and by Lemma \ref{lemcompare} that perturbing $t$ within $\wt h_m$ keeps $p[\wt\bla](t)$ essentially unchanged, we have
\begin{equation}\label{unionplankequal}
\bigcup_{t'\in\cI(h_m):|t'-t|\le \wt h_m }p[\bla](t')\subset p[\wt\bla](t). 
\end{equation} 
Therefore,
\begin{equation*}
\bigcup_{t\in\cI(h_m)}p[\bla](t)\subset \bigcup_{t\in\cI(\wt h_m)}p[\wt\bla](t). 
\end{equation*} 
Therefore we obtain:
    \begin{equation}   N_m(\de_*,Q_*;\bla)\lesssim N_m(\wt\de_*,Q_*;\wt\bla).
    \end{equation}
This completes the initial thickening step. \medskip

We now estimate $N_m(\wt\de_*,Q_*;\wt\bla)$. Let 
\begin{equation}\label{defY}
    Y:=N_\tau (B(0,1)\cap Q_*^{-1}\ZZ^{m+1}). 
\end{equation} 
By definition \eqref{defanisotropic}, we can rewrite the counting function as follows.
\[ N_m(\wt\de_*,Q_*;\wt\bla)\sim   \tau^{-(m+1)} \int \Id_Y \Id_{\bigcup_{t\in\cI(\wt h_m)}p[\wt\bla](t)}.  \]
Since $\bigcup_{t\in\cI(\wt h_m)}p[\wt\bla](t)$ is a union of $\tau$-balls, the Fourier transform of $\Id_{\bigcup_{t\in\cI(\wt h_m)}p[\wt\bla](t)}$ is essentially supported in $B(0,\tau^{-1})$. (One may replace the indicator function by smooth bump function to make this argument rigorous.) Recall $\tc$ from \eqref{tc}. Let $\eta_{\le\tc^{-1}}$ and
$\eta_{r'}$ be smooth cutoffs adapted to $B(0,\tc^{-1})$ and
$\{|\xi|\sim r'\}$, respectively, forming a dyadic partition
of unity:
\[
1=\eta_{\le\tc^{-1}}(\xi)
+\sum_{\substack{\tc^{-1}\le r'\le\tau^{-1}\\ r'\text{ dyadic}}}
\eta_{r'}(\xi),
\qquad |\xi|\le\tau^{-1}.
\]
Applying this decomposition and Fourier inversion, we obtain
\[ \wh\Id_{\bigcup_{t\in\cI(\wt h_m)}p[\wt\bla](t)}=\wh\Id_{\bigcup_{t\in\cI(\wt h_m)}p[\wt\bla](t)}\Big( \eta_{\le\tc^{-1}}+\sum_{\tc^{-1}\le r'\le \tau^{-1}}\eta_{r'}\Big). \]
Therefore, we have
\begin{equation}
N_m(\wt\de_*,Q_*;\wt\bla)\sim \tau^{-(m+1)}  \int \Id_Y \bigg(\Id_{\bigcup_{t\in\cI(\wt h_m)}p[\wt\bla](t)}*\Big( \eta_{\le\tc^{-1}}^\vee+\sum_{\tc^{-1}\le r'\le \tau^{-1}}\eta_{r'}^\vee \Big)\bigg).
\end{equation}
On the right hand side, we call the first term \textit{low part}, and the second term \textit{high part}.

\medskip

\subsubsection{\textbf{\textup{Low part}}} We first deal with the \textit{low part}. Since we have
\[ |\eta_{\le\tc^{-1}}^\vee|\lesssim \frac{1}{|B(0,\tc)|}\Id_{B(0,\tc)} \]
by ignoring the tails,
we see that the low part is bounded by
\[ \tau^{-(m+1)} \int \Id_Y \bigg(\Id_{\bigcup_{t\in\cI(\wt h_m)}p[\wt\bla](t)}* \Big(  \frac{1}{|B(0,\tc)|}\Id_{B(0,\tc)} \Big)\bigg). \]
Convolution with $\Id_{B(0,\tc)}$ blurs the set $\bigcup_{t\in\cI(\wt h_m)}p[\wt\bla](t)$ to its $\tc$-neighborhood, so we have
\begin{equation}\label{convolution}
\Id_{\bigcup_{t\in\cI(\wt h_m)}p[\wt\bla](t)}* \Big(  \frac{1}{|B(0,\tc)|}\Id_{B(0,\tc)} \Big)\sim \kappa \Id_{N_{\tc}(\bigcup_{t\in\cI(\wt h_m)}p[\wt\bla](t))},
\end{equation}
where $\kappa\sim\tau/\tc$. To show that $\kappa\sim\tau/\tc$, we first describe the $\tc$-neighborhood using enlarged planks. 

We  increase the shortest side length from $\tau$ to $\tc$
using the same construction as above. Set
\[
\la_{m-1}':=\la_{m-1}(r_*\tc)^{1/2},
\qquad
\bla':=(\la_1,\dots,\la_{m-2},\la_{m-1}').
\]
By \eqref{upperboundla}, $\la_{m-1}'\le1$, so $\bla'$ is
admissible at scale $\tc^{-1}$. The planks $p[\bla'](t)$ have
dimensions
\[
1\times\rho_1\times\cdots\times\rho_{m-2}
\times(\rho_{m-2}\tc)^{1/2}\times\tc
\]
and aperture $h_m':=\la_1\cdots\la_{m-2}\la_{m-1}'$.
By the same reasoning as \eqref{unionplankequal}, we have
\begin{equation*}
\bigcup_{t'\in\cI(\wt h_m):|t'-t|\le  h_m' }p[\wt\bla](t')\subset p[\bla'](t). 
\end{equation*}
The $\tc$-neighborhood of this union is comparable to the enlarged plank $p[\bla'](t)$:
\begin{equation*}
N_{\tc}\Big(\bigcup_{t'\in\cI(\wt h_m):|t'-t|\le  h_m' }p[\wt\bla](t')\Big)\approx p[\bla'](t). 
\end{equation*}
We omit the proof of this geometric comparison. Hence,
\begin{equation}
N_{\tc}\Big(\bigcup_{t\in\cI(\wt h_m) }p[\wt\bla](t)\Big)\approx \bigcup_{t\in \cI(h_m')}p[\bla'](t). 
\end{equation}
Substituting into \eqref{convolution} gives
\[ \Id_{\bigcup_{t\in\cI(\wt h_m)}p[\wt\bla](t)}* \Big(  \frac{1}{|B(0,\tc)|}\Id_{B(0,\tc)} \Big)\sim \kappa \Id_{\bigcup_{t\in \cI(h_m')}p[\bla'](t)}. \]
Since $\frac{1}{|B(0,\tc)|}\Id_{B(0,\tc)}$ is $L^1$-normalized, integrating the preceding comparison gives
\[
\kappa\sim
\frac{\left|\bigcup_{t\in\cI(\wt h_m)}p[\wt\bla](t)\right|}
{\left|\bigcup_{t\in\cI(h_m')}p[\bla'](t)\right|}.
\]
The two planks differ only in their last two side lengths, so
\[
\frac{|p[\wt\bla](t)|}{|p[\bla'](t)|}
\sim\frac{(\rho_{m-2}\tau)^{1/2}\tau}
{(\rho_{m-2}\tc)^{1/2}\tc}
=\left(\frac{\tau}{\tc}\right)^{3/2}.
\]
The parameter nets have cardinalities comparable to the reciprocals of their spacings. Hence,
\[
\frac{\#\cI(\wt h_m)}{\#\cI(h_m')}
\sim\frac{h_m'}{\wt h_m}
=\frac{\la_{m-1}'}{\wt\la_{m-1}}
=\left(\frac{\tc}{\tau}\right)^{1/2}.
\]
By the bounded overlap in Lemma~\ref{lemcompare}, the volume of each union is comparable to the number of planks times the volume of one plank. Therefore,
\[
\kappa\sim
\left(\frac{\tc}{\tau}\right)^{1/2}
\left(\frac{\tau}{\tc}\right)^{3/2}
=\frac{\tau}{\tc}.
\]
Combining the above discussion, we see that 
\[ \mathrm{Low \; part}\  \lesssim  \tau^{-(m+1)} \int\Id_Y \frac{\tau}{\tc}\Id_{\bigcup_{t\in\cI(h'_m)}p[\bla'](t)}\sim \frac{\tau}{\tc}N_m(\tc Q_*,Q_*; \bla').  \]
The last $\sim$ is from \eqref{defanisotropic}, because it is the lattice counting problem with height $Q_*$ and width $\tc Q_*$.

\medskip

We next apply the \textit{projection argument}. As this section presents a proof sketch, we focus on the geometric idea behind the projection. Recall that $N_m(\tc Q_*,Q_*;\bla')$ counts the number of $Q_*^{-1}\ZZ^{m+1}$ in the union of $1\times \rho_1\times\dots\times \rho_{m-2}\times (\rho_{m-2}\tc)^{1/2}\times \tc$-boxes. We project the lattice points and the boxes onto an $m$-dimensional subspace. The projected boxes have dimensions
\[ 1\times \rho_1\times \dots\times \rho_{m-2}\times (\rho_{m-2}\tc)^{1/2}. \]
Noting that
\[ (\rho_{m-2}\tc)^{1/2}\le (\rho_{m-2}\prod_{i=1}^{m-2}\la_i^{m+1-i})^{1/2}\le \prod_{i=1}^{m-2}\la_i^{m-i}, \]
we may increase the shortest side length of the projected box to $\prod_{i=1}^{m-2}\la_i^{m-i}$, so that it becomes of size
\[ 1\times \rho_1\times \dots\times \rho_{m-2}\times \prod_{i=1}^{m-2}\la_i^{m-i}.\]
These are precisely the dimensions of a $(\la_1,\dots,\la_{m-2})$-plank. Thus, the projected counting problem is bounded by $N_{m-1}(\de_{m-1},Q_*;\la_1,\dots,\la_{m-2})$, where $\de_{m-1}=Q_*\prod_{i=1}^{m-2}\la_i^{m-i}$. Since the thickness $\tc$ is at most the lattice spacing $Q_*^{-1}$, fixing the projected coordinates leaves only $O(1)$ choices for the last coordinate. Therefore,
\[ N_m(\tc Q_*,Q_*;\bla')\lesssim N_{m-1}(\de_{m-1},Q_*;\la_1,\dots,\la_{m-2}). \]
We record the general form of this projection estimate, which will also be used in the induction in Section~\ref{conclude}.
\begin{proposition}\label{input0}
    Let $\de_m=Q_*\prod_{i=1}^{m-1}\la_i^{m+1-i}$ and $\de_{m-1}=Q_*\prod_{i=1}^{m-2}\la_i^{m-i}$. Then
    \[ N_m(\de_m,Q_*;\la_1,\dots,\la_{m-1})\lesssim (1+\de_m)\cdot N_{m-1}(\de_{m-1},Q_*;\la_1,\dots,\la_{m-2}). \]
    The additional factor $1+\de_m$ counts the choice of the last component of the lattice points when their first $m$ components are fixed.
\end{proposition}

We summarize the low part estimate as follows:
\begin{equation}
   \mathrm{Low \; part} \ \lesssim \frac{\tau}{\tc}N_{m-1}(\prod_{i=1}^{m-2}\la_i^{m-i}Q_*,Q_*;\la_1,\dots,\la_{m-2}).
\end{equation}

\subsubsection{\textbf{\textup{High part}}} As this section presents a proof sketch, we restrict attention to the highest-frequency contribution, corresponding to $r'=\tau^{-1}$.  We want to bound
\[ \textrm{Highest part}:= \tau^{-(m+1)}\int \Id_Y \bigg(\Id_{\bigcup_{t\in\cI(\wt h_m)}p[\wt\bla](t)}*\eta_{\tau^{-1}}^\vee\bigg). \]
By Plancherel, it equals
\[ \tau^{-(m+1)} \int  \overline{\wh\Id_Y} \bigg(\wh\Id_{\bigcup_{t\in\cI(\wt h_m)}p[\wt\bla](t)}\eta_{\tau^{-1}}\bigg). \]

We first study $\wh \Id_Y$. Recall \eqref{defY}. We define
\[ Y':=N_{1}(B(0,\tau^{-1})\cap Q_*\ZZ^{m+1}). \]
Then by the same argument as \eqref{whFX}, we have
\[ \wh \Id_Y\approx |Y|\Id_{Y'}\sim\tau^{m+1}Q_*^{m+1} \Id_{Y'}. \]
We next study  $\wh\Id_{\bigcup_{t\in\cI(\wt h_m)}p[\wt\bla](t)}$ in $\{|x|\sim \tau^{-1}\}$, the support of $\eta_{\tau^{-1}}$.
By Lemma~\ref{lemcompare}, the boxes $p[\wt\bla](t)$ are essentially disjoint, so we may write
\[\wh\Id_{\bigcup_{t\in\cI(\wt h_m)}p[\wt\bla](t)}\approx\sum_{t\in\cI(\wt h_m)}\wh\Id_{p[\wt\bla](t)}.\]
Recall that $p[\wt\bla](t)$ is a box of dimensions
\[ 1\times \rho_1\times \dots\times \rho_{m-2}\times (\rho_{m-2}\tau)^{1/2}\times\tau. \]
The dual box, denoted by $p[\wt\bla](t)^*$, has dimensions
\[ \tau^{-1}\times (\rho_{m-2}\tau)^{-1/2}\times \rho_{m-2}^{-1}\times \dots\times \rho_1^{-1}\times 1. \]
So we have
\[ |\wh \Id_{p[\wt\bla](t)}|\lesssim |p[\wt\bla](t)|\cdot \Id_{p[\wt\bla](t)^*}\sim (\prod_{i=1}^{m-2}\rho_i)\cdot(\rho_{m-2}\tau)^{1/2}\tau\cdot \Id_{p[\wt\bla](t)^*}.  \]

Combining the preceding estimates, we obtain
\begin{equation}\label{highest1}
    \textrm{Highest part}\ \lesssim (\prod_{i=1}^{m-2}\rho_i)\cdot(\rho_{m-2}\tau)^{1/2}\tau Q_*^{m+1}\int_{|x|\sim \tau^{-1}}  \Id_{Y'}\sum_{t\in \cI(\wt h_m)}\Id_{p[\wt\bla](t)^*}.
\end{equation}
We regard the last integral as counting the number of points of $Q_*\ZZ^{m+1}\cap B(0,\tau^{-1})$ inside the union of boxes with dimensions $\tau^{-1}\times (\rho_{m-2}\tau)^{-1/2}\times \rho_{m-2}^{-1}\times \dots\times \rho_1^{-1}\times 1$. We project the counting problem to a 3-dimensional plane. It becomes the problem of counting points of $Q_*\ZZ^{3}\cap B(0,\tau^{-1})$ inside the union of boxes with dimensions $\tau^{-1}\times (\rho_{m-2}\tau)^{-1/2}\times \rho_{m-2}^{-1}$. Rescaling by the factor $\tau$, we see that this problem has \textit{height} $\tau^{-1}Q_*^{-1}$ and \textit{width} $\rho_{m-2}^{-1}Q_*^{-1}$. Hence, the projected counting problem contributes at most \[N_2(\rho_{m-2}^{-1}Q_*^{-1},\tau^{-1}Q_*^{-1}).\] For the remaining components, since the lattice $Q_*\ZZ^{m+1}$ is $Q_*$-spaced and the remaining side lengths of $p[\wt\bla](t)^*$ are $\rho_{m-3}^{-1},\dots,\rho_1^{-1},1$, they contribute at most 
\[ \prod_{i=1}^{m-3}(1+\rho_{i}^{-1}Q_*^{-1}). \]
Combining all the estimates, we obtain
\[ \int_{|x|\sim \tau^{-1}}  \Id_{Y'}\sum_{t\in \cI(\wt h_m)}\Id_{p[\wt\bla](t)^*}\lesssim N_2(\rho_{m-2}^{-1}Q_*^{-1},\tau^{-1}Q_*^{-1})\prod_{i=1}^{m-3}(1+\rho_{i}^{-1}Q_*^{-1}). \]
Substituting this estimate into \eqref{highest1} yields
\begin{equation}\label{high}
    \textrm{Highest part}\ \lesssim (\prod_{i=1}^{m-2}\rho_i)\cdot(\rho_{m-2}\tau)^{1/2}\tau Q_*^{m+1}N_2(\rho_{m-2}^{-1}Q_*^{-1},\tau^{-1}Q_*^{-1})\prod_{i=1}^{m-3}(1+\rho_{i}^{-1}Q_*^{-1}).
\end{equation}

The estimates above give the low-frequency and highest-frequency contributions to Proposition~\ref{input2}. We omit the intermediate-frequency analysis from this outline.

\bigskip

\subsection{The two-dimensional theorem of Huxley and Vaughan--Velani}
We record the following counting theorem due to Huxley \cite{Huxley1994} and Vaughan--Velani \cite{VV2006}.
\begin{proposition}\label{input3}
    Let $0<\de\le 1\le Q$. Then for any $\e>0$,
    \begin{equation}
        N_2(\de,Q)\lesssim_\e \de Q^2+Q^{1+\e}.
    \end{equation}
\end{proposition}

\bigskip

\subsection{Concluding the main theorem}\label{conclude}
In this subsection we prove the main theorem using
Propositions~\ref{input1}, \ref{input2}, \ref{input0}, and \ref{input3}. The computation is assisted by GPT-6. We state the
result in the dyadic-denominator notation used throughout this section.

\begin{theorem}\label{thm:section-two-main}
Let $0<\de\leq1\leq Q$. Then for every $\e>0$,
\begin{equation}\label{eq:section-two-main}
 N_n(\de,Q)
 \lesssim_\e
 \de^{n-1}Q^2+Q^\e\left(
 Q+\sum_{k=1}^{n-1}
 \de^{\,k^2/(k+1)}Q^{(2k+1)/(k+1)}\right).
\end{equation}
\end{theorem}

\begin{proof}
\textbf{Step 1. Derivation of \eqref{eq:section-two-normalized-recurrence}.}
We first treat $Q^{-1}\leq\de\leq1$.  In
Proposition~\ref{input1}, write
\[
 M:=\frac rQ.
\]
Then the range of $M$ is
\[ 1\leq M\leq\de^{-1}\leq Q.\]
Fix one tuple
$\bla=(\la_1,\ldots,\la_{n-1})$ that is admissible at scale $r$, and put
$h:=\prod_{i=1}^{n-1}\la_i$.  Admissibility gives
\begin{equation}\label{eq:section-two-full-relation}
 \prod_{i=1}^{n-1}\la_i^{\,n+1-i}=r^{-1}=(QM)^{-1}.
\end{equation}

To make the computation simple, we introduce some new parameters. Let
\begin{equation}\label{eq:section-two-Lj}
 \ell_0:=M,
 \qquad
 \ell_j:=M\prod_{i=1}^{j}\la_i^{\,j+1-i}
 \quad(1\leq j\leq n-1),
\end{equation}
and, for $2\leq m\leq n$, set
\begin{equation}\label{eq:section-two-deltam}
 \de_m:=M\prod_{i=1}^{m-1}\la_i^{\,m+1-i}
      =\frac{\ell_{m-1}^2}{\ell_{m-2}}.
\end{equation}
Hence the shortest side of the $m$-dimensional $(\la_1,\dots,\la_{m-1})$-plank is
$\de_mM^{-1}$, and \eqref{eq:section-two-full-relation} says that
$\de_n=Q^{-1}$.  Write
\[
 \mathbf N_m:=N_m(\de_m,M;\la_1,\ldots,\la_{m-1}).
\]
Start from $m=n$.
If $\de_m\le 1$, we use Proposition \ref{input2}. If $\de_m>1$, the iteration stops and we use the volume bound \eqref{volumebound}.   Introduce also the volume scale
\begin{equation}\label{eq:section-two-volume}
 \mathcal V_m
 :=\left(\prod_{j=0}^{m-2}\ell_j\right)\ell_{m-1}^2.
\end{equation}
This is the volume of the union of the  $(\la_1,\dots,\la_{m-1})$-planks multiplied by the
lattice density $M^{m+1}$. The same reasoning as in \eqref{volumebound} shows that if $\de_m>1$, then
\begin{equation}\label{volumeestimate}
    \mathbf N_m\lesssim \mathcal V_m.
\end{equation} 

We next rewrite Proposition~\ref{input2} in these parameters. We use $(\de_m,M)$ in place of the $(\de_*,Q_*)$ from Proposition \ref{input2}. For a fixed
$m\geq3$, put
\begin{equation}\label{eq:section-two-local-parameters}
 \ell:=\ell_{m-2},\qquad \ell_-:=\ell_{m-3},\qquad
 S:=\prod_{j=0}^{m-3}\ell_j,\qquad
 g:=\prod_{i=1}^{m-2}\la_i=\frac \ell{\ell_-},\qquad
 w:=\ell g^2=\frac{\ell^3}{\ell_-^2}.
\end{equation}
Since $\rho_i=\ell_i/M$ for $i\leq m-2$, the cutoff in
Proposition~\ref{input2} is
\[
 \tc=M^{-1}\min\{1,w\}.
\]
If $\tau$ denotes the parameter in Proposition~\ref{input2}, put
$\sigma:=M\tau$.  Thus the condition $\tau\in [\de_* Q_*^{-1},\tc]$ can be rewritten as
\begin{equation}\label{eq:section-two-sigma-range}
 \de_m\leq\sigma\leq\min\{1,w\}.
\end{equation}

The projected count $N_{m-1}(\prod_{i=1}^{m-2}\la_i^{m-i}Q_*,Q_*;\la_1,\dots,\la_{m-2})$ in Proposition~\ref{input2} is $\mathbf N_{m-1}$, and
\[
 \frac{\tau}{\tc}
 =\frac{\sigma}{\min\{1,w\}}
 \sim\sigma(1+w^{-1}).
\]
Moreover,
\[
 \prod_{i=1}^{m-2}\rho_i=\frac{S\ell}{M^{m-1}},
 \qquad
 (\rho_{m-2}\tau)^{1/2}=\frac{(\ell\sigma)^{1/2}}M,
\]
and the planar count $N_2(\rho_{m-2}^{-1}Q_*^{-1},\tau^{-1}Q_*^{-1})$ appearing there is
\[
 N_2(\ell^{-1},\sigma^{-1}).
\]
The full coefficient multiplying this planar count is therefore
\[
 \left(\prod_{i=1}^{m-2}\rho_i\right)
 (\rho_{m-2}\tau)^{1/2}\tau M^{m+1}
 =S\ell^{3/2}\sigma^{3/2}.
\]
When $\ell>1$, the remaining fiber factor is harmless: indeed
$\ell_0\geq \ell_1\geq\cdots\geq \ell_{m-2}=\ell$, and hence
\[
 \prod_{i=1}^{m-3}(1+\rho_i^{-1}M^{-1})
 =\prod_{i=1}^{m-3}(1+\ell_i^{-1})\lesssim_m1.
\]
Since $\ell>1$ and $\sigma\leq1$, the hypotheses
$\ell^{-1}\leq1\leq\sigma^{-1}$ of Proposition~\ref{input3} hold, and it gives
\[
 N_2(\ell^{-1},\sigma^{-1})
 \lesssim_\e \ell^{-1}\sigma^{-2}+\sigma^{-1-\e}.
\]
In the rest of this subsection, we write $X\lessapprox Y$ if, for every $\eta>0$,
$X\lesssim_\eta Q^\eta Y$. Notice that $\si\ge \de_m\ge \de_n=Q^{-1}$. Hence $\sigma^{-\e}$ is absorbed by $Q^\e$.
Consequently Propositions~\ref{input2} and \ref{input3} imply that if $\ell>1$ and \eqref{eq:section-two-sigma-range} hold, then  
\begin{equation}\label{eq:section-two-normalized-recurrence}
 \mathbf N_m\lessapprox
 \sigma(1+w^{-1})\mathbf N_{m-1}
 +S\left(\frac \ell\sigma\right)^{1/2}
 +S\ell^{3/2}\sigma^{1/2}.
\end{equation}

\medskip
\noindent\textbf{Step 2. The inductive bound and preliminary estimates.}
We now iterate \eqref{eq:section-two-normalized-recurrence}.  For
$0\leq k\leq n-1$, let
\begin{equation}\label{eq:section-two-Dnk}
 D_{n,k}:=(n-1)(k+1)-k^2=n-1+k(n-1-k),
\end{equation}
and define
\begin{equation}\label{eq:section-two-En}
 E_n(Q^{-1},M)
 :=\max_{0\leq k\leq n-1}
 \left(Q^{-1}M^{D_{n,k}}\right)^{1/(k+1)},
 \qquad
 U:=\frac{E_n(Q^{-1},M)}h.
\end{equation}
We claim that $U$ and the volume scale $\mathcal V_m$ control $\mathbf N_m$:
\begin{equation}\label{eq:section-two-induction}
 \mathbf N_m\lessapprox\mathcal V_m+U,
 \qquad 2\leq m\leq n.
\end{equation}

We will prove this bound by induction on $m$. We first establish the numerical inequalities needed for the induction. From
\eqref{eq:section-two-full-relation} and \eqref{eq:section-two-Lj}, for $0\leq j\leq n-1$,
\[
 Q\ell_j=\prod_{i=1}^j\la_i^{-(n-j)}\prod_{i=j+1}^{n-1}\la_i^{-(n+1-i)}\leq h^{-(n-j)}.
\]
Fix $0\leq j\leq n-1$ and choose $k=n-j-1$ in
\eqref{eq:section-two-En}. Then 
\[
 D_{n,n-j-1}
 =(n-1)(n-j)-(n-j-1)^2
 =(j+1)(n-j)-1.
\]
Therefore,
\[
 U
 \geq h^{-1}
 \left(Q^{-1}M^{D_{n,n-j-1}}\right)^{1/(n-j)}
 =h^{-1}Q^{-1/(n-j)}M^{j+1-1/(n-j)}.
\]
Moreover, the preceding estimate for $Q\ell_j$ gives
$h^{-1}\geq(Q\ell_j)^{1/(n-j)}$.  Substituting this into the preceding inequality yields
\begin{equation}\label{eq:section-two-basic-U}
 U
 \geq (Q\ell_j)^{1/(n-j)}Q^{-1/(n-j)}
 M^{j+1-1/(n-j)} 
 =M^{j+1}\left(\frac{\ell_j}{M}\right)^{1/(n-j)}.
\end{equation}
In particular,
\begin{equation}\label{eq:section-two-U-consequences}
 U\geq M,\qquad
 U\geq M^{m-3/2}\ell^{1/2},\qquad
 U\geq M^{m-1}\ell_{m-1}.
\end{equation}
Indeed, for $j=m-2$ the exponent in
\eqref{eq:section-two-basic-U} is $1/(n-m+2)\leq1/2$, while $\ell/M\leq1$.
For $j=m-1$, use $1/(n-m+1)\leq1$ and $\ell_{m-1}/M\leq1$.
Since every $\ell_j\leq M$ and $S\leq M^{m-2}$, it follows that
\begin{equation}\label{eq:section-two-SL}
 S\ell\le M^{m-2}M^{1/2}\ell^{1/2} \leq U.
\end{equation}
We shall also use
\begin{equation}\label{eq:section-two-min-bound}
 \min\{S\ell^2,S\ell_-^{2/3}\}
 \leq (S\ell^2)^{1/4}(S\ell_-^{2/3})^{3/4}
 =S(\ell \ell_-)^{1/2}\le M^{m-2}\ell^{1/2}M^{1/2} \leq U
\end{equation}
and
\begin{equation}\label{eq:section-two-volume-over-w}
 \frac{\mathcal V_m}{w}\leq U.
\end{equation}
Indeed,
$\ell_{m-1}\leq \ell^2/\ell_-$ and
$\mathcal V_m=S\ell\,\ell_{m-1}^2$, so
\[
 \frac{\mathcal V_m}{w}
 \leq S\ell_-\ell_{m-1}
 \leq M^{m-1}\ell_{m-1}\leq U.
\]

\medskip
\noindent\textbf{Step 3. Proof of \eqref{eq:section-two-induction} by induction.}
For the case $m=2$ of \eqref{eq:section-two-induction}, if $\de_2\geq1$ the
elementary volume estimate \eqref{volumeestimate} gives $\mathbf N_2\lesssim\mathcal V_2$.  If
$\de_2<1$, Proposition~\ref{input3} gives
$\mathbf N_2\lessapprox\de_2M^2+M=\mathcal V_2+M$, which is bounded by
$\mathcal V_2+U$ because $U\geq M$.

Suppose the claim is known in dimension $m-1$. If $\de_m\geq1$, the volume estimate \eqref{volumeestimate} gives
$\mathbf N_m\lessapprox\mathcal V_m$.  If $\de_m<1$, Proposition \ref{input0} gives
$\mathbf N_m\lessapprox(1+\de_m)\mathbf N_{m-1}\lesssim\mathbf N_{m-1}$. Thus,
when $\mathcal V_{m-1}\leq U$, the induction hypothesis already gives
$\mathbf N_m\lessapprox U$.

It remains only to consider
\begin{equation}\label{eq:section-two-hard-case}
 \de_m<1,
 \qquad
 \mathcal V_{m-1}=S\ell^2>U.
\end{equation}
The induction hypothesis then gives
$\mathbf N_{m-1}\lessapprox S\ell^2$.  Since $S\ell^2>U\geq S\ell$ (see \eqref{eq:section-two-SL}), we have
$\ell>1$.  Furthermore, \eqref{eq:section-two-min-bound} implies
$S\ell_-^{2/3}\leq U$.  Comparing this with $S\ell^2>U$ gives
$\ell>\ell_-^{1/3}$, and hence
\begin{equation}\label{eq:section-two-admissibility-bound}
 \frac{S\ell_-}{\ell}\leq S\ell_-^{2/3}\leq U.
\end{equation}

Choose
\begin{equation}\label{eq:section-two-sigma-choice}
 T:=\max\{\mathcal V_m,U\},
 \qquad
 \sigma_0:=\frac{S^2\ell}{T^2},
 \qquad
 \sigma=\max\{\de_m,\sigma_0\}.
\end{equation}
We verify that $\sigma$ satisfies \eqref{eq:section-two-sigma-range} and that each of the three terms on the right-hand side of \eqref{eq:section-two-normalized-recurrence} is $\lessapprox T$.

Because $S\ell\leq U\leq T$ and $\ell>1$,
\[
 \sigma_0=\frac{(S\ell)^2}{T^2\ell}\leq \ell^{-1}<1.
\]
Also, by \eqref{eq:section-two-admissibility-bound},
\[
 \frac{\sigma_0}{w}
 =\left(\frac{S\ell_-}{\ell T}\right)^2\leq1.
\]
Finally,
\[
 \de_m=w\la_{m-1}^2\leq w.
\]
Together with \eqref{eq:section-two-hard-case}, these estimates verify
\eqref{eq:section-two-sigma-range}.

We apply \eqref{eq:section-two-normalized-recurrence}.  The middle term is
at most $T$, since
\[ S\left(\frac{\ell}{\sigma}\right)^{1/2} \le S\left(\frac{\ell}{\sigma_0}\right)^{1/2}=T.\]
For the first term, 
\[
\begin{aligned}
 \sigma(1+w^{-1})\mathbf N_{m-1}
 &\lessapprox(\de_m+\sigma_0)(1+w^{-1})S\ell^2 \\
 &=\mathcal V_m+\frac{\mathcal V_m}{w}
   +\frac{(S\ell)^3+(S\ell_-^{2/3})^3}{T^2} \\
 &\leq T+U+\frac{2U^3}{T^2}
 \leq4T.
\end{aligned}
\]
Indeed, the second line uses
\[
 \de_mS\ell^2=\mathcal V_m,
 \qquad
 \sigma_0S\ell^2=\frac{(S\ell)^3}{T^2},
 \qquad
 \frac{(S\ell)^3}{w}=(S\ell_-^{2/3})^3.
\]
The last line follows from $\mathcal V_m\leq T$,
\eqref{eq:section-two-volume-over-w}, the bounds
$S\ell\leq U$ and $S\ell_-^{2/3}\leq U$ obtained above, and $U\leq T$.
Finally,
\[
 \bigl(S\ell^{3/2}\sigma^{1/2}\bigr)^2
 =\max\left\{(S\ell)\mathcal V_m,\frac{(S\ell)^4}{T^2}\right\}
 \leq T^2.
\]
This closes the induction and proves \eqref{eq:section-two-induction}.

\medskip
\noindent\textbf{Step 4. Completion of the proof of Theorem~\ref{thm:section-two-main}.}
At $m=n$, the volume term can be absorbed into $E_n(Q^{-1},M)$.
We first compute it directly from \eqref{eq:section-two-Lj} and
\eqref{eq:section-two-volume}:
\[
\begin{aligned}
 h\mathcal V_n
 &=\left(\prod_{i=1}^{n-1}\la_i\right)
   \left(\prod_{j=0}^{n-2}
      \left(M\prod_{i=1}^{j}\la_i^{\,j+1-i}\right)\right)
   \left(M\prod_{i=1}^{n-1}\la_i^{\,n-i}\right)^2 \\
 &=M^{n+1}\prod_{i=1}^{n-1}
   \la_i^{\,1+\sum_{j=i}^{n-2}(j+1-i)+2(n-i)} \\
 &=M^{n+1}\prod_{i=1}^{n-1}
   \la_i^{(n+1-i)(n+2-i)/2}.
\end{aligned}
\]
The admissibility relation \eqref{eq:section-two-full-relation} therefore gives
\[
\begin{aligned}
 h\mathcal V_n
 &=M^{n+1}
   \left(\prod_{i=1}^{n-1}\la_i^{\,n+1-i}\right)
   \prod_{i=1}^{n-1}\la_i^{(n-i)(n+1-i)/2} \\
 &=Q^{-1}M^n
   \prod_{i=1}^{n-1}\la_i^{(n-i)(n+1-i)/2}.
\end{aligned}\]
Consequently, using $M\le \de^{-1}\leq Q$,
\[
 h\mathcal V_n
 \leq Q^{-1}M^n(Q^{-1}M^{-1})^{1/2}  
 \leq Q^{-1}M^{n-1}
 \leq E_n(Q^{-1},M).
\]

Multiplying \eqref{eq:section-two-induction} by $h$ gives the estimate needed
in Proposition~\ref{input1}:
\begin{equation}\label{eq:section-two-inner-bound}
 hN_n(Q^{-1},M;\bla)
 \lessapprox E_n(Q^{-1},M)
 \leq\sum_{k=0}^{n-1}
 Q^{-1/(k+1)}M^{D_{n,k}/(k+1)}.
\end{equation}

It remains to sum over the dyadic radii.  There are only $O((\log(2Q))^{n-1})$
admissible dyadic $\bla$-tuples at each radius, while the auxiliary set $A$
has $O_n(1)$ possibilities.  Since $r$ is dyadic, $M=r/Q$ ranges over a
dyadic sequence.  As $D_{n,k}>0$,
\[
 \sum_{\substack{1\leq M\leq\de^{-1}\\M=r/Q,\ r\ {\rm dyadic}}}
 M^{D_{n,k}/(k+1)}
 \lesssim \de^{-D_{n,k}/(k+1)}.
\]
Substituting \eqref{eq:section-two-inner-bound} into
Proposition~\ref{input1}, we find that the $k$-th contribution is
\begin{align*}
 \de^{n-1}Q^2Q^{-1/(k+1)}
 \de^{-D_{n,k}/(k+1)}
 &=\de^{\,k^2/(k+1)}Q^{(2k+1)/(k+1)}.
\end{align*}
This finishes the proof when $\de\ge Q^{-1}$.

Finally, if $0<\de<Q^{-1}$, monotonicity gives
$N_n(\de,Q)\leq N_n(Q^{-1},Q)$.  The estimate already proved at
$\de=Q^{-1}$ is $O_\e(Q^{1+\e})$: indeed the exponent of $Q$ in every
$k\geq1$ term is
$(2k+1-k^2)/(k+1)\leq1$, and the main term is also $O(Q)$.  Since the
right-hand side of \eqref{eq:section-two-main} contains $Q^{1+\e}$, this
proves the remaining range.
\end{proof}

\section{Reduction to estimating oscillatory integrals}

In this section, we reduce Theorem~\ref{thm:main} to Proposition~\ref{prop:conical-n}.
Fix nonnegative functions
\[
\omega\in C_c^\infty((1/2,3)),
\qquad
\chi\in C_c^\infty(\mathbb R),
\]
such that
\[
\omega(s)\geq 1 \quad\text{for }1\leq s\leq 2,
\qquad
\chi(t)\geq 1 \quad\text{for }0\leq t\leq 1.
\]
We write $e(u):=\exp(2\pi i u)$ for $u\in\mathbb R$.
We may assume that $\gamma_1(t)=t$. Set $\gamma_0(t):=1$ and write
\[
    \widetilde{\gamma}(t)
    :=(\gamma_0(t),\ldots,\gamma_n(t))=(1,\gamma(t)),
\]
and define the compactly supported measure $\mu_{\omega,\chi}$ on
$\mathbb{R}^{n+1}$ by
\begin{equation}\label{defmu}
    \int_{\mathbb{R}^{n+1}} F(x)\,d\mu_{\omega,\chi}(x)
    :=
    \int_{\mathbb{R}}\int_{\mathbb{R}}
    F\bigl(s\widetilde{\gamma}(t)\bigr)
    \omega(s)\chi(t)\,dt\,ds,
    \qquad F\in C_c(\mathbb{R}^{n+1}).
\end{equation}
We then have
\[
    \widehat{\mu_{\omega,\chi}}(\tau,\xi)
    =
    \int_{\mathbb{R}}\int_{\mathbb{R}}
    \omega(s)\chi(t)
    e\bigl(-s(\tau+\xi\cdot\gamma(t))\bigr)
    \,dt\,ds,
    \qquad
    (\tau,\xi)\in\mathbb{R}\times\mathbb{R}^n.
\]
This can be understood as a Fourier extension operator for the following manifold:
\begin{equation}
    \Gamma= \{ s\widetilde{\gamma}(t): s,t \in \R \}.
\end{equation}
We use the notation
\[
\mathbf j=(j_0,j_1,\ldots,j_n),\qquad
\mathbf j_{\leq i}=(j_0,\ldots,j_i),\qquad
\mathbf j_{\geq i}=(j_i,\ldots,j_n)
\]
for $\mathbf j\in\mathbb Z^{n+1}$ and $0\leq i\leq n$.
Recall from \eqref{eq:section-two-En} that
\[
E_n(\delta,Q)
=\max_{0\le k\le n-1}
\bigl(\delta Q^{(n-1)(k+1)-k^2}\bigr)^{1/(k+1)}.
\]
We show that Theorem~\ref{thm:main} can be deduced from the following proposition.

\begin{proposition}
\label{prop:conical-n}
Let $n\ge3$, $Q\geq 1$, $Q^{-1}\le\delta\le 1$, and
$\Psi\in C_c^\infty(\mathbb R^{n+1})$. Then, for every $\epsilon>0$,
\[
\left|
    \sum_{{
        \mathbf{j}\in\mathbb{Z}^{n+1}:
        \mathbf{j}_{\geq 1}\neq 0
    }}
    \Psi(\delta\mathbf{j})
    \widehat{\mu_{\omega,\chi}}(Q\mathbf{j})
\right|
\lesssim Q^{\epsilon}E_n(Q^{-1},\delta^{-1}).
\]
\end{proposition}

\begin{proof}[Proof of Theorem~\ref{thm:main} assuming Proposition~\ref{prop:conical-n}]
We write
\begin{equation}\label{eq:dyadic-decomposition}
\mathcal R_{\mathcal C}(\delta,Q)
=\sum_{{j\ge0:2^{-j}Q\ge2}}
\bigl(\mathcal R_{\mathcal C}(\delta,2^{-j}Q)
-\mathcal R_{\mathcal C}(\delta,2^{-j-1}Q)\bigr)
+O_{\mathcal C}(1).
\end{equation}
We will show that the $j=0$ term satisfies
\begin{equation}\label{eq:dyadic-bound}
\begin{aligned}
\mathcal R_{\mathcal C}(\delta,Q)
-\mathcal R_{\mathcal C}(\delta,Q/2)\lesssim_\epsilon
\delta^{n-1}Q^2
+Q^\epsilon\left(
Q+\sum_{k=1}^{n-1}
\delta^{k^2/(k+1)}Q^{(2k+1)/(k+1)}
\right).
\end{aligned}
\end{equation}
The same estimate applies to every other term after replacing $Q$ by $2^{-j}Q$.
Since all powers of $Q$ on the right hand side are positive, summing these
estimates gives \eqref{eq:intro-refined-bound}. It therefore suffices to prove
\eqref{eq:dyadic-bound} for $j=0$, which counts the pairs with $Q/2<q\le Q$.

For $\xi\in\mathbb R^n$, write
\[
\widehat{d\sigma_\chi}(\xi)
:=\int_{\mathbb R}\chi(t)e(-\xi\cdot\gamma(t))\,dt.
\]
We consider two cases depending on the size of $\delta$:
\begin{equation}
    2Q^{-1} \leq \delta \leq 1, \;\;\;\;\;\;\;\;\;\;\;\; 0<\delta<2Q^{-1}.
\end{equation}

\subsection*{Case 1: \texorpdfstring{$2Q^{-1}\le\delta\le1$}{}}
Choose a nonnegative $\Phi\in\mathcal S(\mathbb R^n)$ such that
$\Phi(x)\ge1$ for $|x|\le2$ and $\widehat{\Phi}\in C_c^\infty(\mathbb R^n)$.
For every $\mathbf p\in\mathbb Z^n$ and every integer $q$ with $Q/2<q\le Q$, one may see that
\[
\mathbf1_{\{\operatorname{dist}(\mathbf p/q,\mathcal C)\le\delta/q\}}
\lesssim_\gamma\frac Q\delta
\int_{\mathbb R}\chi(t)
\Phi\!\left(\frac{\mathbf p-q\gamma(t)}{\delta}\right)\,dt.
\]
Summing over $\mathbf p$ and $q$, and using $\omega(2q/Q)\ge1$ for
$Q/2<q\le Q$, we obtain
\begin{equation}\label{eq:counting-integral-majorant}
\begin{aligned}
\mathcal R_{\mathcal C}(\delta,Q)-\mathcal R_{\mathcal C}(\delta,Q/2)\lesssim_\gamma\frac Q\delta
\sum_{q\in\mathbb Z}\omega(2q/Q)
\int_{\mathbb R}\chi(t)\sum_{\mathbf p\in\mathbb Z^n}
\Phi\!\left(\frac{\mathbf p-q\gamma(t)}{\delta}\right)\,dt.
\end{aligned}
\end{equation}
Poisson summation in $\mathbf p$ gives
\[
\sum_{\mathbf p\in\mathbb Z^n}
\Phi\!\left(\frac{\mathbf p-q\gamma(t)}{\delta}\right)
=\delta^n\sum_{\mathbf j_{\geq 1}\in\mathbb Z^n}
\widehat{\Phi}(\delta\mathbf j_{\geq 1})e(-q\mathbf j_{\geq 1}\cdot\gamma(t)).
\]
Separating the $\mathbf j_{\geq 1}=0$ term, we obtain
\begin{equation}
\begin{aligned}
&\mathcal R_{\mathcal C}(\delta,Q)-\mathcal R_{\mathcal C}(\delta,Q/2)\\
&\quad\lesssim_\gamma\frac Q\delta\delta^n
\sum_{q\in\mathbb Z}\sum_{\mathbf j_{\geq 1}\in\mathbb Z^n}
\omega(2q/Q)\widehat{\Phi}(\delta\mathbf j_{\geq 1})
\int_{\mathbb R}\chi(t)e(-q\mathbf j_{\geq 1}\cdot\gamma(t))\,dt\\
&\quad\lesssim\delta^{n-1}Q^2+\delta^{n-1}Q
\left|\sum_{q\in\mathbb Z}\omega(2q/Q)
\sum_{\mathbf j_{\geq 1}\in\mathbb Z^n\setminus\{0\}}
\widehat{\Phi}(\delta\mathbf j_{\geq 1})\widehat{d\sigma_\chi}(q\mathbf j_{\geq 1})\right|.
\end{aligned}
\end{equation}
In the remaining term, apply Poisson summation in $q$:
\[
\begin{aligned}
&\sum_{q\in\mathbb Z}\omega(2q/Q)
\sum_{\mathbf j_{\geq 1}\in\mathbb Z^n\setminus\{0\}}
\widehat{\Phi}(\delta\mathbf j_{\geq 1})
\widehat{d\sigma_\chi}(q\mathbf j_{\geq 1})=\frac Q2\sum_{\substack{\mathbf j\in\mathbb Z^{n+1}:
\mathbf j_{\geq 1}\neq0}}
\widehat{\Phi}(\delta\mathbf j_{\geq 1})
\widehat{\mu_{\omega,\chi}}\bigl(\tfrac Q2\mathbf j\bigr).
\end{aligned}
\]
Hence, it suffices to prove that
\[
\left|\sum_{\substack{\mathbf j\in\mathbb Z^{n+1}:
\mathbf j_{\geq 1}\neq0}}
\widehat{\Phi}(\delta\mathbf j_{\geq 1})
\widehat{\mu_{\omega,\chi}}\bigl(\tfrac Q2\mathbf j\bigr)\right|
\lesssim_\epsilon Q^\epsilon E_n(Q^{-1},\delta^{-1}).
\]
To prove this estimate, choose $\rho\in C_c^\infty(\mathbb R)$ equal to $1$
on a sufficiently large neighborhood of the origin, and set
$\Psi(u_0,u):=\rho(u_0)\widehat{\Phi}(u)$. We decompose the sum into two terms
\[
\begin{aligned}
\sum_{\substack{\mathbf j\in\mathbb Z^{n+1}\\
\mathbf j_{\geq 1}\neq0}}
\Psi(\delta\mathbf j)
\widehat{\mu_{\omega,\chi}}\bigl(\tfrac Q2\mathbf j\bigr)+\sum_{\substack{\mathbf j\in\mathbb Z^{n+1}\\
\mathbf j_{\geq 1}\neq0}}
\bigl(1-\rho(\delta j_0)\bigr)\widehat{\Phi}(\delta\mathbf j_{\geq 1})
\widehat{\mu_{\omega,\chi}}\bigl(\tfrac Q2\mathbf j\bigr).
\end{aligned}
\]
For the first term, Proposition~\ref{prop:conical-n}, applied with height $Q/2$, gives
\[
\begin{aligned}
\left|\sum_{\substack{\mathbf j\in\mathbb Z^{n+1}:
\mathbf j_{\geq 1}\neq0}}
\Psi(\delta\mathbf j)
\widehat{\mu_{\omega,\chi}}\bigl(\tfrac Q2\mathbf j\bigr)\right|
&\lesssim_\epsilon Q^\epsilon E_n(2Q^{-1},\delta^{-1})\lesssim_\epsilon Q^\epsilon E_n(Q^{-1},\delta^{-1}).
\end{aligned}
\]
By the choice of $\rho$, every nonzero summand in the second term satisfies
\[
\left|j_0+\mathbf j_{\geq 1}\cdot\gamma(t)\right|\gtrsim |j_0|\gtrsim\delta^{-1}
\qquad(t\in\operatorname{supp}\chi).
\]
Integration by parts in $s$ gives
\[
\left|\widehat{\mu_{\omega,\chi}}\bigl(\tfrac Q2\mathbf j\bigr)\right|
\lesssim_M (Q|j_0|)^{-M}.
\]
Summing this estimate with $M$ sufficiently large shows that the second term is
$O_N(Q^{-N})$ for every $N>0$.
Combining the two estimates proves the desired bound. 

\bigskip

\subsection*{Case 2: \texorpdfstring{$0<\delta<2Q^{-1}$}{}}
We will show that the $j=0$ term in \eqref{eq:dyadic-decomposition} satisfies
\[
\mathcal R_{\mathcal C}(\delta,Q)-\mathcal R_{\mathcal C}(\delta,Q/2)
\lesssim_{\epsilon,\gamma}Q^{1+\epsilon}.
\]
Since $\gamma_1(t)=t$ and $\gamma''(t)\ne0$, after a finite subdivision
and a permutation of the last $n-1$ coordinates, it suffices to treat
$\mathcal C=\gamma(I)$ with
\[
|\gamma_2''(t)|\gtrsim_\gamma1\quad(t\in I).
\]
Define the planar curve
\[
\mathcal C_{12}:=\{(t,\gamma_2(t)):t\in I\}\subset\mathbb R^2.
\]
For a counted pair $(\mathbf p,q)$, there is $t\in I$ such that
\[
|\mathbf p-q\gamma(t)|\le\delta
\quad\Longrightarrow\quad
\begin{cases}
|(p_1,p_2)-q(t,\gamma_2(t))|\le\delta,\\
|t-p_1/q|\le\delta/q.
\end{cases}
\]
Fix $p_1,p_2,q$ and choose
\[
t_0\in I\cap\left[\frac{p_1}q-\frac\delta q,\frac{p_1}q+\frac\delta q\right]
\]
if this intersection is nonempty; otherwise there are no points to count.
Then $|t-t_0|\le2\delta/q$, and for $3\le r\le n$,
\[
|p_r-q\gamma_r(t_0)|
\le\delta+q\|\gamma_r'\|_{L^\infty(I)}|t-t_0|
\lesssim_\gamma\delta.
\]
Consequently,
\[
\#\left\{(p_3,\ldots,p_n)\in\mathbb Z^{n-2}:
\operatorname{dist}\!\left(\frac{(p_1,\ldots,p_n)}q,\mathcal C\right)
\le\frac\delta q\right\}
\lesssim_\gamma(1+\delta)^{n-2}\lesssim_\gamma1.
\]
The planar estimate \eqref{eq:intro-planar-bound} therefore gives
\[
\begin{aligned}
\mathcal R_{\mathcal C}(\delta,Q)-\mathcal R_{\mathcal C}(\delta,Q/2)
&\lesssim_\gamma
\mathcal R_{\mathcal C_{12}}(\delta,Q)-\mathcal R_{\mathcal C_{12}}(\delta,Q/2)\\
&\lesssim_{\epsilon,\gamma}\delta Q^2+Q^{1+\epsilon}
\lesssim_{\epsilon,\gamma}Q^{1+\epsilon},
\end{aligned}
\]
where the last inequality uses $\delta<2Q^{-1}$.
This proves \eqref{eq:dyadic-bound} in Case~2 and completes the proof.
\end{proof}

\section{Fourier decay estimate}

\label{sec:fourier-decay-n}

The goal of this section is to reduce Proposition~\ref{prop:conical-n}
to Proposition~\ref{prop:n-derivative-counting}.
\subsection{Non-stationary points}\label{sec:n-nonstationary}

In this subsection, we first prove that  Proposition \ref{prop:conical-n} follows from the following. For every $\epsilon>0$,
\begin{equation}\label{eq:n-single-shell}
\sum_{\substack{\mathbf j\in\mathbb Z^{n+1}\\
 |\mathbf j|\sim \delta^{-1},\ \mathbf j_{\geq 2}\neq0}}
\left|\widehat\mu_{\omega,\chi}(Q\mathbf j)\right|
\lesssim_\epsilon Q^\epsilon E_n(Q^{-1},\delta^{-1}).
\end{equation}
To show this, let us prove that the terms with $\mathbf j_{\geq 2}=0$ and $j_1\neq0$ are harmless.
Indeed, if $\mathbf j_{\geq 2}=0$, then 
\begin{equation}
    \widehat\mu_{\omega,\chi}
 \bigl(Q(j_0,j_1,0,\ldots,0) \bigr) = \iint \omega(s)\chi(t) e \left( -Qs(j_0+j_1t) \right)dtds.
\end{equation}
By the integration by parts, and from the fact that the support of $\omega$ is away from the origin, we obtain
\[
\left|\widehat\mu_{\omega,\chi}
 \bigl(Q(j_0,j_1,0,\ldots,0)\bigr)\right|
\lesssim_L
\bigl(1+Q|(j_0,j_1)|\bigr)^{-L}
\]
for every $L>0$. Taking $L>2$ and summing over $j_0\in\mathbb Z$ and $j_1\in\mathbb Z\setminus\{0\}$ gives 
\begin{equation}
    \sum_{\substack{\mathbf j\in\mathbb Z^{n+1}\\
 |\mathbf j|\lesssim \delta^{-1},\ \mathbf j_{\geq 2}=0,\ j_1\neq0}}
\left|\widehat\mu_{\omega,\chi}(Q\mathbf j)\right|
\lesssim_L \sum_{\substack{j_0\in\mathbb Z\\ j_1\in\mathbb Z\setminus\{0\}}} \big(1+Q|(j_0,j_1)| \big)^{-L}   \lesssim Q^{-L}.
\end{equation}
Thus the terms with $\mathbf j_{\geq 2}=0$ and $j_1\neq0$ are harmless.
We next decompose $|\mathbf j|\lesssim\delta^{-1}$ into dyadic shells.
Their frequency sizes are at most a constant multiple of $\delta^{-1}\le Q$.
Since every $D_{n,k}$ is positive, $E_n(Q^{-1},\delta^{-1})$
is nondecreasing in its second argument. Multiplying that argument
by a constant changes the bound by at most a factor depending on
that constant and $n$.
Thus every shell is bounded by the right hand side of
\eqref{eq:n-single-shell}. There are $O(\log(2Q))$ shells;
using a smaller exponent and the boundedness of $\Psi$ proves
Proposition~\ref{prop:conical-n}.

\bigskip

\subsection{Gan--Maldague--Oh argument}
Let $0<\delta\leq1$. We take
$
 R:=Q\delta^{-1}.
$
For $\xi=(\xi_0,\ldots,\xi_n)\in\mathbb R^{n+1}$ define
\[
P_\xi(t):=\xi\cdot\widetilde\gamma(t)
=\xi_0+\xi_1t+\sum_{k=2}^n\xi_k\gamma_k(t).
\]
Then 
\[
\widehat\mu_{\omega,\chi}(\xi)
=\int\!\!\int\omega(s)\chi(t)e\bigl(-sP_\xi(t)\bigr)
\,dt\,ds,
\qquad
P_{Q\mathbf j}^{(k)}=QP_{\mathbf j}^{(k)}.
\]
Choose a dyadic number $\lambda_1$ so that
\begin{equation}
    R^{-\frac1n} \leq \lambda_1 \leq 1.
\end{equation}
For $2\leq i\leq n-2$, choose $\lambda_i$ successively from
the dyadic values in
\begin{equation}\label{eq:n-lambda-range}
\left(R\prod_{j=1}^{i-1}\lambda_j^{\,n+1-j}\right)
 ^{-1/(n+1-i)}
\leq\lambda_i\leq1,
\end{equation}
including the lower endpoint. Define the last scale by
\begin{equation}\label{eq:n-terminal-scale}
\lambda_{n-1}:=
\left(R\prod_{j=1}^{n-2}\lambda_j^{\,n+1-j}\right)^{-1/2},
\qquad
\lambda_{\mathrm{prod}}:=\prod_{j=1}^{n-1}\lambda_j.
\end{equation}
Write
$\boldsymbol\lambda=(\lambda_1,\ldots,\lambda_{n-2})$;
$\lambda_{n-1}$ is determined by this tuple. The ranges above
ensure that all the $\lambda_i$ belong to $(0,1]$ and
\begin{equation}\label{eq:n-scale-normalization}
R\prod_{j=1}^{n-1}\lambda_j^{\,n+1-j}=1.
\end{equation}
If $\lambda_i$ reaches the lower endpoint in \eqref{eq:n-lambda-range}, then every subsequent
$\lambda_j$ ($j>i$) equals $1$.
Dividing \eqref{eq:n-scale-normalization} by $\lambda_{\mathrm{prod}}^2$ gives
\begin{equation}\label{eq:n-h-formula}
\displaystyle
\lambda_{\mathrm{prod}}^{-1}=\left(R\prod_{j=1}^{n-2}
                  \lambda_j^{\,n-1-j}\right)^{1/2}.
\end{equation}
The number $\lambda_{\mathrm{prod}}$ will be the decay rate of the Fourier transform of a surface measure associated to $\lambda_1,\ldots,\lambda_{n-2}$.
Since $2\leq n+1-j\leq n$ and $0<\lambda_j\leq1$,
\[
\lambda_{\mathrm{prod}}^n\leq\prod_{j=1}^{n-1}\lambda_j^{\,n+1-j}
=R^{-1}\leq \lambda_{\mathrm{prod}}^2,
\qquad
R^{-1/2}\leq \lambda_{\mathrm{prod}}\leq R^{-1/n}.
\]

\begin{definition}[Admissibility]\label{def41}
    Let $\varnothing\neq A\subset\{1,\ldots,n-1\}$. We call
$(\boldsymbol\lambda,A)$ admissible at scale $R$ if \eqref{eq:n-scale-normalization} holds and
\begin{equation}\label{eq:n-admissibility}
i\notin A\quad\Longrightarrow\quad\lambda_i=1.
\end{equation}
\end{definition}

Set $\lambda_n=1$, and index the derivative scales by their order:
\begin{equation}\label{eq:n-derivative-scales}
L_k:=\delta^{-1}
\prod_{i=1}^{n-k}\lambda_i^{\,n+1-k-i}
\quad(0\leq k\leq n),
\end{equation}
where an empty product is understood to be $1$. 
The quantities $L_k$ represent width scales in the derivative coordinates. Notably, $L_0=Q^{-1}, L_n=\de^{-1}$. One also has the monotonicity
\[ L_0\le L_1\le \dots\le L_n, \]
and log-concavity
\begin{equation}\label{convexity}
    L_k L_{k+2}\le L_{k+1}^2. 
\end{equation} 

\begin{remark}
    {\rm
    We remark that $L_k=\de^{-1}\rho_{n-k}$ where $\rho_{n-k}$ is given by \eqref{defrho}. The reason we use this index-reversed notation is to make the following definition of $\Pi_{\tau,\boldsymbol\lambda,A}$ has consistent index in $P_\xi^{(\cdot)}$ and $L_{(\cdot)}$.
    }
\end{remark}

Let $0<\beta<1/(10n)$. The parameter $\beta$ will be chosen in terms of
$\epsilon$ at the end of the reduction. 
Let $\mathcal I(\lambda_{\mathrm{prod}})$ partition $I$ into intervals of length $\lambda_{\mathrm{prod}}$,
and denote their centers by $t_\tau$. Define
\begin{equation}\label{eq:n-plank}
\Pi_{\tau,\boldsymbol\lambda,A}:=
\left\{\xi\in\mathbb R^{n+1}:
\begin{array}{l}
|\xi|\sim R,\\
|P_\xi^{(k)}(t_\tau)|\lesssim R^\beta QL_k
       \quad(k=0,1),\\
|P_\xi^{(k)}(t_\tau)|\lesssim QL_k
       \quad(2\leq k\leq n),\\
|P_\xi^{(k)}(t_\tau)|\gtrsim QL_k
       \quad(n+1-k\in A)
\end{array}\right\}.
\end{equation}

Morally speaking, $\Pi_{\tau,\boldsymbol\lambda,A}$ is a box of dimensions
\[ Q(L_0 \times L_1\times\dots\times L_n )=1\times QL_1\times \dots\times Q L_{n-1}\times R, \]
with coordinate axis given by the osculating frame of $\wt \gamma$ at point $t_\tau$.

Our main ingredient is a pointwise decay estimate for
$\widehat\mu_{\omega,\chi}$, obtained from the Fourier decomposition of
Gan--Maldague--Oh~\cite[Definition~3.17, Lemma~3.20 and Theorem~3.22]{GMO2025}.
Their decomposition is formulated in Frenet coordinates, so we must
express its bounds in terms of the derivatives $P_\xi^{(k)}$ used in
\eqref{eq:n-plank}. We carry out this comparison and prove the following
two propositions in Appendix~\ref{app:derivative-fourier}.

\begin{proposition}
\label{prop:n-fourier-input}
For every
$0<\beta<1/(10n)$ and every $L>0$,
\begin{equation}\label{eq:n-fourier-decomposition}
|\widehat\mu_{\omega,\chi}(\xi)|
\lesssim_{\beta,L,\gamma,\omega,\chi}
R^{O(\beta)}
\sum_{(\boldsymbol\lambda,A)\ \mathrm{admissible~at~}R}
\lambda_{\mathrm{prod}}\sum_{\tau\in\mathcal I(\lambda_{\mathrm{prod}})}
\mathbf1_{\Pi_{\tau,\boldsymbol\lambda,A}}(\xi)
+R^{-L},\qquad |\xi|\sim R.
\end{equation}
\end{proposition}

\begin{proposition}
\label{prop:n-overlap-input}
For every  $(\boldsymbol\lambda,A)$ admissible at $R$ and $0<\beta<1/(10n)$,
\begin{equation}\label{eq:n-plank-overlap}
\sum_{\tau\in\mathcal I(\lambda_{\mathrm{prod}})}
\mathbf1_{\Pi_{\tau,\boldsymbol\lambda,A}}(\xi)
\lesssim_{\beta,\gamma}R^{O(\beta)},\qquad |\xi|\sim R.
\end{equation}
\end{proposition}

We are now ready to state the lattice points counting.

\begin{proposition}[Lattice points counting]
\label{prop:n-derivative-counting}
Let $n\ge3$, $Q\geq1$, $Q^{-1}\le\delta\le1$ and
$0<\beta<1/(10n)$. For every admissible $(\boldsymbol\lambda,A)$, define
\[
M_{\boldsymbol\lambda,A}:=
\Biggl\{\mathbf j\in\mathbb Z^{n+1}:
\begin{array}{l}
|\mathbf j|\sim\delta^{-1},\quad\mathbf j_{\geq2}\neq0;
\quad\text{there exists }\tau\in\mathcal I(\lambda_{\mathrm{prod}}) \text{~such that}:\\[1mm]
|P_{\mathbf j}^{(k)}(t_\tau)|\lesssim R^\beta L_k
 \quad(k=0,1),\\[1mm]
|P_{\mathbf j}^{(k)}(t_\tau)|\lesssim L_k
 \quad(2\leq k\leq n),\\[1mm]
|P_{\mathbf j}^{(k)}(t_\tau)|\gtrsim L_k
 \quad(n+1-k\in A)
\end{array}\Biggr\}.
\]
Then, for every $\epsilon>0$,
\begin{equation}\label{eq:n-stable-counting-target}
\#M_{\boldsymbol\lambda,A}
\lesssim_{\epsilon,\beta,\gamma}
Q^\epsilon R^{O(\beta)}\lambda_{\mathrm{prod}}^{-1}
E_n(Q^{-1},\delta^{-1}).
\end{equation}
\end{proposition}

For $|\mathbf j|\sim\delta^{-1}$ and $\mathbf j_{\geq2}\neq0$,
the identity $P_{Q\mathbf j}^{(k)}=QP_{\mathbf j}^{(k)}$ gives
\[
\mathbf j\in M_{\boldsymbol\lambda,A}
\quad\Longleftrightarrow\quad
Q\mathbf j\in\bigcup_{\tau\in\mathcal I(\lambda_{\mathrm{prod}})}
\Pi_{\tau,\boldsymbol\lambda,A}.
\]
The factors $R^\beta$ in orders $0$ and $1$ must be retained when proving
Proposition~\ref{prop:n-derivative-counting}.

\begin{proof}[Proof of Proposition~\ref{prop:conical-n} assuming Proposition~\ref{prop:n-derivative-counting}]
Assume Propositions~\ref{prop:n-fourier-input},
\ref{prop:n-overlap-input} and~\ref{prop:n-derivative-counting}.
Let $\epsilon>0$ and choose $0<\beta<1/(10n)$ sufficiently small
in terms of $\epsilon$ and $n$.
Summing \eqref{eq:n-fourier-decomposition} over the frequency shell and
then applying \eqref{eq:n-plank-overlap} gives
\begin{align*}
\sum_{\substack{\mathbf j\in\mathbb Z^{n+1}\\
 |\mathbf j|\sim\delta^{-1},\ \mathbf j_{\geq2}\neq0}}
|\widehat\mu_{\omega,\chi}(Q\mathbf j)|
&\lesssim R^{O(\beta)}
\sum_{(\boldsymbol\lambda,A)\ \mathrm{admissible}}
\lambda_{\mathrm{prod}}\sum_{\mathbf j\in M_{\boldsymbol\lambda,A}}
\sum_{\tau\in\mathcal I(\lambda_{\mathrm{prod}})}
\mathbf1_{\Pi_{\tau,\boldsymbol\lambda,A}}(Q\mathbf j)
+\delta^{-(n+1)}R^{-L}\\
&\lesssim R^{O(\beta)}
\sum_{(\boldsymbol\lambda,A)\ \mathrm{admissible}}
\lambda_{\mathrm{prod}}\,\#M_{\boldsymbol\lambda,A}
+\delta^{-(n+1)}R^{-L}.
\end{align*}
Apply Proposition~\ref{prop:n-derivative-counting} with sufficiently small $\beta$ and $\epsilon$. There are $O_n((\log(2R))^{n-2})$ admissible pairs and $R^{O(\beta)} \lesssim Q^{O(\epsilon)}$, so the
last expression is bounded by
\[
Q^{O(\epsilon)}E_n(Q^{-1},\delta^{-1})
+\delta^{-(n+1)}R^{-L}.
\]
Since $\delta^{-1}\le Q$, $R\le Q^2$, and
\[
 E_n(Q^{-1},\delta^{-1})\geq
 (Q\delta^{n-1})^{-1/n}\geq Q^{-1/n},
\]
the rapidly decaying remainder is absorbed by choosing $L$ sufficiently large.
The same lower bound absorbs the nonstationary remainder in
Subsection~\ref{sec:n-nonstationary}. Choosing the exponents
in the preceding estimates sufficiently small proves
Proposition~\ref{prop:conical-n}.
\end{proof}

\section{Dual counting estimate}

\label{sec:n-sufficient-dual-counting}

We record the following $n$-dimensional version of the normalization in
\cite[Lemma~4.1]{CSST2026}. The proof is identical, so we omit the proof.

\begin{lemma}[Reduction to a perturbation of the moment curve]
\label{lem:n-integral-affine-normalization}
Let $n\ge3$ and let $\gamma$ be smooth on a neighborhood of a compact
interval $I$, with $\gamma_1(t)=t$ and
\[
\det\bigl(\gamma'(t),\ldots,\gamma^{(n)}(t)\bigr)\ne0
\qquad(t\in I).
\]
For every $0<\epsilon_0\le1$, there are smooth nondegenerate curves
\[
\mathcal C_d=\gamma_d([-1,1])\subset\mathbb R^n
\quad(1\le d\le M),\qquad
M\lesssim_\gamma\epsilon_0^{-1},
\]
and a constant $1\le A\lesssim_\gamma\epsilon_0^{-n}$ with the following
properties. They satisfy
\begin{equation}\label{eq:n-integral-affine-model}
\left\|\gamma_d(t)-\gamma_d(0)
-\left(t,\frac{t^2}{2!},\ldots,\frac{t^n}{n!}\right)\right\|_{C^{n+1}([-2,2])}
\le\epsilon_0.
\end{equation}
Writing $\mathcal C=\gamma(I)$, we have
\begin{equation}\label{eq:n-integral-affine-count}
\mathcal R_{\mathcal C}(\delta,Q)
\le\sum_{d=1}^M\mathcal R_{\mathcal C_d}(A\delta,Q)
\qquad(Q\ge1,\ 0<\delta<1/2).
\end{equation}
\end{lemma}

By the above lemma, we may assume that our nondegenerate curve $\gamma$ always satisfies
\begin{equation}
\left\|\gamma(t)-\gamma(0)
-\left(t,\frac{t^2}{2!},\ldots,\frac{t^n}{n!}\right)\right\|_{C^{n+1}([-2,2])}
\le\epsilon_0.
\end{equation}
where $\epsilon_0$ is sufficiently small.

\subsection{Derivative adapted to a curve}

We recall that
\[
P_{\mathbf j}(t)=\sum_{\nu=0}^n j_\nu \gamma_\nu(t)
=j_0+j_1t+\sum_{\nu=2}^n j_\nu \gamma_\nu(t).
\]
The first ordinary derivative eliminates $j_0$, and the second also
eliminates $j_1$. At order $k\ge3$, however, the coefficients
$j_2,\ldots,j_{k-1}$ generally remain. We therefore use derivatives
adapted to the curve.

We use the convention
\[
\operatorname{Wr}(f_0,\ldots,f_k)(t)
:=\det\bigl(f_j^{(i)}(t)\bigr)_{0\le i,j\le k}.
\]
For the coordinate functions, set
\[
W_k(t):=\operatorname{Wr}(\gamma_0,\ldots,\gamma_k)(t)
\quad(0\le k\le n),\qquad W_{-1}:=1.
\]

After localizing the parameter and permuting $\gamma_2,\ldots,\gamma_n$
and the corresponding integer coordinates, choose compact intervals
$I\Subset I^*$ so that, on a neighborhood of $I^*$,
\[
W_k(t)\ne0
\qquad(0\le k\le n).
\]
Define successively
\begin{equation}\label{eq:n-normalized-derivatives}
\mathcal D_0f=f,\qquad
\mathcal D_kf=\frac{(\mathcal D_{k-1}f)'}{(\mathcal D_{k-1}\gamma_k)'}
\quad(1\le k\le n).
\end{equation}
Note that these $\mathcal D_k$ are linear operators.

The functions $(\mathcal D_{k-1}\gamma_k)'$ are smooth, and
\begin{equation}
    \bigl|(\mathcal D_{k-1}\gamma_k)'(t)\bigr|\asymp_\gamma1
\qquad(t\in I^*,\ 1\le k\le n).\end{equation}
With this notation,
\[
W_k(t):=\operatorname{Wr}(\gamma_0,\ldots,\gamma_k)(t),
\]
one can prove by induction that
\[
\mathcal D_k f
=\frac{\operatorname{Wr}(\gamma_0,\ldots,\gamma_{k-1},f)}{W_k}.
\]
Applying this with $f=\gamma_k$ and taking derivative, one has
\[
(\mathcal D_{k-1}\gamma_k)'
=\frac{W_{k-2}W_k}{W_{k-1}^2}.
\]
Since $\gamma-\gamma(0)$ is $\e_0$-close to the moment curve in the $C^{n+1}$-norm, one has $|W_r(t)|\asymp_\gamma 1$ there for $0\le r\le n$. The preceding identity therefore yields the claimed comparability.

By definition, the normalized derivatives satisfy
\[
\mathcal D_k\gamma_j=0\quad(j<k),\qquad \mathcal D_k\gamma_k=1.
\]

\begin{proposition}\label{prop:n-normalized-derivative-expansion}
For every smooth $f$ on a neighborhood of $I^*$ and $1\le k\le n$,
\begin{equation}\label{eq:n-normalized-derivative-expansion}
\mathcal D_kf(t)
=\frac{f^{(k)}(t)}{\displaystyle\prod_{\nu=1}^k(\mathcal D_{\nu-1}\gamma_\nu)'(t)}
+\sum_{m=0}^{k-1}c_{k,m}(t)f^{(m)}(t).
\end{equation}
The coefficients $c_{k,m}$ are smooth functions depending only on
the curve, and, uniformly on $I^*$,
\[
|c_{k,m}(t)|\lesssim_\gamma1,\qquad
c_{k,0}=c_{k,1}=0\quad(k\ge2).
\]
\end{proposition}

\begin{proof}
The formula follows by induction on $k$ using
\eqref{eq:n-normalized-derivatives}. Differentiating the formula
at order $k-1$ and dividing by $(\mathcal D_{k-1}\gamma_k)'$ gives leading coefficient
\[
\left(\prod_{\nu=1}^k(\mathcal D_{\nu-1}\gamma_\nu)'\right)^{-1}.
\]
The remaining coefficients are smooth combinations of the functions
$(\mathcal D_{\nu-1}\gamma_\nu)'$ for $1\le\nu\le k$,
their derivatives, and their reciprocals.
The bounds follow from smoothness and nonvanishing on $I^*$.
For $k\ge2$, applying the expansion to $1$ and $t$ and using
$\mathcal D_k1=\mathcal D_kt=0$ gives $c_{k,0}=c_{k,1}=0$.
\end{proof}

Consequently,
\begin{equation}\label{eq:n-normalized-tail}
\mathcal D_kP_{\mathbf j}(t)
=j_k+\sum_{\nu=k+1}^n j_\nu\mathcal D_k\gamma_\nu(t).
\end{equation}
Thus $\mathcal D_kP_{\mathbf j}$ depends only on
$\mathbf j_{\ge k}=(j_k,\ldots,j_n)$, and
\[
(\mathcal D_kP_{\mathbf j})'=(\mathcal D_k\gamma_{k+1})'\mathcal D_{k+1}P_{\mathbf j}\quad(0\le k\le n-1),
\qquad \mathcal D_nP_{\mathbf j}=j_n.
\]

The normalized derivatives also produce nondegenerate projected curves.

\begin{proposition}\label{prop:n-projected-curves}
Let $r,m$ be integers with $0\le r\le n-2$ and $2\le m\le n-r$.
Define
\[
\mathcal C_{r,m}
:=\{(\mathcal D_r\gamma_{r+1}(t),\ldots,
       \mathcal D_r\gamma_{r+m}(t)):t\in I^*\}\subset\mathbb R^m.
\]
Then $\mathcal C_{r,m}$ is nondegenerate, and
$s=\mathcal D_r\gamma_{r+1}(t)$ gives a smooth reparametrization.
The same statement holds with $I^*$ replaced by any compact interval
in the neighborhood where $\operatorname{Wr}(\gamma_0,\ldots,\gamma_k)\ne0$
for $0\le k\le n$.
\end{proposition}

\begin{proof}
Let $t_0\in I^*$. To prove nondegeneracy, suppose that real numbers
$a_{r+1},\ldots,a_{r+m}$ satisfy
\[
\sum_{\nu=r+1}^{r+m}a_\nu
\frac{d^k}{dt^k}(\mathcal D_r\gamma_\nu)(t_0)=0
\qquad(1\le k\le m).
\]
We show that all these coefficients are zero.
By the recursive definition of the normalized derivatives,
$\mathcal D_{r+k}\gamma_\nu$ is a linear combination of the first
$k$ ordinary derivatives of $\mathcal D_r\gamma_\nu$,
with coefficients independent of $\nu$.
Indeed, each step differentiates the preceding expression and divides
by the nonzero function $(\mathcal D_{r+k-1}\gamma_{r+k})'$.
Consequently,
\[
\sum_{\nu=r+1}^{r+m}a_\nu\mathcal D_{r+k}\gamma_\nu(t_0)=0
\qquad(1\le k\le m).
\]
Since $\mathcal D_{r+k}\gamma_\nu=0$ for $\nu<r+k$ and
$\mathcal D_{r+k}\gamma_{r+k}=1$, these equations become
\[
0=a_{r+k}+\sum_{\nu=r+k+1}^{r+m}
 a_\nu\mathcal D_{r+k}\gamma_\nu(t_0)
\qquad(1\le k\le m).
\]
Taking $k=m$ gives $a_{r+m}=0$; then $k=m-1$ gives
$a_{r+m-1}=0$. Continuing down to $k=1$ gives
$a_{r+1}=\cdots=a_{r+m}=0$.
Thus the first $m$ derivative vectors are linearly independent at
$t_0$, proving nondegeneracy.

Finally, $ds/dt=(\mathcal D_r\gamma_{r+1})'\ne0$, so the change of
parameter is smooth and preserves nondegeneracy.
The argument applies throughout the neighborhood where the
normalizing denominators are nonzero, proving the last assertion.
\end{proof}

\bigskip

\subsection{Dual counting problem}

Throughout the counting argument, $Q^{-1}\le\delta\le1$.
We use $E_n$ from Section~\ref{sec:algebraic-computation}
with height $\delta^{-1}$ and terminal width $Q^{-1}$:
\begin{equation}\label{eq:n-full-range-budget}
\begin{aligned}
\lambda_{\mathrm{prod}}^{-1}E_n(Q^{-1},\delta^{-1})
&=\lambda_{\mathrm{prod}}^{-1}
\max_{0\le k\le n-1}
(Q\delta^{D_{n,k}})^{-1/(k+1)}.
\end{aligned}
\end{equation}
The factor multiplying $\lambda_{\mathrm{prod}}^{-1}$ is independent
of the scale tuple. It is at most one when $\delta\ge Q^{-1/D_n}$,
and we retain its actual value even when it is smaller than one.
It gives precisely the nonzero-frequency error terms in
Theorem~\ref{thm:main}:
\[
\delta^{n-1}Q^2 E_n(Q^{-1},\delta^{-1})
\le Q+
\sum_{k=1}^{n-1}\delta^{k^2/(k+1)}Q^{(2k+1)/(k+1)}.
\]

In definitions, $O(\beta)$ denotes a common positive constant multiple
of $\beta$, chosen sufficiently large in terms of $n$ and the curve.
In estimates, the constants implicit in $O(\beta)$ may vary.

Recall $L_k$ in \eqref{eq:n-derivative-scales}.
For $0\le r\le n-2$, define
\begin{equation}\label{eq:n-dual-scales}
\ell_r:=\frac{L_{r+1}}{L_{r+2}}
=\prod_{j=1}^{n-r-1}\lambda_j,
\qquad
\widetilde{L}_r:=\frac{L_{r+1}^2}{L_{r+2}}
=\delta^{-1}\prod_{j=1}^{n-r-1}\lambda_j^{n-r+1-j}.
\end{equation}
From \eqref{convexity}, one sees that $\widetilde{L}_r\ge L_r$.

For $r\in\mathcal R_{\boldsymbol\lambda}$ with $r\ge1$ (see the definition below), a displacement of
size $O_\gamma(\ell_r)$ reaches a zero of $\mathcal D_{r+1}P_{\mathbf j}$
and changes $\mathcal D_rP_{\mathbf j}$ by $O_\gamma(\widetilde{L}_r)$.
At $r=0$, the corresponding bounds are $O_\gamma(R^\beta\ell_0)$
and $O_\gamma(R^{O(\beta)}\widetilde{L}_0)$.
Recall that $\lambda_n=1$. We will use
\[
\widetilde{L}_0=Q^{-1},\qquad
L_r=\widetilde{L}_r\lambda_{n-r}\le\widetilde{L}_r,\qquad
L_0\le L_1\le\cdots\le L_n.
\]

Choose $c_*>0$ sufficiently small in terms of the curve, and put
\[
\begin{gathered}
A_*:=\{1\le i\le n-1:\lambda_i\le c_*\},\\
\mathcal R_{\boldsymbol\lambda}
:=\{0\le r\le n-2:n-r-1\in A_*\ \text{when}\ r\ge 1;\ R^\beta\lambda_{n-1}\le c_*\ \text{when}\ r=0\}.
\end{gathered}
\] 
Admissibility (recalling \eqref{eq:n-admissibility}) implies $A_*\subset A$. The additional condition at $r=0$
allows the displacement of size $O_\gamma(R^\beta\ell_0)$ used below.
If $\mathcal R_{\boldsymbol\lambda}=\varnothing$, then
$\lambda_i>c_*$ for $1\le i\le n-2$ and
$\lambda_{n-1}>c_*R^{-\beta}$. Hence the scale normalization gives
\[
R=\prod_{i=1}^{n-1}\lambda_i^{-(n+1-i)}
\lesssim_\gamma R^{O(\beta)}.
\]
Taking $\beta>0$ sufficiently small, the preceding inequality bounds
$R$, and hence both $Q$ and $\delta^{-1}$, in terms of the curve and $n$.
The desired counting estimate is then immediate, since
$E_n(Q^{-1},\delta^{-1})\geq Q^{-1/n}\gtrsim_\gamma1$.
For each $r\in\mathcal R_{\boldsymbol\lambda}$, define
\begin{equation}\label{eq:n-sufficient-dual-set}
\mathcal T^{(r)}_{\delta,\boldsymbol\lambda}(I^*)
:=\left\{\mathbf j_{\ge r}\in\mathbb Z^{n+1-r}:
\begin{array}{l}
|\mathbf j_{\ge r}|\lesssim\delta^{-1};\quad
\text{there exists }t_*\in I^* \text{ such that}:\\[1mm]
\mathcal D_{r+1}P_{\mathbf j}(t_*)=0,\quad
|\mathcal D_rP_{\mathbf j}(t_*)|\lesssim R^{O(\beta)}\widetilde{L}_r,\\[1mm]
|\mathcal D_kP_{\mathbf j}(t_*)|\lesssim L_k
\quad(r+2\le k\le n),\\[1mm]
|\mathcal D_kP_{\mathbf j}(t_*)|\gtrsim L_k
\quad(r+2\le k\le n,\ n+1-k\in A_*)
\end{array}\right\}.
\end{equation}
For each $t_*$, the parameters $L_k$, $r+2\le k\le n$,
specify the width scales in the normalized derivative coordinates
$\mathcal D_kP_{\mathbf j}(t_*)$.
For $\mathbf j_{\ge r}\in
\mathcal T^{(r)}_{\delta,\boldsymbol\lambda}(I^*)$,
we call any point $t_*\in I^*$ satisfying all the conditions
in \eqref{eq:n-sufficient-dual-set} a \emph{witness} for
$\mathbf j_{\ge r}$.
In particular, since $r\in\mathcal R_{\boldsymbol \lambda}$ implies $n-r-1\in A_*$, one has $|\mathcal D_{r+2}P_{\mathbf j}(t_*)|\asymp_\gamma L_{r+2}$
at every such witness $t_*$. Let us state the  counting estimate for  \eqref{eq:n-sufficient-dual-set}. 

\begin{proposition}\label{prop:n-sufficient-dual}
Let $n\ge3$, $Q\ge1$, $Q^{-1}\le\delta\le1$, and
$0<\beta<1/(10n)$, with $R=Q\delta^{-1}$.
For every admissible $(\boldsymbol\lambda,A)$ with
$\mathcal R_{\boldsymbol\lambda}\ne\varnothing$ and every $\epsilon>0$,
\begin{equation}\label{eq:n-sufficient-dual-input}
\begin{aligned}
\min_{r\in\mathcal R_{\boldsymbol\lambda}}
\left\{\left(\prod_{s=0}^{r-1}(1+\widetilde{L}_s)\right)
\#\mathcal T^{(r)}_{\delta,\boldsymbol\lambda}(I^*)\right\}\lesssim_{\epsilon,\beta,\gamma}
Q^\epsilon R^{O(\beta)}\lambda_{\mathrm{prod}}^{-1}E_n(Q^{-1},\delta^{-1}).
\end{aligned}
\end{equation}
\end{proposition}
For each $r\in\mathcal R_{\boldsymbol\lambda}$, the expression inside the minimum in
\eqref{eq:n-sufficient-dual-input} has two factors:
\[
\underbrace{\prod_{s=0}^{r-1}(1+\widetilde{L}_s)}_{
\substack{\text{cost of recovering}\\j_0,\ldots,j_{r-1}}}
\quad
\underbrace{\#\mathcal T^{(r)}_{\delta,\boldsymbol\lambda}(I^*)}_{
\substack{\text{number of choices for}\\(j_r,\ldots,j_n)}}.
\]
Every vector counted in Proposition~\ref{prop:n-derivative-counting}
projects to a vector in $\mathcal T^{(r)}_{\delta,\boldsymbol\lambda}(I^*)$.
For each such projection, the first factor bounds the number of choices
for the omitted coefficients, up to $R^{O(\beta)}$.

Let us explain why the sidelength $L_s$ is thickened to $\widetilde{L}_s$.
\begin{equation}
    |\mathcal D_rP_{\mathbf j}(t_*) - \mathcal D_rP_{\mathbf j}(t_\tau)| \lesssim |t_*-t_\tau|\sup |\mathcal D_{r+1}P_{\mathbf j}|.
\end{equation}
After some computations, we can see that $|t_*-t_\tau| \lesssim \ell_r$ and $\sup |\mathcal D_{r+1}P_{\mathbf j}| \lesssim L_{r+1}$. So we get the bound $\ell_r L_{r+1}=\widetilde{L}_r$.
For $s<r$, analogous variation estimates on the intervals of possible
witnesses give the recovery factors $1+\widetilde{L}_s$.

We prove these claims in the next subsection and postpone the proof of
Proposition~\ref{prop:n-sufficient-dual}.

\bigskip

\subsection{Proof of Proposition~\ref{prop:n-derivative-counting} assuming Proposition \ref{prop:n-sufficient-dual}}

Let $M_{\boldsymbol\lambda,A}$ be the set defined in
Proposition~\ref{prop:n-derivative-counting}. We prove, for each $r\in\mathcal R_{\boldsymbol\lambda}$, that
\begin{equation}\label{eq:n-dual-recovery}
\#M_{\boldsymbol\lambda,A}
\lesssim_\gamma R^{O(\beta)}
\left(\prod_{s=0}^{r-1}(1+\widetilde{L}_s)\right)
\#\mathcal T^{(r)}_{\delta,\boldsymbol\lambda}(I^*).
\end{equation}
Taking the minimum over $r$ and applying
Proposition~\ref{prop:n-sufficient-dual} will give the result.
Let us briefly explain how to prove \eqref{eq:n-dual-recovery}. We first prove that
\begin{equation}
    \mathbf j\in M_{\boldsymbol\lambda,A} \Longrightarrow \mathbf j_{\ge r}\in\mathcal T^{(r)}_{\delta,\boldsymbol\lambda}(I^*).
\end{equation}
And then we prove that for given $\mathbf j_{\ge r}\in \mathcal T^{(r)}_{\delta,\boldsymbol\lambda}(I^*)$
the number of possible choices of $\mathbf j_{\le r-1}$ such that $\mathbf j=(\mathbf j_{\le r-1},\mathbf j_{\ge r})\in  M_{\boldsymbol\lambda,A}$ is bounded by
\[
O\left(R^{O(\beta)}\prod_{s=0}^{r-1}(1+\widetilde{L}_s)\right).
\]
Combining these two gives the bound \eqref{eq:n-dual-recovery}.

\subsubsection*{\textbf{Counting $\mathbf{j_{\ge r}}$.}}

Let $\mathbf j\in M_{\boldsymbol\lambda,A}$, and let $t_\tau$
be a witness for this membership, that is, a point satisfying
the defining conditions of $M_{\boldsymbol\lambda,A}$.
Take $r\in\mathcal R_{\boldsymbol\lambda}$.
We show that
$\mathbf j_{\ge r}\in\mathcal T^{(r)}_{\delta,\boldsymbol\lambda}(I^*)$.

The collection $M_{\boldsymbol\lambda,A}$ is defined by using the ordinary differentiation. We'll first relate this to the differentiation $\mathcal{D}_k$. 
First, $\mathcal D_0P_{\mathbf j}=P_{\mathbf j}$ and
$\mathcal D_1P_{\mathbf j}=P_{\mathbf j}'$. For $k\ge2$, the coefficients
$c_{k,0}$ and $c_{k,1}$ vanish, so
\eqref{eq:n-normalized-derivative-expansion} gives
\[
|\mathcal D_kP_{\mathbf j}(t_\tau)|
\lesssim_\gamma\sum_{m=2}^k|P_{\mathbf j}^{(m)}(t_\tau)|
\lesssim_\gamma L_k.
\]
Consequently,
\begin{equation}\label{DkPbound}
    |\mathcal D_kP_{\mathbf j}(t_\tau)|\lesssim_\gamma
\begin{cases}
R^\beta L_k,&k=0,1,\\
L_k,&2\le k\le n.
\end{cases}
\end{equation}

If $n+1-k\in A_*$, then
\[
\frac{L_{k-1}}{L_k}
=\prod_{i=1}^{n+1-k}\lambda_i\le c_*.
\]
The leading term in \eqref{eq:n-normalized-derivative-expansion}
has size $\gtrsim_\gamma L_k$, whereas its remaining terms satisfy
\[
\sum_{m=2}^{k-1}|c_{k,m}(t_\tau)P_{\mathbf j}^{(m)}(t_\tau)|
\lesssim_\gamma L_{k-1}\le c_*L_k.
\]
For $k=2$ this sum is empty. Taking $c_*$ sufficiently small gives
\[
|\mathcal D_kP_{\mathbf j}(t_\tau)|\gtrsim_\gamma L_k
\qquad(n+1-k\in A_*).
\]
In particular, 
\begin{equation}\label{Dr+2Pbound}
    |\mathcal D_{r+2}P_{\mathbf j}(t_\tau)|
\asymp_\gamma L_{r+2}.
\end{equation}

We next find a zero of $\mathcal D_{r+1}P_{\mathbf j}$ near $t_\tau$.
For $C_1>0$ sufficiently large in terms of the curve, set
\[
J:=
\begin{cases}
[t_\tau-C_1\ell_r,t_\tau+C_1\ell_r],&r\ge1,\\
[t_\tau-C_1R^\beta\ell_0,t_\tau+C_1R^\beta\ell_0],&r=0.
\end{cases}
\]
The condition $r\in\mathcal R_{\boldsymbol\lambda}$ gives $\ell_r\le c_*$ for $r\ge1$ and
$R^\beta\ell_0\le R^\beta\lambda_{n-1}\le c_*$ for $r=0$.
Thus $J\subset I^*$ after choosing $c_*$ sufficiently small after $C_1$.
For $r\ge1$ and $r+2\le k\le n-1$, we have
\begin{equation}\label{lrineq}
    \ell_r\frac{L_{k+1}}{L_k}
=\prod_{i=n-k+1}^{n-r-1}\lambda_i
\le\lambda_{n-r-1}\le c_*.
\end{equation}

For $r=0$ and $2\le k\le n-1$, the corresponding bound is
\begin{equation}\label{l0ineq}
    R^\beta\ell_0\frac{L_{k+1}}{L_k}
=R^\beta\prod_{i=n-k+1}^{n-1}\lambda_i
\le R^\beta\lambda_{n-1}\le c_*.
\end{equation}

Starting with $\mathcal D_nP_{\mathbf j}=j_n$ and integrating
successively down to order $r+2$, we prove the following result.
\begin{lemma}
For $t\in J,\ r+2\le k\le n$, we have
\begin{equation}\label{supDkP}
    \sup_{t\in J}|\mathcal D_kP_{\mathbf j}(t)|\lesssim_\gamma L_k,
\end{equation}
\begin{equation}\label{DiffDkP}
    |\mathcal D_kP_{\mathbf j}(t)-\mathcal D_kP_{\mathbf j}(t_\tau)|
\lesssim_\gamma C_1c_*L_k.
\end{equation}
\end{lemma}
\begin{proof}
    We prove the lemma by induction on $k$. When $k=n$, \eqref{supDkP} and \eqref{DiffDkP} hold since $\mathcal D_n P_{\mathbf j}\equiv j_n$. Assuming we have proved for $k+1$, we consider the case at $k$. Note that (recalling \eqref{eq:n-normalized-derivatives})
\[ |\mathcal D_kP_{\mathbf j}(t)-\mathcal D_kP_{\mathbf j}(t_\tau)|\le |t-t_\tau|\cdot\sup_{u\in J}
| (\mathcal D_k P_{\mathbf j})'(u)|=
|t-t_\tau|\cdot\sup_{u\in J}
| (\mathcal D_k\gamma_{k+1})'(u)\mathcal D_{k+1}P_{\mathbf j}(u)|.
\]
Noting that $|(\mathcal D_k\gamma_{k+1})'(u)|\asymp 1$ and by induction hypothesis that $|\mathcal D_{k+1}P_{\mathbf j}(u)|\lesssim L_{k+1}$, we have
\[ |\mathcal D_kP_{\mathbf j}(t)-\mathcal D_kP_{\mathbf j}(t_\tau)|\lesssim C_1\ell_r L_{k+1}\ \ \ \text{when }r\ge 1; \]
\[ |\mathcal D_kP_{\mathbf j}(t)-\mathcal D_kP_{\mathbf j}(t_\tau)|\lesssim C_1R^\beta\ell_0 L_{k+1}\ \ \ \text{when }r=0. \]
By \eqref{lrineq} and \eqref{l0ineq}, we have
\[ |\mathcal D_kP_{\mathbf j}(t)-\mathcal D_kP_{\mathbf j}(t_\tau)|\lesssim C_1 c_* L_k. \]
This proves \eqref{DiffDkP} at case $k$.

Then using \eqref{DkPbound} and
choosing $c_*$ sufficiently
small so that $C_1c_*\ll 1$, one has
\[ |\mathcal D_k P_{\mathbf j}(t)|\le |\mathcal D_k P_{\mathbf j}(t_\tau)|+ |\mathcal D_kP_{\mathbf j}(t)-\mathcal D_kP_{\mathbf j}(t_\tau)| \lesssim L_k.\]
This proves \eqref{supDkP} at case $k$.
\end{proof}

Now that $c_*$ is chosen so that $C_1c_*\ll 1$, we see that $\mathcal D_{r+2}P_{\mathbf j}$ has constant sign on $J$ and
\[
|\mathcal D_{r+2}P_{\mathbf j}(t)|\asymp_\gamma L_{r+2}
\qquad(t\in J).
\]
It follows from
\[
(\mathcal D_{r+1}P_{\mathbf j})'
=(\mathcal D_{r+1}\gamma_{r+2})'\mathcal D_{r+2}P_{\mathbf j}
\]
that $\mathcal D_{r+1}P_{\mathbf j}$ is monotone on $J$, with derivative
of absolute value $\asymp_\gamma L_{r+2}$.
Its initial absolute value (see \eqref{DkPbound}) is $O_\gamma(L_{r+1})$ for $r\ge1$
and $O_\gamma(R^\beta L_1)$ for $r=0$.
Since $\ell_rL_{r+2}=L_{r+1}$, taking $C_1$ sufficiently
large makes its values on the endpoints of $J$ have opposite signs. Thus there exists
$t_*\in J$ such that
\[
\mathcal D_{r+1}P_{\mathbf j}(t_*)=0,
\qquad
|t_*-t_\tau|\lesssim_\gamma
\begin{cases}\ell_r,&r\ge1,\\R^\beta\ell_0,&r=0.\end{cases}
\]
Between $t_\tau$ and $t_*$, monotonicity bounds
$|\mathcal D_{r+1}P_{\mathbf j}|$ by its initial absolute value.
Integrating the identity for $(\mathcal D_rP_{\mathbf j})'$ therefore gives
\[
|\mathcal D_rP_{\mathbf j}(t_*)-\mathcal D_rP_{\mathbf j}(t_\tau)|
\lesssim_\gamma
\begin{cases}
\ell_rL_{r+1}=\widetilde{L}_r,&r\ge1,\\
R^{O(\beta)}\ell_0L_1=R^{O(\beta)}\widetilde{L}_0,&r=0.
\end{cases}
\]
(Recall $\widetilde{L}_r$ in \eqref{eq:n-dual-scales}.)
Using $L_r\le\widetilde{L}_r$ and the initial bounds, we conclude that
\[
|\mathcal D_rP_{\mathbf j}(t_*)|\lesssim_\gamma R^{O(\beta)}\widetilde{L}_r.
\]
All the higher-order conditions hold at $t_*$ because $t_*\in J$.
Hence $\mathbf j_{\ge r}\in\mathcal T^{(r)}_{\delta,\boldsymbol\lambda}(I^*)$.

\subsubsection*{\textbf{Counting $\mathbf j_{\le r-1}$.}}

We first establish a counting lemma for $0\le s\le n-2$,
and then apply it for $s<r$.
Let $(j_{s+1},\ldots,j_n)$ be given.
Our goal is to count the possible integers $j_s$ so that $(\dots,j_s,j_{s+1},\dots,j_n)\in M_{\boldsymbol\lambda,A}$.
By \eqref{eq:n-normalized-tail},
\[
\mathcal D_sP_{\mathbf j}(t)
=j_s+\sum_{\nu=s+1}^n j_\nu\mathcal D_s\gamma_\nu(t)
=:j_s+G_s(t).
\]
At the original witness $t_\tau$, we have
\[
|j_s+G_s(t_\tau)|\lesssim_\gamma
\begin{cases}R^\beta L_s,&s=0,1,\\ L_s,&s\ge2.\end{cases}
\]
Define
\[
E_s:=\left\{t\in I:
\begin{array}{l}
|\mathcal D_kP_{\mathbf j}(t)|\le CL_k
\quad(\max\{2,s+1\}\le k\le n),\\[1mm]
|\mathcal D_1P_{\mathbf j}(t)|\le CR^\beta L_1
\quad\text{if }s=0,\\[1mm]
|\mathcal D_kP_{\mathbf j}(t)|\ge cL_k
\quad(s+1\le k\le n,\ n+1-k\in A_*)
\end{array}\right\}.
\]
Here $C>0$ and $c>0$ depend only on the curve and are chosen
so that every original witness belongs to $E_s$.

\begin{lemma}\label{lem:n-fibre-counting}
Let $0\le s\le n-2$. The following conclusions hold.
They also hold with $I$ replaced by any compact interval on which
the same smoothness and nonvanishing assumptions hold, with constants
depending only on the interval length and the corresponding curve bounds.
    \begin{enumerate}
        \item $E_s$ is a union of  $O(1)$ connected components $J$.

        \item Each connected component satisfies
        \[
        |J|\lesssim_\gamma
        \begin{cases}R^\beta\ell_0,&s=0,\\\ell_s,&s\ge1.\end{cases}
        \]
        \item On each connected component,
        \[
        \operatorname{diam}(G_s(J))\lesssim_\gamma
        \begin{cases}R^{O(\beta)}\widetilde{L}_0,&s=0,\\\widetilde{L}_s,&s\ge1.\end{cases}
        \]
    \end{enumerate}
\end{lemma}
Let us quickly see how to prove Proposition \ref{prop:n-derivative-counting} assuming Proposition \ref{prop:n-sufficient-dual} and Lemma \ref{lem:n-fibre-counting}.
Let $0\le s<r$ and let $(j_{s+1},\dots,j_n)$ be given. We count the possible integers $j_s$ so that 
\begin{equation}\label{inMA}
    (\dots,j_s,j_{s+1},\dots,j_n)\in M_{\mathbf{\lambda},A}.
\end{equation}
The condition \eqref{inMA} implies that there exists $t\in E_s$ such that 
\[ |j_s+G_s(t)|=|\mathcal D_s P_{\mathbf j}(t)|\lesssim L_s\ \text{when}\ s\ge 2\ \ (\text{or}\ \lesssim R^\beta L_s\ \text{when}\ s=0,1 ). \]
By Lemma \ref{lem:n-fibre-counting} and the initial bounds for
$|j_s+G_s(t_\tau)|$, we see that the number of choices for $j_s$ is at most
\[
\begin{cases}
O_\gamma\bigl(R^{O(\beta)}(1+\widetilde{L}_0)\bigr),&s=0,\\
O_\gamma\bigl(R^\beta(1+\widetilde{L}_1)\bigr),&s=1,\\
O_\gamma(1+\widetilde{L}_s),&s\ge2.
\end{cases}
\]
Here we used $L_s\le\widetilde{L}_s$ and $R\ge1$.
Applying this successively for $s=r-1,\ldots,0$ gives
\eqref{eq:n-dual-recovery}; the factor $R^{3\beta}$ also covers $r=0,1$.
Taking the minimum over $r\in\mathcal R_{\boldsymbol\lambda}$ and using
Proposition~\ref{prop:n-sufficient-dual} proves
Proposition~\ref{prop:n-derivative-counting}.

\bigskip

\subsection{Proof of Lemma \ref{lem:n-fibre-counting}}
In the rest of the section, we prove Lemma \ref{lem:n-fibre-counting}.

We first prove item 1. For each
$0\leq k\leq n$ and $a\in\mathbb R$, the function
$\mathcal D_kP_{\mathbf j}-a$ either vanishes identically or has at most
$n-k$ distinct zeros on $I$, as we now show.

We prove this by descending induction on $k$.
For $k=n$, the assertion follows from $\mathcal D_nP_{\mathbf j}=j_n$.
Suppose the assertion holds for $k+1$. Recall that
\[
(\mathcal D_kP_{\mathbf j}-a)'(t)
=(\mathcal D_k\gamma_{k+1})'(t)\mathcal D_{k+1}P_{\mathbf j}(t).
\]
If $\mathcal D_{k+1}P_{\mathbf j}$ vanishes identically, then
$\mathcal D_kP_{\mathbf j}-a$ is constant, so the assertion follows.
Otherwise, the induction hypothesis and the fact that
$(\mathcal D_k\gamma_{k+1})'$ never vanishes imply that
$(\mathcal D_kP_{\mathbf j}-a)'$ has at most $n-k-1$ distinct zeros.
By Rolle's theorem, $\mathcal D_kP_{\mathbf j}-a$ has at most $n-k$
distinct zeros. This completes the induction.

Now consider the boundary equations
\[
\mathcal D_kP_{\mathbf j}(t)=\pm CL_k
\quad(s+1\leq k\leq n)
\]
and
\[
\mathcal D_kP_{\mathbf j}(t)=\pm cL_k
\quad(s+1\leq k\leq n,\ n+1-k\in A_*).
\]
When $s=0$, replace the upper boundary equations at order $1$ by
$\mathcal D_1P_{\mathbf j}(t)=\pm CR^\beta L_1$.
Let $Z$ be the union of the solution sets of those equations
that do not hold identically. By the preceding argument,
\[
\#Z\leq 4\sum_{k=s+1}^{n}(n-k)=O_n(1).
\]
On each connected component of $I\setminus Z$, every
inequality defining $E_s$ is either satisfied everywhere
or nowhere. Hence $E_s$ is a union of some of these intervals
and some points of $Z$. It therefore has at most
$2\#Z+1=O_n(1)$ connected components, proving item 1.

\medskip

We next prove item 2. Let $J$ be a connected component of $E_s$.
Suppose first that no index $s+2\le m\le n$ satisfies
$n+1-m\in A_*$. Then
\[
\lambda_i>c_*\quad(1\le i\le n-s-1),
\qquad
\ell_s=\prod_{i=1}^{n-s-1}\lambda_i\ge c_*^{n-s-1}.
\]
Since $I$ has bounded length, we have
$|J|\lesssim_\gamma\ell_s$.
This proves item 2 in this case, also for $s=0$ because $R^\beta\ge1$.

Otherwise, choose the smallest $s+2\le m\le n$ such that
$n+1-m\in A_*$. Put
\[
d:=m-s-1.
\]
By the definition of $E_s$, on $J$ we have
\[
|\mathcal D_{s+1}P_{\mathbf j}|\le
\begin{cases}CR^\beta L_1,&s=0,\\C L_{s+1},&s\ge1,\end{cases}
\qquad
|\mathcal D_mP_{\mathbf j}|\geq cL_m.
\]

Let us first obtain a lower bound for
$\partial_t^d(\mathcal D_{s+1}P_{\mathbf j})$.
Repeatedly using
$(\mathcal D_kP_{\mathbf j})'=(\mathcal D_k\gamma_{k+1})'\mathcal D_{k+1}P_{\mathbf j}$ gives
\[
\partial_t^d(\mathcal D_{s+1}P_{\mathbf j})
=
\left(\prod_{k=s+2}^{m}(\mathcal D_{k-1}\gamma_k)'\right)\mathcal D_mP_{\mathbf j}
+
\sum_{k=s+2}^{m-1}a_k(t)\mathcal D_kP_{\mathbf j},
\]
where the coefficients $a_k$ are bounded in terms of the curve.
Since $n+1-m\in A_*$,
\[
\frac{L_{m-1}}{L_m}
=
\prod_{i=1}^{n+1-m}\lambda_i
\leq\lambda_{n+1-m}
\leq c_*.
\]
Consequently, the upper bounds defining $E_s$ imply
\[
\left|
\sum_{k=s+2}^{m-1}a_k(t)\mathcal D_kP_{\mathbf j}(t)
\right|
\lesssim_\gamma L_{m-1}
\leq c_*L_m.
\]
Each factor in the leading coefficient is bounded away from zero. Thus, if $c_*$
is sufficiently small, this sum is absorbed by the leading
term, and
\[
|\partial_t^d(\mathcal D_{s+1}P_{\mathbf j})(t)|
\gtrsim_\gamma L_m
\quad(t\in J).
\]

We now compare this lower bound with the upper bound for
$\mathcal D_{s+1}P_{\mathbf j}$. Take any $a<b$ in $J$ and put
$u=(b-a)/d$. The mean value theorem for finite differences
gives some $\xi\in(a,b)$ such that
\[
u^d\partial_t^d(\mathcal D_{s+1}P_{\mathbf j})(\xi)
=
\sum_{\nu=0}^{d}
(-1)^{d-\nu}\binom{d}{\nu}
\mathcal D_{s+1}P_{\mathbf j}(a+\nu u).
\]
All the points $a+\nu u$ belong to $J$. Therefore,
\[
u^dL_m
\lesssim_\gamma
\sum_{\nu=0}^{d}\binom{d}{\nu}
|\mathcal D_{s+1}P_{\mathbf j}(a+\nu u)|
\lesssim_\gamma
\begin{cases}R^\beta L_1,&s=0,\\ L_{s+1},&s\ge1.\end{cases}
\]
Since $a,b$ were arbitrary, we obtain
\[
|J|\lesssim_\gamma
\begin{cases}
R^{\beta/d}(L_1/L_m)^{1/d},&s=0,\\
(L_{s+1}/L_m)^{1/d},&s\ge1.
\end{cases}
\]
The same conclusion is immediate if $J$ is a single point.

It remains to compare the last expression with $\ell_s$.
By the minimality of $m$,
\[
\lambda_i>c_*
\quad(n+2-m\leq i\leq n-s-1).
\]
Hence, for $s+1\leq k\leq m-1$,
\[
\ell_s
\leq
\frac{L_k}{L_{k+1}}
=
\prod_{i=1}^{n-k}\lambda_i
=
\ell_s
\prod_{i=n-k+1}^{n-s-1}\lambda_i^{-1}
\leq c_*^{-(k-s-1)}\ell_s.
\]
Multiplying these inequalities gives
\[
\ell_s^d
\leq
\frac{L_{s+1}}{L_m}
\leq
c_*^{-d(d-1)/2}\ell_s^d.
\]
Since $c_*$ depends only on the curve and $R^{\beta/d}\le R^\beta$,
these comparisons give the lengths asserted in item 2.

\medskip

We finally prove item 3. Recall that
\[
\mathcal D_sP_{\mathbf j}(t)=j_s+G_s(t).
\]
Since $j_s$ is independent of $t$, we have
\[
G_s'(t)=(\mathcal D_sP_{\mathbf j})'(t)
=(\mathcal D_s\gamma_{s+1})'(t)\mathcal D_{s+1}P_{\mathbf j}(t).
\]
Let $J$ be a connected component of $E_s$. For any $a<b$ in $J$,
\[
|G_s(b)-G_s(a)|
\le\int_a^b
|(\mathcal D_s\gamma_{s+1})'(t)\mathcal D_{s+1}P_{\mathbf j}(t)|\,dt.
\]
For $s\ge1$, the integrand is $O_\gamma(L_{s+1})$ and
$|J|\lesssim_\gamma\ell_s$, giving variation $O_\gamma(\widetilde{L}_s)$.
For $s=0$, the integrand is $O_\gamma(R^\beta L_1)$ and
$|J|\lesssim_\gamma R^\beta\ell_0$, giving variation
$O_\gamma(R^{O(\beta)}\widetilde{L}_0)$. Taking the supremum over $a,b\in J$
proves item 3.

\section{Proof of Proposition \ref{prop:n-sufficient-dual}}

In this section, we prove Proposition \ref{prop:n-sufficient-dual}. 
Let us recall the statement. Throughout the section, $\beta$ is a sufficiently small number.

\begin{proposition}\label{prop:bp36-dual-counting}
Let $n\ge3$, $Q\geq1$, $Q^{-1}\le\delta\le1$.
For every $\epsilon>0$, for each admissible scale tuple $(\boldsymbol \lambda,A)$ with
$\mathcal R_{\boldsymbol\lambda}\neq\varnothing$, one has
\begin{equation}
\begin{aligned}
\min_{r\in\mathcal R_{\boldsymbol\lambda}}
\left\{
\#\mathcal T^{(r)}_{\delta,\boldsymbol\lambda}(I^*) \prod_{s=0}^{r-1}(1+\widetilde{L}_s)  \right\}\lesssim_{\epsilon,\beta,\gamma}
Q^\epsilon R^{O(\beta)}\lambda_{\mathrm{prod}}^{-1}E_n(Q^{-1},\delta^{-1}).
\end{aligned}
\end{equation}
\end{proposition}


We abbreviate $\mathcal T^{(r)}_{\delta,\boldsymbol\lambda}(I^*)$
by $\mathcal T^{(r)}$.
We use a compact interval $J_0$ with $I^*\Subset J_0$, contained
in the neighborhood of $I^*$ where
$\operatorname{Wr}(\gamma_0,\ldots,\gamma_k)\ne0$ for $0\le k\le n$.
For each $r\in\mathcal R_{\boldsymbol\lambda}$, define
\[
\widetilde{\mathcal T}^{(r+1)}_{\delta,\boldsymbol\lambda}(J_0):=
\left\{
(j_{r+1},\ldots,j_n)\in\mathbb Z^{n-r}:
\begin{array}{l}
|(j_{r+1},\ldots,j_n)|\lesssim \delta^{-1};\\[1mm]
\text{there exists }t_*\in J_0\text{ such that}:\\[1mm]
\mathcal D_{r+1}P_{\mathbf j}(t_*)=0,\\[1mm]
|\mathcal D_kP_{\mathbf j}(t_*)|\le CL_k
\quad(r+2\le k\le n),\\[1mm]
|\mathcal D_kP_{\mathbf j}(t_*)|\ge cL_k
\quad(r+2\le k\le n,\ n+1-k\in A_*)
\end{array}
\right\}.
\]
We write $\widetilde{\mathcal T}^{(r+1)}$ for
$\widetilde{\mathcal T}^{(r+1)}_{\delta,\boldsymbol\lambda}(J_0)$.
The superscript $r+1$ records the first coefficient in the vector.
This set retains the critical point and higher derivative conditions
from \eqref{eq:n-sufficient-dual-set}, with the parameter interval $J_0$;
the coordinate $j_r$ and the condition on $|\mathcal D_rP_{\mathbf j}|$
are omitted. A witness for a vector in this set is any $t_*\in J_0$
satisfying its defining conditions.

\subsection{Proof of Proposition~\ref{prop:bp36-dual-counting}}
Proposition~\ref{lem:bp36-lead-rough} is proved in
Subsection~\ref{sec:bp36-rough-proof}.
The proofs of Proposition~\ref{prop:bp36-lead-update} and
Lemma~\ref{lem:bp36-lead-product} are given in
Sections~\ref{sec:bp36-lemma-proofs} and~\ref{sec:bp36-gain-proof}, respectively.
Let $\eta>0$ be chosen at the end of the proof. We need to show that
\[
\begin{aligned}
\min_{r\in\mathcal R_{\boldsymbol\lambda}}
\left\{\left(\prod_{s=0}^{r-1}(1+\widetilde{L}_s)\right)
\#\mathcal T^{(r)}\right\}\lesssim_{\epsilon,\beta,\gamma}
Q^\epsilon R^{O(\beta)}\lambda_{\mathrm{prod}}^{-1}E_n(Q^{-1},\delta^{-1}).
\end{aligned}
\]
We begin with a rough bound. Each of $j_n,j_{n-1}$ has
$O(\delta^{-1})$ choices, and the remaining coefficients
$j_s$, $r\le s\le n-2$, contribute the factors $1+\widetilde{L}_s$.
Multiplying by $\prod_{s=0}^{r-1}(1+\widetilde{L}_s)$ gives the following estimate.
We give a detailed proof in Subsection~\ref{sec:bp36-rough-proof}.

\begin{proposition}\label{lem:bp36-lead-rough}
For every $r\in\mathcal R_{\boldsymbol\lambda}$,
\begin{equation}\label{eq:bp36-lead-rough}
\left(\prod_{s=0}^{r-1}(1+\widetilde{L}_s) \right)\#\mathcal T^{(r)}\lesssim_\gamma R^{O(\beta)}\delta^{-2}\prod_{s=0}^{n-2}(1+\widetilde{L}_s).
\end{equation}
\end{proposition}

The bound in \eqref{eq:bp36-lead-rough} may exceed the desired estimate.
To improve it, write $\mathcal R_{\boldsymbol\lambda}=\{r_1,\ldots,r_N\}$
with $r_1>\cdots>r_N$, and define
\[
\begin{aligned}
Z_0&:=\delta^{-2}\prod_{s=0}^{n-2}(1+\widetilde{L}_s),\\
Z_k&:=\min\left\{Z_{k-1},
\left(\prod_{s=0}^{r_k-1}(1+\widetilde{L}_s)\right)\#\mathcal T^{(r_k)}\right\}
\qquad(1\le k\le N).
\end{aligned}
\]
Note that $Z_N\le Z_{N-1} \le \cdots \le Z_0$. The upper bound in \eqref{eq:bp36-lead-rough} is $O(R^{O(\beta)}Z_0)$. The next proposition allows us to improve this bound through iteration.

\begin{proposition}\label{prop:bp36-lead-update}
Let $n\ge3$. For every $r\in\mathcal R_{\boldsymbol\lambda}$, define
\[
\theta_r:=
\begin{cases}
\displaystyle\min\left\{1,
R^{O(\beta)}(1+L_{r+3}^2/L_{r+2}^3)
\max\{\widetilde{L}_r,\delta\}\right\},
&r\le n-3,\ L_{r+2}\ge1,\ \widetilde{L}_r\le1,\\[3mm]
1,
&r\le n-3,\ L_{r+2}<1\ \text{or}\ \widetilde{L}_r>1,\\[3mm]
\displaystyle\min\left\{1,
R^{O(\beta)}\frac{\widetilde{L}_{n-2}+\delta}{1+\widetilde{L}_{n-2}}\right\},
&r=n-2.
\end{cases}
\]
For every $\eta>0$ and every $1\le k\le N$, one has
\[
Z_k\lesssim_{\eta,\gamma}\delta^{-\eta}
\bigl(\theta_{r_k}Z_{k-1}+R^{O(\beta)}\lambda_{\mathrm{prod}}^{-1}E_n(Q^{-1},\delta^{-1})\bigr).
\]
\end{proposition}

The second and third cases in the definition of $\theta_r$ are handled
by elementary computations, while the first case uses
Proposition~\ref{lem:bp36-B}.
We postpone the proof to Section~\ref{sec:bp36-lemma-proofs}.

We now complete the proof of Proposition~\ref{prop:bp36-dual-counting}
using  Proposition~\ref{prop:bp36-lead-update}.
Since $N=\#\mathcal R_{\boldsymbol\lambda}\le n-1$ and
$0<\theta_r\le1$ for $r\in\mathcal R_{\boldsymbol\lambda}$,
iterating this estimate for $k=1,\ldots,N$ gives
\begin{equation}\label{0908}
\begin{aligned}
Z_N&=\min\left\{Z_0,
\min_{r\in\mathcal R_{\boldsymbol\lambda}}
\left(\prod_{s=0}^{r-1}(1+\widetilde{L}_s)\right)\#\mathcal T^{(r)}\right\}\\
&\lesssim_{\eta,\gamma}\delta^{-O(\eta)}
\left(Z_0\prod_{r\in\mathcal R_{\boldsymbol\lambda}}\theta_r
+R^{O(\beta)}\lambda_{\mathrm{prod}}^{-1}E_n(Q^{-1},\delta^{-1})\right).
\end{aligned}
\end{equation}
Combining with Proposition \ref{lem:bp36-B} and the definition of $Z_0$, one has
\begin{equation}\label{0908.54}
\begin{aligned}
&\min_{r\in\mathcal R_{\boldsymbol\lambda}}
\left(\prod_{s=0}^{r-1}(1+\widetilde{L}_s)\right)\#\mathcal T^{(r)}
\\
&\quad\lesssim_{\eta,\gamma} R^{O(\beta)}
\delta^{-O(\eta)}
\left(\delta^{-2}\prod_{s=0}^{n-2}(1+\widetilde{L}_s)\prod_{r\in\mathcal R_{\boldsymbol\lambda}}\theta_r
+R^{O(\beta)}\lambda_{\mathrm{prod}}^{-1}E_n(Q^{-1},\delta^{-1})\right).
\end{aligned}
\end{equation}
It remains to bound the first term inside the parentheses on the right hand side of \eqref{0908.54}.

\begin{lemma}\label{lem:bp36-lead-product}
Let $n\ge3$. Then
\begin{equation}\label{eq:bp36-lead-product}
\delta^{-2}\prod_{s=0}^{n-2}(1+\widetilde{L}_s)
\prod_{r\in\mathcal R_{\boldsymbol\lambda}}\theta_r
\lesssim_\gamma R^{O(\beta)}\lambda_{\mathrm{prod}}^{-1}E_n(Q^{-1},\delta^{-1}).
\end{equation}
\end{lemma}

Applying this lemma in the preceding inequality and absorbing the extra
powers of $R^\beta$ into $R^{O(\beta)}$ gives
\[
\begin{aligned}
\min_{r\in\mathcal R_{\boldsymbol\lambda}}
\left(\prod_{s=0}^{r-1}(1+\widetilde{L}_s)\right)\#\mathcal T^{(r)}\lesssim_{\eta,\gamma}
\delta^{-O(\eta)}R^{O(\beta)}\lambda_{\mathrm{prod}}^{-1}E_n(Q^{-1},\delta^{-1}).
\end{aligned}
\]
Since $\delta^{-1}\le Q$, choosing $\eta>0$ sufficiently small
in terms of $\epsilon$ and $n$ gives
$\delta^{-O(\eta)}\lesssim_{\epsilon,\gamma}Q^\epsilon$.
This finishes the proof of Proposition~\ref{prop:bp36-dual-counting} assuming
Propositions~\ref{lem:bp36-lead-rough} and~\ref{prop:bp36-lead-update},
and Lemma~\ref{lem:bp36-lead-product}.\qed

\bigskip

\subsection{Proof of Proposition~\ref{lem:bp36-lead-rough}}\label{sec:bp36-rough-proof}
We prove the two estimates
\begin{equation}\label{eq:bp36-rough-support}
\#\mathcal T^{(r)}\lesssim_\gamma R^{O(\beta)}(1+\widetilde{L}_r)\#\widetilde{\mathcal T}^{(r+1)},
\qquad
\#\widetilde{\mathcal T}^{(r+1)}\lesssim_\gamma\delta^{-2}\prod_{s=r+1}^{n-2}(1+\widetilde{L}_s).
\end{equation}

For the second estimate, choose $j_n,j_{n-1}$ in $O(\delta^{-2})$
ways and recover $j_{n-2},\ldots,j_{r+1}$ successively.
Let $r+1\le s\le n-2$, and suppose that $j_{n},\ldots,j_{s+1}$
have been chosen. By \eqref{eq:n-normalized-tail},
\[
\mathcal D_sP_{\mathbf j}(t)
=j_s+\sum_{\nu=s+1}^n j_\nu\mathcal D_s\gamma_\nu(t)
=:j_s+G_s(t).
\]
Thus $G_s$ depends only on the coefficients already chosen.
At a witness $t_*$ for a vector in $\widetilde{\mathcal T}^{(r+1)}$,
\[
|j_s+G_s(t_*)|\lesssim_\gamma L_s.
\]
For $s=r+1$ this value is zero; for $s\ge r+2$ the bound is one
of the defining conditions of $\widetilde{\mathcal T}^{(r+1)}$. 

The point $t_*$ may depend on $j_s$. To include all the possible
$t$-values, let
\[
E_s:=\left\{t\in J_0:
\begin{array}{l}
|\mathcal D_kP_{\mathbf j}(t)|\le CL_k
\quad(s+1\le k\le n),\\[1mm]
|\mathcal D_kP_{\mathbf j}(t)|\ge cL_k
\quad(s+1\le k\le n,\ n+1-k\in A_*)
\end{array}\right\},
\]
with the constants from the definition of $\widetilde{\mathcal T}^{(r+1)}$.
These conditions depend only on $j_{s+1},\ldots,j_n$, and every
witness under consideration belongs to $E_s$.
Since $s\ge1$, this is the set $E_s$ in
Lemma~\ref{lem:n-fibre-counting} with $I$ replaced by $J_0$.
The lemma gives $O_n(1)$ connected components $J$, each satisfying
\[
|J|\lesssim_\gamma\ell_s,
\qquad
\operatorname{diam}G_s(J)\lesssim_\gamma\widetilde{L}_s.
\]
Thus the condition $|j_s+G_s(t_*)|\lesssim_\gamma L_s$
places $j_s$ in an interval of length $O_\gamma(\widetilde{L}_s+L_s)$
for each component. Since $L_s\le\widetilde{L}_s$, there are
$O_\gamma(1+\widetilde{L}_s)$ integer choices for $j_s$.
There are only $O_n(1)$ components, so the same bound holds
for all witnesses. Multiplying these bounds for
$s=n-2,\ldots,r+1$ proves the second estimate in
\eqref{eq:bp36-rough-support}.

For the first estimate, let
$\mathbf j_{\ge r+1}\in\widetilde{\mathcal T}^{(r+1)}$.
By \eqref{eq:n-normalized-tail},
\[
\mathcal D_{r+1}P_{\mathbf j}(t)
=j_{r+1}+\sum_{\nu=r+2}^n
j_\nu\mathcal D_{r+1}\gamma_\nu(t).
\]
Thus this function and its zeros depend only on
$\mathbf j_{\ge r+1}$, independent of the coefficient $j_r$
that remains to be chosen.
We first bound the number of witnesses $t_*\in J_0$ for this vector,
each of which satisfies $\mathcal D_{r+1}P_{\mathbf j}(t_*)=0$.

Since $r\in\mathcal R_{\boldsymbol\lambda}$, $n-r-1\in A_*$.
Taking $k=r+2$ in the lower bounds defining
$\widetilde{\mathcal T}^{(r+1)}$, we have
$n+1-k=n-r-1\in A_*$ and hence
\[
|\mathcal D_{r+2}P_{\mathbf j}(t_*)|
\ge cL_{r+2}>0
\]
at every witness. The differentiation identity therefore gives
\[
(\mathcal D_{r+1}P_{\mathbf j})'(t_*)
=(\mathcal D_{r+1}\gamma_{r+2})'(t_*)
\mathcal D_{r+2}P_{\mathbf j}(t_*)\ne0.
\]
Since at least one witness exists, $\mathcal D_{r+1}P_{\mathbf j}$
is not identically zero.

We now follow the proof of item~1 of Lemma~\ref{lem:n-fibre-counting},
with $I$ replaced by $J_0$.
Recall that $\mathcal D_kP_{\mathbf j}$ is either identically zero
or has at most $n-k$ distinct zeros.
Indeed, $(\mathcal D_kP_{\mathbf j})'$ and
$\mathcal D_{k+1}P_{\mathbf j}$ have the same zeros because
$(\mathcal D_k\gamma_{k+1})'$ never vanishes.
If $\mathcal D_{k+1}P_{\mathbf j}$ is not identically zero,
Rolle's theorem bounds the number of zeros of
$\mathcal D_kP_{\mathbf j}$ by one more than that of
$\mathcal D_{k+1}P_{\mathbf j}$; otherwise
$\mathcal D_kP_{\mathbf j}$ is constant.
Starting from $\mathcal D_nP_{\mathbf j}=j_n$ gives the assertion
by descending induction. Applying it with $k=r+1$ shows that
there are at most $n-r-1$ possible witnesses.

At each such point $t_*$, the remaining condition for $j_r$ is
\[
\left|j_r+\sum_{\nu=r+1}^n
j_\nu\mathcal D_r\gamma_\nu(t_*)\right|
\lesssim_\gamma R^{O(\beta)}\widetilde{L}_r.
\]
The sum is already determined by $\mathbf j_{\ge r+1}$ and $t_*$.
Thus $j_r$ lies in an interval of length
$O_\gamma(R^{O(\beta)}\widetilde{L}_r)$, which contains at most
$O_\gamma(1+R^{O(\beta)}\widetilde{L}_r)$ integers.
Since $R\ge1$ and there are only $O_n(1)$ possible witnesses,
this gives $O_\gamma(R^{O(\beta)}(1+\widetilde{L}_r))$ choices for $j_r$
for each $\mathbf j_{\ge r+1}$.
Every vector in $\mathcal T^{(r)}$ projects to a vector in
$\widetilde{\mathcal T}^{(r+1)}$, so summing over these vectors yields
\[
\#\mathcal T^{(r)}
\lesssim_\gamma R^{O(\beta)}(1+\widetilde{L}_r)
\#\widetilde{\mathcal T}^{(r+1)},
\]
which is the first estimate in \eqref{eq:bp36-rough-support}.
Finally, combining the two estimates and multiplying by
$\prod_{s=0}^{r-1}(1+\widetilde{L}_s)$ gives
\eqref{eq:bp36-lead-rough}.\qed

\section{Proof of Proposition~\ref{prop:bp36-lead-update} using the B-process}\label{sec:bp36-lemma-proofs}

\subsection{Van der Corput B-process}\label{sec:bp36-bprocess-estimate}

We begin with the B-process estimate which will be used in the proof of
Proposition~\ref{prop:bp36-lead-update}.

Choose a compact interval $J_1$ with $J_0\Subset J_1$ in the same
neighborhood where
$\operatorname{Wr}(\gamma_0,\ldots,\gamma_k)\ne0$ for $0\le k\le n$.
For $r\in\mathcal R_{\boldsymbol\lambda}$ with $r\le n-3$,
use the projected curve $\mathcal C_{r,n-r-1}$ from
Proposition~\ref{prop:n-projected-curves} with parameter interval $J_1$.
It is nondegenerate by that proposition. We use the reparametrization
$s=\mathcal D_r\gamma_{r+1}(t)$.
For a sufficiently large $A=A(n)>n+1$, write
\[
\mathfrak R_{\mathcal C}(u,D)
=\sum_{\nu\ge0}2^{-A\nu}(1+2^\nu u)
\mathcal R_{\mathcal C}(C2^\nu u,D),
\]
where the constant $C$ depends only on the curve.

\begin{proposition}[B-process]\label{lem:bp36-B}
Let $r\in\mathcal R_{\boldsymbol\lambda}$, $r\le n-3$.
If $L_{r+2}\ge1$ and $1\le K\le\widetilde{L}_r^{-1}$, then
\begin{align}\label{eq:bp36-B}
\#\mathcal T^{(r)}
\lesssim{}&
R^{O(\beta)}\frac{1+L_{r+3}^2/L_{r+2}^3}{K}\#\widetilde{\mathcal T}^{(r+1)}
\notag\\
&+R^{O(\beta)}\frac{\sqrt{L_{r+2}}}{K}
\left(\prod_{k=r+2}^nL_k\right)
\sum_{\substack{1+L_{r+3}^2/L_{r+2}^3\lesssim D\lesssim K\\
D\text{ dyadic}}}
D^{-\frac12}\mathfrak R_{\mathcal C_{r,n-r-1}}(L_{r+2}^{-1},D).
\notag
\end{align}
\end{proposition}

We prove this proposition in Section~\ref{sec:bp36-proof}.

The following lemma bounds the rational-point count in the second term.

\begin{lemma}\label{lem:bp36-planar-bound}
Let $r\in\mathcal R_{\boldsymbol\lambda}$ with $r\le n-3$.
For every $\eta>0$ and $D\ge1$,
\[
\mathfrak R_{\mathcal C_{r,n-r-1}}(u,D)
\lesssim_{\eta,\gamma}D^\eta(uD^2+D),\qquad 0<u\le1.
\]
\end{lemma}

\begin{proof}
The planar projection of $\mathcal C_{r,n-r-1}$ is
$\mathcal C_{r,2}=\{(s,f_r(s))\}$, with
\[
f_r(s)=\mathcal D_r\gamma_{r+2}(t(s)),\qquad f_r'(s)=\frac{(\mathcal D_{r}\gamma_{r+2})'(t(s))}{(\mathcal D_r\gamma_{r+1})'(t(s))}=\mathcal D_{r+1}\gamma_{r+2}(t(s)),\qquad
f_r''(s)=\frac{(\mathcal D_{r+1}\gamma_{r+2})'(t(s))}{(\mathcal D_r\gamma_{r+1})'(t(s))}\ne0.
\]
The curve is smooth, and $|f_r''|$ is bounded above and away from zero. 
The planar counting theorem of Vaughan and Velani~\cite[Theorem~3 and the following remark]{VV2006},
applied with thickness $\max\{u,D^{-1}\}$, gives
\[
\mathcal R_{\mathcal C_{r,2}}(u,D)
\lesssim_{\eta,\gamma}D^\eta(uD^2+D),\qquad 0<u\le1.
\]
Recall that $\mathcal R_{\mathcal C_{r,n-r-1}}(u,D)$ counts pairs
$(\mathbf p,d)\in\mathbb Z^{n-r-1}\times\mathbb N$ satisfying
\[
d\lesssim D,\qquad
\operatorname{dist}(\mathbf p/d,\mathcal C_{r,n-r-1})\le u/d.
\]
Their projections $((p_1,p_2),d)$ are counted by
$\mathcal R_{\mathcal C_{r,2}}(u,D)$.
For given $d,p_1,p_2$, the first coordinate condition
$|ds-p_1|\le u$ confines $s$ to an interval of length $2u/d$.
Each remaining coordinate function has bounded first derivative,
so its values multiplied by $d$ lie in an interval of length
$O_\gamma(u)$. Allowing an error of $u$ therefore leaves
$O_\gamma(1)$ choices for each remaining integer coordinate.
There are $n-r-3$ such coordinates, and hence
\[
\mathcal R_{\mathcal C_{r,n-r-1}}(u,D)
\lesssim_\gamma \mathcal R_{\mathcal C_{r,2}}(u,D)
\lesssim_{\eta,\gamma}D^\eta(uD^2+D)
\]
for $0<u\le1$. 

It remains to sum the weights defining $\mathfrak R$.
For any $v>0$ and a fixed denominator $d$, there are
$O_\gamma(d+v)$ choices for the first integer coordinate.
Once it is fixed, the same argument as above gives
$O_\gamma(1+v)$ choices for each remaining coordinate. Hence
\[
\begin{aligned}
\mathcal R_{\mathcal C_{r,n-r-1}}(v,D)
&\lesssim_\gamma\sum_{1\le d\le D}(d+v)(1+v)^{n-r-2}\\
&\lesssim_\gamma D(D+v)(1+v)^{n-r-2}
\le D^2(1+v)^{n-r-1}.
\end{aligned}
\]
Since $C$ is fixed, this bound and the preceding small-thickness estimate give
\[
\mathcal R_{\mathcal C_{r,n-r-1}}(Cv,D)
\lesssim_{\eta,\gamma}D^\eta(vD^2+D),\qquad 0<v\le1.
\]
Indeed, if $Cv>1$, then $v>C^{-1}$ and the elementary bound applies.
Thus the terms with $2^\nu u\le1$ contribute at most
\[
D^\eta\sum_{\substack{\nu\ge0\\2^\nu u\le1}}
2^{-A\nu}(1+2^\nu u)(2^\nu uD^2+D)
\lesssim_{\eta,\gamma}D^\eta(uD^2+D).
\]
For the remaining terms, the elementary bound gives
\[
\begin{aligned}
&\sum_{\substack{\nu\ge0\\2^\nu u>1}}
2^{-A\nu}(1+2^\nu u)
\mathcal R_{\mathcal C_{r,n-r-1}}(C2^\nu u,D)\\
&\qquad\lesssim_\gamma
D^2u^{n-r}\sum_{\substack{\nu\ge0\\2^\nu u>1}}
2^{-(A-n+r)\nu}\lesssim_\gamma D^2u^A\le uD^2.
\end{aligned}
\]
Here $A>n-r$ and $0<u\le1$. These estimates prove the lemma.
\end{proof}

\bigskip

\subsection{Proof of Proposition~\ref{prop:bp36-lead-update}}\label{sec:bp36-update-proof}

We first state two auxiliary lemmas and prove
Proposition~\ref{prop:bp36-lead-update} using the B-process estimate and these lemmas.
Their proofs are given in Subsections~\ref{sec:bp36-oscillatory-proof}
and~\ref{sec:bp36-comparison-proof}, respectively.
The first lemma controls the oscillatory term in the B-process estimate.

\begin{lemma}\label{lem:bp36-lead-oscillatory}
Let $n\ge3$ and $r\in\mathcal R_{\boldsymbol\lambda}$ satisfy $r\le n-3$,
$L_{r+2}\ge1$ and $\widetilde{L}_r\le1$. We have
\begin{equation}\label{eq:bp36-lead-oscillatory}
\begin{aligned}
\left(\prod_{k=r+2}^nL_k\right)
\left(
\sqrt{\frac{\min\{\delta^{-1},\widetilde{L}_r^{-1}\}}{L_{r+2}}}
+\sqrt{\frac{L_{r+2}}{\min\{\delta^{-1},\widetilde{L}_r^{-1}\}}}
\right)\lesssim_\gamma \lambda_{\mathrm{prod}}^{-1}E_n(Q^{-1},\delta^{-1}).
\end{aligned}
\end{equation}
\end{lemma}

The next lemma compares the counts at different indices in $\mathcal R_{\boldsymbol\lambda}$
and also supplies the estimate needed at the terminal index.

\begin{lemma}\label{lem:bp36-lead-tails}
For $r,t\in\mathcal R_{\boldsymbol\lambda}$ with $r<t$, we have
\begin{equation}\label{eq:bp36-lead-support}
\left(\prod_{s=0}^{r}(1+\widetilde{L}_s)\right)
\#\widetilde{\mathcal T}^{(r+1)}\lesssim_\gamma
\left(\prod_{s=0}^{t-1}(1+\widetilde{L}_s)\right)\#\mathcal T^{(t)}.
\end{equation}
If $n-2\in\mathcal R_{\boldsymbol\lambda}$, then, for every $\eta>0$,
\begin{equation}\label{eq:bp36-terminal}
\#\mathcal T^{(n-2)}
\lesssim_{\eta,\gamma}\delta^{-\eta}(\delta^{-2}R^{O(\beta)}\widetilde{L}_{n-2}+\delta^{-1}).
\end{equation}
\end{lemma}

\begin{proof}[Proof of Proposition~\ref{prop:bp36-lead-update}]
Let $\eta>0$ and let $k$ be an integer with $1\le k\le N$.
We need to prove the following inequality.
\[
Z_k\lesssim_{\eta,\gamma}\delta^{-\eta}
\bigl(\theta_{r_k}Z_{k-1}+R^{O(\beta)}\lambda_{\mathrm{prod}}^{-1}E_n(Q^{-1},\delta^{-1})\bigr).
\]
We first treat two elementary cases. In the remaining range,
we apply Proposition~\ref{lem:bp36-B}.

\medskip\noindent\textit{Case 1: $r_k\le n-3$ and either $L_{r_k+2}<1$ or $\widetilde{L}_{r_k}>1$.}

By definition, $\theta_{r_k}=1$. Then the desired estimate follows
from $Z_k\le Z_{k-1}$ alone.

\medskip\noindent\textit{Case 2: $r_k=n-2$.}

Since the elements of $\mathcal R_{\boldsymbol\lambda}$ are listed in decreasing order, $r_k=n-2$
implies $k=1$. Thus this step starts from $Z_0$.
The definition in Proposition~\ref{prop:bp36-lead-update} gives
\begin{equation}\label{eq:bp36-lead-theta-terminal}
\theta_{n-2}=
\min\left\{1,R^{O(\beta)}\frac{\widetilde{L}_{n-2}+\delta}{1+\widetilde{L}_{n-2}}\right\}.
\end{equation}
By \eqref{eq:bp36-terminal},
\[
\begin{aligned}
\left(\prod_{s=0}^{n-3}(1+\widetilde{L}_s)\right)\#\mathcal T^{(n-2)}
&\lesssim_{\eta,\gamma}\delta^{-\eta}
\left(\prod_{s=0}^{n-3}(1+\widetilde{L}_s)\right)
(\delta^{-2}R^{O(\beta)}\widetilde{L}_{n-2}+\delta^{-1})\\
&\le\delta^{-\eta}R^{O(\beta)}
\frac{\widetilde{L}_{n-2}+\delta}{1+\widetilde{L}_{n-2}}Z_0.
\end{aligned}
\]
Since $r_1=n-2$, the definition of $Z_1$ and the preceding estimate give
\[
\begin{aligned}
Z_1&=\min\left\{Z_0,\left(\prod_{s=0}^{n-3}(1+\widetilde{L}_s)\right)\#\mathcal T^{(n-2)}\right\}\\
&\lesssim_{\eta,\gamma}\delta^{-\eta}
\min\left\{1,R^{O(\beta)}\frac{\widetilde{L}_{n-2}+\delta}{1+\widetilde{L}_{n-2}}\right\}Z_0
=\delta^{-\eta}\theta_{n-2}Z_0.
\end{aligned}
\]
This proves the desired estimate for $k=1$.

\medskip\noindent\textit{Case 3: $r_k\le n-3$, $L_{r_k+2}\ge1$, and $\widetilde{L}_{r_k}\le1$.}

Since $\widetilde L_s\le\widetilde L_{r_k}\le1$ for $0\le s\le r_k$,
all factors $1+\widetilde L_s$ with $s\le r_k$ can be absorbed into
constants depending only on $n$. In particular,
\[
Z_k\lesssim_n\min\{Z_{k-1},\#\mathcal T^{(r_k)}\}.
\]
We apply Proposition~\ref{lem:bp36-B}.
Take $K=\min\{\delta^{-1},\widetilde{L}_{r_k}^{-1}\}$, so that
$1\le K\le\widetilde{L}_{r_k}^{-1}$. The proposition gives
\[
\begin{aligned}
\#\mathcal T^{(r_k)}
\lesssim{}&R^{O(\beta)}
\frac{1+L_{r_k+3}^2/L_{r_k+2}^3}{K}
\#\widetilde{\mathcal T}^{(r_k+1)}\\
&+R^{O(\beta)}\frac{\sqrt{L_{r_k+2}}}{K}
\left(\prod_{s=r_k+2}^nL_s\right)
\sum_{\substack{1+L_{r_k+3}^2/L_{r_k+2}^3\lesssim D\lesssim K\\
D\text{ dyadic}}}
D^{-1/2}\mathfrak R_{\mathcal C_{r_k,n-r_k-1}}(L_{r_k+2}^{-1},D).
\end{aligned}
\]
Since $L_{r_k+2}^{-1}\le1$, Lemma~\ref{lem:bp36-planar-bound} gives
\[
\mathfrak R_{\mathcal C_{r_k,n-r_k-1}}(L_{r_k+2}^{-1},D)
\lesssim_{\eta,\gamma}
D^\eta\left(L_{r_k+2}^{-1}D^2+D\right).
\]
For the dyadic sum, we use
\[
\sum_{\substack{1\le D\lesssim K\\D\text{ dyadic}}}
\left(L_{r_k+2}^{-1}D^{3/2+\eta}+D^{1/2+\eta}\right)
\lesssim_\eta K^\eta
\left(\frac{K^{3/2}}{L_{r_k+2}}+K^{1/2}\right).
\]
Substituting these two bounds into the preceding estimate for
$\#\mathcal T^{(r_k)}$ and using $K^\eta\le\delta^{-\eta}$, we obtain
\begin{equation}\label{eq:bp36-usable-B}
\begin{aligned}
\#\mathcal T^{(r_k)}
&\lesssim_{\eta,\gamma} R^{O(\beta)}\delta^{-\eta}
\Biggl(\frac{(1+L_{r_k+3}^2/L_{r_k+2}^3)\#\widetilde{\mathcal T}^{(r_k+1)}}{\min\{\delta^{-1},\widetilde{L}_{r_k}^{-1}\}}\\
&\qquad+\left(\prod_{s=r_k+2}^nL_s\right)
\left(\sqrt{\frac{\min\{\delta^{-1},\widetilde{L}_{r_k}^{-1}\}}{L_{r_k+2}}}
+\sqrt{\frac{L_{r_k+2}}{\min\{\delta^{-1},\widetilde{L}_{r_k}^{-1}\}}}\right)\Biggr).
\end{aligned}
\end{equation}
By Lemma~\ref{lem:bp36-lead-oscillatory}, the second term in
\eqref{eq:bp36-usable-B} is
$\lesssim_\gamma\lambda_{\mathrm{prod}}^{-1}E_n(Q^{-1},\delta^{-1})$.
For the first term, \eqref{eq:bp36-rough-support} and
Lemma~\ref{lem:bp36-lead-tails} give
\[
\begin{aligned}
\#\widetilde{\mathcal T}^{(r_k+1)}&\lesssim_\gamma Z_0,\\
\#\widetilde{\mathcal T}^{(r_k+1)}&\lesssim_\gamma
\left(\prod_{s=0}^{t-1}(1+\widetilde L_s)\right)\#\mathcal T^{(t)}
\quad(t\in\mathcal R_{\boldsymbol\lambda},\ t>r_k).
\end{aligned}
\]
Taking the minimum over $t=r_1,\ldots,r_{k-1}$ and the bound by $Z_0$,
we obtain
\[
\#\widetilde{\mathcal T}^{(r_k+1)}\lesssim_\gamma Z_{k-1}.
\]
For $k=1$, this follows directly from the bound by $Z_0$.
Substituting these estimates into \eqref{eq:bp36-usable-B} and using
$K^{-1}=\max\{\widetilde L_{r_k},\delta\}$ gives
\begin{equation}\label{eq:bp36-lead-step}
\begin{aligned}
\#\mathcal T^{(r_k)}
&\lesssim_{\eta,\gamma}R^{O(\beta)}\delta^{-\eta}
\Biggl(\left(1+\frac{L_{r_k+3}^2}{L_{r_k+2}^3}\right)
\max\{\widetilde L_{r_k},\delta\}Z_{k-1}
+\lambda_{\mathrm{prod}}^{-1}E_n(Q^{-1},\delta^{-1})\Biggr).
\end{aligned}
\end{equation}
In this case, the definition in Proposition~\ref{prop:bp36-lead-update} gives
\begin{equation}\label{eq:bp36-lead-theta}
\theta_{r_k}=\min\left\{1,
R^{O(\beta)}\left(1+\frac{L_{r_k+3}^2}{L_{r_k+2}^3}\right)
\max\{\widetilde L_{r_k},\delta\}\right\}.
\end{equation}
Combining \eqref{eq:bp36-lead-step} with
$Z_k\lesssim_n\min\{Z_{k-1},\#\mathcal T^{(r_k)}\}$ proves
\begin{equation}\label{eq:bp36-lead-update}
Z_k\lesssim_{\eta,\gamma}\delta^{-\eta}
\bigl(\theta_{r_k}Z_{k-1}+R^{O(\beta)}\lambda_{\mathrm{prod}}^{-1}E_n(Q^{-1},\delta^{-1})\bigr).
\end{equation}

This proves the estimate in all three cases.
\end{proof}

\bigskip

\subsection{Proof of Lemma~\ref{lem:bp36-lead-oscillatory}}\label{sec:bp36-oscillatory-proof}

\begin{proof}
We have $L_k\le\delta^{-1}$ for $1\le k\le n$.
We first record numerical bounds valid for every such index.
The scale normalization implies
\begin{equation}\label{eq:bp36-normalization}
\prod_{j=1}^{n-1}\lambda_j^{n+1-j}=\frac{\delta}{Q},
\qquad QL_k\le\lambda_{\mathrm{prod}}^{-k}
\quad(1\le k\le n).
\end{equation}
Indeed, the second inequality follows by dividing the formula for $L_k$
by $Q^{-1}$ and comparing every resulting exponent with $k$.
The definition \eqref{eq:section-two-En} includes only the individual error terms,
with no additional $1$. Since $D_{n,k-1}=k(n-k+1)-1$, its $(k-1)$-th term gives
\begin{equation}\label{eq:bp36-monomial-bounds}
\begin{aligned}
\lambda_{\mathrm{prod}}^{-1}E_n(Q^{-1},\delta^{-1})&\ge\lambda_{\mathrm{prod}}^{-1}
\bigl(Q^{-1}\delta^{-D_{n,k-1}}\bigr)^{1/k}\\
&\ge\bigl(L_k\delta^{-D_{n,k-1}}\bigr)^{1/k}\\
&=\delta^{k-n-1}\left(\delta L_k\right)^{1/k}
\qquad(1\le k\le n).
\end{aligned}
\end{equation}
In particular, $\lambda_{\mathrm{prod}}^{-1}E_n(Q^{-1},\delta^{-1})\ge \delta^{-1}$ by taking $k=n$. Taking $k=r+2$ and
$k=r+1$, and using $0<\delta L_k\le1$, yields
\begin{equation}\label{eq:bp36-budget-scales}
\begin{aligned}
\delta^{r+2-n}\sqrt{\frac{L_{r+2}}{\delta}}&\le \lambda_{\mathrm{prod}}^{-1}E_n(Q^{-1},\delta^{-1}),\\
\delta^{r+1-n}L_{r+1}&\le \lambda_{\mathrm{prod}}^{-1}E_n(Q^{-1},\delta^{-1})
\qquad(0\le r\le n-3).
\end{aligned}
\end{equation}
These inequalities hold for every $r$ in the displayed range.

Now let $r$ satisfy the hypotheses of the lemma and keep the choice
$K=\min\{\delta^{-1},\widetilde L_r^{-1}\}$. The first term is bounded by
\[
\begin{aligned}
\left(\prod_{j=r+2}^nL_j\right)\sqrt{\frac{K}{L_{r+2}}}
&\le\left(\prod_{j=r+3}^nL_j\right)\sqrt{\frac{L_{r+2}}{\delta}}\\
&\le \delta^{r+2-n}\sqrt{\frac{L_{r+2}}{\delta}}\le \lambda_{\mathrm{prod}}^{-1}E_n(Q^{-1},\delta^{-1}).
\end{aligned}
\]
For the second, $K^{-1}=\max\{\delta,\widetilde L_r\}$ and
$L_{r+2}\widetilde L_r=L_{r+1}^2$ give
\[
\begin{aligned}
\left(\prod_{j=r+2}^nL_j\right)\sqrt{\frac{L_{r+2}}K}
&\le\left(\prod_{j=r+2}^nL_j\right)
\left(\sqrt{\delta L_{r+2}}+L_{r+1}\right)\\
&\le \delta^{r+2-n}L_{r+2}+\delta^{r+1-n}L_{r+1}\\
&\le2\lambda_{\mathrm{prod}}^{-1}E_n(Q^{-1},\delta^{-1}).
\end{aligned}
\]
Here $L_{r+2}\le\sqrt{\frac{L_{r+2}}{\delta}}$ was used in the last step.
Combining the two estimates proves the lemma.
\end{proof}

\bigskip

\subsection{Proof of Lemma~\ref{lem:bp36-lead-tails}}\label{sec:bp36-comparison-proof}

\begin{proof}
Let $r,t\in\mathcal R_{\boldsymbol\lambda}$ with $r<t$. To prove \eqref{eq:bp36-lead-support},
it suffices to show that
\begin{equation}\label{eq:bp36-support-recursion}
\#\widetilde{\mathcal T}^{(r+1)}\lesssim_\gamma
\left(\prod_{s=r+1}^{t-1}(1+\widetilde{L}_s)\right)\#\mathcal T^{(t)}.
\end{equation}
The comparison uses the interval and cutoff convention in
Section~\ref{sec:bp36-proof}. In particular, the parameter interval
and derivative bounds for $\widetilde{\mathcal T}^{(r+1)}$ leave room for the displacement
into $\mathcal T^{(t)}$ whenever $t\in\mathcal R_{\boldsymbol\lambda}$ and $t>r$.
Let $t_*$ witness a vector in $\widetilde{\mathcal T}^{(r+1)}$.
Then $|\mathcal D_{t+1}P_{\mathbf j}(t_*)|\lesssim L_{t+1}$ and
$|\mathcal D_{t+2}P_{\mathbf j}(t_*)|\asymp L_{t+2}$.
Since $(\mathcal D_kP_{\mathbf j})'=(\mathcal D_k\gamma_{k+1})'\mathcal D_{k+1}P_{\mathbf j}$ and
\[
\ell_t\frac{L_{k+1}}{L_k}
\le\lambda_{n-t-1}\le c_*\quad(t+2\le k\le n-1),
\qquad (\mathcal D_nP_{\mathbf j})'=0,
\]
the higher derivative bounds persist over distance $O_\gamma(\ell_t)$.
In particular, $(\mathcal D_{t+1}P_{\mathbf j})'$ has constant sign and size
$\asymp_\gamma L_{t+2}$ there. The intermediate value theorem
therefore gives a zero of $\mathcal D_{t+1}P_{\mathbf j}$ within that distance.
At this zero,
\[
|\mathcal D_tP_{\mathbf j}|\lesssim_\gamma
L_t+\ell_tL_{t+1}\lesssim\widetilde{L}_t,
\]
so $(j_t,\ldots,j_n)$ is counted by $\mathcal T^{(t)}$.
For each such vector, apply Lemma~\ref{lem:n-fibre-counting}
at the original witnesses, as in the proof of
Proposition~\ref{lem:bp36-lead-rough}. This gives $O_\gamma(1+\widetilde{L}_s)$
choices for $j_s$, successively for $s=t-1,\ldots,r+1$.
This proves \eqref{eq:bp36-support-recursion}. Multiplying both sides by
$\prod_{s=0}^{r}(1+\widetilde{L}_s)$ gives
\[
\left(\prod_{s=0}^{r}(1+\widetilde{L}_s)\right)\#\widetilde{\mathcal T}^{(r+1)}
\lesssim_\gamma
\left(\prod_{s=0}^{t-1}(1+\widetilde{L}_s)\right)
\#\mathcal T^{(t)}.
\]
This proves \eqref{eq:bp36-lead-support}.

Finally, suppose that $n-2\in\mathcal R_{\boldsymbol\lambda}$. Let $\mathbf j_{\ge n-2}\in \mathcal T^{(n-2)}$.
Then $|j_n|\asymp L_n=\delta^{-1}$. For a witness $t_*$, one has $\mathcal D_{n-1} P_{\mathbf j}=0$. By \eqref{eq:n-normalized-tail} with $k=n-1,n-2$, one has 
\[
\frac{j_{n-1}}{j_n}=-\mathcal D_{n-1}\gamma_n,
\qquad
\left|\frac{j_{n-2}}{j_n}
- (\mathcal D_{n-2}\gamma_{n-1}\cdot\mathcal D_{n-1}\gamma_n
-\mathcal D_{n-2}\gamma_n\bigr)\right|
\lesssim_\gamma\delta R^{O(\beta)}\widetilde{L}_{n-2}.
\]
Thus we count rational points at height $\delta^{-1}$ and thickness
$R^{O(\beta)}\widetilde{L}_{n-2}$ near the planar curve
\[
t\longmapsto
\left(-\mathcal D_{n-1}\gamma_n(t),\,
\mathcal D_{n-2}\gamma_{n-1}(t)\cdot \mathcal D_{n-1}\gamma_n(t)
-\mathcal D_{n-2}\gamma_n(t)\right).
\]
Its first coordinate has derivative $-(\mathcal D_{n-1}\gamma_n)'\ne0$.
With this coordinate as parameter, we rewrite the curve as
\[ s\longmapsto (-s,y(s)), \]
where $y(s)=\mathcal D_{n-2}\gamma_{n-1}(t(s))\cdot\mathcal D_{n-1}\gamma_n(t(s))-\mathcal D_{n-2}\gamma_n(t(s))$. Noting that $\frac{d}{ds}=\frac{1}{(\mathcal D_{n-1}\gamma_n)'}\frac{d}{dt}$, one can compute (using \eqref{eq:n-normalized-derivatives}) that
\[ \frac{d}{ds} y= \frac{1}{(\mathcal D_{n-1}\gamma_n)'}\Big( (\mathcal D_{n-2}\gamma_{n-1})'\cdot \mathcal D_{n-1}\gamma_n+\mathcal D_{n-2}\gamma_{n-1}\cdot(\mathcal D_{n-1}\gamma_n)'-(\mathcal D_{n-2}\gamma_n)' \Big)=\mathcal D_{n-2}\gamma_{n-1}(t(s)), \]
\[ \frac{d^2}{ds^2}y=\frac{(\mathcal D_{n-2}\gamma_{n-1})'}{(\mathcal D_{n-1}\gamma_n)'}\ne0. \]

For $0<R^{O(\beta)}\widetilde{L}_{n-2}\le1$, the planar counting theorem used in
Lemma~\ref{lem:bp36-planar-bound} gives
\[
\#\mathcal T^{(n-2)}
\lesssim_{\eta,\gamma}\delta^{-\eta}
(\delta^{-2}R^{O(\beta)}\widetilde{L}_{n-2}+\delta^{-1}).
\]
For $R^{O(\beta)}\widetilde{L}_{n-2}>1$, choose $j_n,j_{n-1}$ in
$O(\delta^{-2})$ ways. Each pair $(j_{n-1},j_n)$ has $O_n(1)$ relevant critical points,
each allowing $O_\gamma(1+R^{O(\beta)}\widetilde{L}_{n-2})$ choices for $j_{n-2}$.
Hence the count is $O_\gamma(\delta^{-2}R^{O(\beta)}\widetilde{L}_{n-2})$.
This proves \eqref{eq:bp36-terminal}.
\end{proof}

\section{Proof of Lemma~\ref{lem:bp36-lead-product}}\label{sec:bp36-gain-proof}
We use the existing quantities $Z_0$, $\theta_r$, $L_k$, and
$\widetilde L_r$. The main step is
\eqref{eq:bp36-volume-step}:
multiplication by $\theta_r$ replaces
$L_{r+2}^2\prod_{j=r+3}^nL_j$ by
$L_{r+1}^2\prod_{j=r+2}^nL_j$, up to an error bounded by
$3\lambda_{\mathrm{prod}}^{-1}E_n(Q^{-1},\delta^{-1})$.
Iterating this inequality proves the lemma.

\subsection{The starting index}
\label{sec:bp36-gain-preliminary}
As in the definition in Proposition~\ref{prop:bp36-lead-update},
define $\theta_r$ for every $0\le r\le n-2$ by
\[
\theta_r:=
\begin{cases}
\displaystyle\min\left\{1,
R^{O(\beta)}(1+L_{r+3}^2/L_{r+2}^3)
\max\{\widetilde L_r,\delta\}\right\},
&r\le n-3,\ L_{r+2}\ge1,\ \widetilde L_r\le1,\\[3mm]
1,&r\le n-3,\ L_{r+2}<1\ \text{or}\ \widetilde L_r>1,\\[2mm]
\displaystyle\min\left\{1,
R^{O(\beta)}\frac{\widetilde L_{n-2}+\delta}{1+\widetilde L_{n-2}}
\right\},&r=n-2.
\end{cases}
\]
For $r\in\mathcal R_{\boldsymbol\lambda}$ this agrees with the original
definition. Until Subsection~\ref{sec:bp36-gain-restore}, we evaluate
these formulas at $\beta=0$, replacing $R^{O(\beta)}$ by $1$.
We shall prove
\begin{equation}\label{eq:bp36-gain-result}
Z_0\prod_{r=0}^{n-2}\theta_r\lesssim_n \lambda_{\mathrm{prod}}^{-1}E_n(Q^{-1},\delta^{-1}).
\end{equation}
Here
\begin{equation*}
Z_0=\delta^{-2}\prod_{s=0}^{n-2}(1+\widetilde L_s).
\end{equation*}

Choose
\[
r_*:=\max\{0\le r\le n-2:\widetilde L_r\le1\}.
\]
It exists because $\widetilde L_0=Q^{-1}\le\delta\le1$.
Since $\widetilde L_s$ is nondecreasing, $1+\widetilde L_s\asymp1$
for $s\le r_*$, whereas $1+\widetilde L_s\asymp\widetilde L_s$
for $s>r_*$. Hence
\begin{equation}\label{eq:bp36-gain-start}
\begin{aligned}
Z_0&\lesssim_n\delta^{-2}\prod_{s=r_*+1}^{n-2}\widetilde L_s=\begin{cases}
\displaystyle L_{r_*+2}^2\prod_{j=r_*+3}^nL_j,&r_*\le n-3,\\[2mm]
\delta^{-2},&r_*=n-2.
\end{cases}
\end{aligned}
\end{equation}
Here we used $\widetilde L_s=L_{s+1}^2/L_{s+2}$ and $L_n=\delta^{-1}$
to cancel the products.

\bigskip

\subsection{The estimate at one index}
\label{sec:bp36-gain-one-step}
Let $0\le r\le r_*$ and $r\le n-3$. We prove
\begin{equation}\label{eq:bp36-volume-step}
\begin{aligned}
\theta_r L_{r+2}^2\prod_{j=r+3}^nL_j
&\le L_{r+1}^2\prod_{j=r+2}^nL_j
+3\lambda_{\mathrm{prod}}^{-1}E_n(Q^{-1},\delta^{-1}).
\end{aligned}
\end{equation}
Since $L_j\le\delta^{-1}$, \eqref{eq:bp36-budget-scales} gives
\begin{equation}\label{eq:bp36-volume-numerical}
\begin{aligned}
\left(\prod_{j=r+3}^nL_j\right)L_{r+2}
&\le\left(\prod_{j=r+3}^nL_j\right)\sqrt{\frac{L_{r+2}}{\delta}}\\
&\le\delta^{r+2-n}\sqrt{\frac{L_{r+2}}{\delta}}
\le\lambda_{\mathrm{prod}}^{-1}E_n(Q^{-1},\delta^{-1}),\\
\delta^{r+1-n}L_{r+1}
&\le\lambda_{\mathrm{prod}}^{-1}E_n(Q^{-1},\delta^{-1}).
\end{aligned}
\end{equation}
If
\[
L_{r+2}^2\prod_{j=r+3}^nL_j
\le\lambda_{\mathrm{prod}}^{-1}E_n(Q^{-1},\delta^{-1}),
\]
then \eqref{eq:bp36-volume-step} follows from $\theta_r\le1$.
Otherwise, \eqref{eq:bp36-volume-numerical} implies
\[
L_{r+2}^2>\sqrt{\frac{L_{r+2}}{\delta}}
\quad\Longrightarrow\quad L_{r+2}^3>\delta^{-1}\ge1.
\]
Also, $\widetilde L_r\le1$ because $r\le r_*$. Thus the first case
in the definition of $\theta_r$ applies. Using
$\max\{\widetilde L_r,\delta\}\le\widetilde L_r+\delta$ and
$\widetilde L_r=L_{r+1}^2/L_{r+2}$, we obtain
\[
\begin{aligned}
&\theta_r L_{r+2}^2\prod_{j=r+3}^nL_j\\
&\quad\le(\widetilde L_r+\delta)
\left(1+\frac{L_{r+3}^2}{L_{r+2}^3}\right)
L_{r+2}^2\prod_{j=r+3}^nL_j\\
&\quad=L_{r+1}^2\prod_{j=r+2}^nL_j
+\left(\prod_{j=r+3}^nL_j\right)
\left(\frac{L_{r+3}^2L_{r+1}^2}{L_{r+2}^2}
+\delta L_{r+2}^2+\frac{\delta L_{r+3}^2}{L_{r+2}}\right).
\end{aligned}
\]
Each of the three additional terms is at most
$\lambda_{\mathrm{prod}}^{-1}E_n(Q^{-1},\delta^{-1})$.
For the first term, use
\[
L_{r+1}=\lambda_{n-r-1}\frac{L_{r+2}^2}{L_{r+3}}
\le\frac{L_{r+2}^2}{L_{r+3}}.
\]
Then
\[
\begin{aligned}
\left(\prod_{j=r+3}^nL_j\right)
\frac{L_{r+3}^2L_{r+1}^2}{L_{r+2}^2}
&\le\left(\prod_{j=r+3}^nL_j\right)L_{r+3}L_{r+1}\\
&\le\delta^{r+1-n}L_{r+1}
\le\lambda_{\mathrm{prod}}^{-1}E_n(Q^{-1},\delta^{-1}).
\end{aligned}
\]
For the second term, $\delta L_{r+2}\le1$ gives
\[
\left(\prod_{j=r+3}^nL_j\right)\delta L_{r+2}^2
\le\left(\prod_{j=r+3}^nL_j\right)L_{r+2}
\le\lambda_{\mathrm{prod}}^{-1}E_n(Q^{-1},\delta^{-1}).
\]
For the third, $L_{r+3}\le\delta^{-1}$ and
$L_{r+2}^3>\delta^{-1}$ give
\[
\begin{aligned}
\left(\prod_{j=r+3}^nL_j\right)\frac{\delta L_{r+3}^2}{L_{r+2}}
&\le\left(\prod_{j=r+3}^nL_j\right)\frac{1}{\delta L_{r+2}}\\
&\le\left(\prod_{j=r+3}^nL_j\right)\sqrt{\frac{L_{r+2}}{\delta}}\\
&\le\lambda_{\mathrm{prod}}^{-1}E_n(Q^{-1},\delta^{-1}).
\end{aligned}
\]
These bounds prove \eqref{eq:bp36-volume-step}.

\bigskip

\subsection{Iteration and the final bound}
\label{sec:bp36-gain-volume-iteration}
If $r_*\le n-3$, apply \eqref{eq:bp36-volume-step} at $r=r_*$
to \eqref{eq:bp36-gain-start}:
\[
Z_0\theta_{r_*}\lesssim_n
L_{r_*+1}^2\prod_{j=r_*+2}^nL_j
+\lambda_{\mathrm{prod}}^{-1}E_n(Q^{-1},\delta^{-1}).
\]
If $r_*=n-2$, the terminal formula gives the same bound, since
\eqref{eq:bp36-monomial-bounds} with $k=n$ implies
$\delta^{-1}\le\lambda_{\mathrm{prod}}^{-1}E_n(Q^{-1},\delta^{-1})$:
\[
\begin{aligned}
Z_0\theta_{n-2}
&\lesssim_n\delta^{-2}(\widetilde L_{n-2}+\delta)
=L_{n-1}^2L_n+\delta^{-1}\\
&\le L_{n-1}^2L_n
+\lambda_{\mathrm{prod}}^{-1}E_n(Q^{-1},\delta^{-1}).
\end{aligned}
\]
We now apply \eqref{eq:bp36-volume-step} successively for
$r=r_*-1,\ldots,0$. At each step, the preceding error is multiplied
by $\theta_r\le1$. There are at most $n-1$ steps, so
\[
\begin{aligned}
Z_0\prod_{s=r}^{r_*}\theta_s
&\lesssim_n L_{r+1}^2\prod_{j=r+2}^nL_j
+\lambda_{\mathrm{prod}}^{-1}E_n(Q^{-1},\delta^{-1})
\qquad(0\le r\le r_*).
\end{aligned}
\]
In particular,
\[
\begin{aligned}
Z_0\prod_{s=0}^{n-2}\theta_s
&\le Z_0\prod_{s=0}^{r_*}\theta_s\\
&\lesssim_n L_1^2\prod_{j=2}^nL_j
+\lambda_{\mathrm{prod}}^{-1}E_n(Q^{-1},\delta^{-1}).
\end{aligned}
\]
To bound the remaining product, use the scale normalization and
$Q\delta\ge1$:
\[
\begin{aligned}
\lambda_{\mathrm{prod}}L_1^2\prod_{j=2}^nL_j
&=Q^{-1}\delta^{-n}\prod_{i=1}^{n-1}
\lambda_i^{(n-i)(n+1-i)/2}\\
&\le Q^{-1}\delta^{-n}\left(\frac{\delta}{Q}\right)^{1/2}\\
&=Q^{-1}\delta^{1-n}(Q\delta)^{-1/2}
\le Q^{-1}\delta^{1-n}\\
&\le E_n(Q^{-1},\delta^{-1}).
\end{aligned}
\]
The last inequality follows from the $k=0$ term in
\eqref{eq:n-full-range-budget}. Hence
\begin{equation}\label{eq:bp36-budget-endpoints}
L_1^2\prod_{j=2}^nL_j
\le\lambda_{\mathrm{prod}}^{-1}E_n(Q^{-1},\delta^{-1}),
\end{equation}
and \eqref{eq:bp36-gain-result} follows.

\bigskip

\subsection{Return to \texorpdfstring{$\beta>0$ and $r\in\mathcal R_{\boldsymbol\lambda}$}{beta > 0 and r in R(lambda)}}
\label{sec:bp36-gain-restore}
For all $0\le r\le n-2$, the definitions imply
\[
\left.\theta_r\right|_{\beta=0}
\le\theta_r\le R^{O(\beta)}\left.\theta_r\right|_{\beta=0}.
\]
The scale identity
\begin{equation}\label{eq:bp36-key-identity}
\widetilde L_r\frac{L_{r+3}^2}{L_{r+2}^3}
=\lambda_{n-r-1}^2\qquad(0\le r\le n-3)
\end{equation}
implies $\theta_r\ge\lambda_{n-r-1}^2$ in the first case of
its definition; the same bound holds when $\theta_r=1$.
For the terminal index, use
$\widetilde L_{n-2}=\delta^{-1}\lambda_1^2$ to get
$\theta_{n-2}\ge\lambda_1^2/2$.
If $r\ge1$ and $r\notin\mathcal R_{\boldsymbol\lambda}$, then
$\lambda_{n-r-1}>c_*$, so $\theta_r\gtrsim_\gamma1$.
If $0\notin\mathcal R_{\boldsymbol\lambda}$, then
$\lambda_{n-1}>c_*R^{-\beta}$, so
$\theta_0\gtrsim_\gamma R^{-O(\beta)}$.
Consequently,
\[
\begin{aligned}
Z_0\prod_{r\in\mathcal R_{\boldsymbol\lambda}}\theta_r
&\lesssim_\gamma R^{O(\beta)}Z_0\prod_{r=0}^{n-2}\theta_r\\
&\lesssim_\gamma R^{O(\beta)}Z_0
\prod_{r=0}^{n-2}\left.\theta_r\right|_{\beta=0}\\
&\lesssim_\gamma R^{O(\beta)}\lambda_{\mathrm{prod}}^{-1}E_n(Q^{-1},\delta^{-1}).
\end{aligned}
\]
This is \eqref{eq:bp36-lead-product}.\qed

\section{Proof of Proposition~\ref{lem:bp36-B}}\label{sec:bp36-proof}

We follow the B-process in \cite[Section 7, Lemmas 7.5 and 7.6]{CSST2026},
retaining the individual scales $L_{r+2},\ldots,L_n$.
Let $r\in\mathcal R_\lambda$ with $r\le n-3$.
We may assume $\mathcal T^{(r)}\ne\varnothing$, and hence
\[
\#\widetilde{\mathcal T}^{(r+1)}\ge1,
\qquad
1\le K\le\widetilde L_r^{-1}
\quad\Longrightarrow\quad \widetilde L_r\le1.
\]
Choose the intervals below so that
$\operatorname{Wr}(\gamma_0,\ldots,\gamma_k)\ne0$ on $J_1$ for $0\le k\le n$:
\[
\underbrace{I^*}_{\mathcal T^{(r)}}
\Subset
\underbrace{J_0}_{\widetilde{\mathcal T}^{(r+1)}}
\Subset
\underbrace{J_1}_{\mathcal C_{r,n-r-1}}.
\]
For $r,t\in\mathcal R_{\boldsymbol\lambda}$ with $r<t$, choose the intervals and thresholds successively
so that the conditions for $\mathcal T^{(t)}$ include the displacements
from $\widetilde{\mathcal T}^{(r+1)}$ in the proof of
Lemma~\ref{lem:bp36-lead-tails}.
These choices depend only on the curve.

\subsection*{Step 1. The critical-value sum}\label{subsec:bp52-critical-value}
Use $s=\mathcal D_r\gamma_{r+1}(t)$ and denote the images of the
parameter intervals by the same symbols. Write
\begin{equation}\label{Hj}
    H_{\mathbf j}(s)=\sum_{\nu=r+2}^n j_\nu\mathcal D_r\gamma_\nu(t(s)).
\end{equation}

Then
\begin{equation}\label{Hj'}
    (\mathcal D_rP_{\mathbf j})(t(s))=j_r+j_{r+1}s+H_{\mathbf j}(s),
\qquad (\mathcal D_{r+1}P_{\mathbf j})(t(s))=j_{r+1}+H_{\mathbf j}'(s).
\end{equation}

The functions $H_{\mathbf j}$ depend only on
$\mathbf j_{\ge r+2}$.
The differentiation identity in this coordinate is
\[
\frac{d}{ds}(\mathcal D_kP_{\mathbf j})(t(s))
=\frac{(\mathcal D_k\gamma_{k+1})'(t(s))}{(\mathcal D_r\gamma_{r+1})'(t(s))}
(\mathcal D_{k+1}P_{\mathbf j})(t(s))\qquad(r\le k\le n-1).
\]
Its coefficients are smooth and bounded above and away from zero.

Choose a nonnegative $a\in C_c^\infty(\operatorname{int}J_0\times\mathbb R^{n-r-1})$, written as $a(s,\mathbf j_{\ge r+2})$, with
\[
0\le a\lesssim_\gamma1,\qquad a\ge1\quad\text{at the original witnesses}.
\]
We choose its smooth cutoffs so that
\[
a(s,\mathbf j_{\ge r+2})\ne0
\quad\Longrightarrow\quad
\begin{cases}
s\in\operatorname{int}J_0,\\[1mm]
|\mathbf j_{\ge r+2}|\lesssim\delta^{-1},\quad
|H_{\mathbf j}'(s)|\lesssim\delta^{-1},\\[1mm]
|(\mathcal D_kP_{\mathbf j})(t(s))|\le CL_k
& (r+2\le k\le n),\\[1mm]
|(\mathcal D_kP_{\mathbf j})(t(s))|\ge cL_k
& (r+2\le k\le n,\ n+1-k\in A_*).
\end{cases}
\]
These are cutoffs in the coefficient variables $s$, $\mathbf j_{\ge r+2}$, $H'_{\mathbf j}(s)$, $\delta H_{\mathbf j}'(s)$,
and $(\mathcal D_kP_{\mathbf j})(t(s))/L_k$. For example, we may choose smooth bump functions $\phi_{J_0},\phi,\psi$ adapted to
$J_0, [-C,C], [-C,-c]\cup[c,C]$, respectively, and let
\begin{equation}\label{defa}
    a(s,\mathbf{j}_{\ge r+2})=\phi_{J_0}(s)\phi(\de \mathbf{j}_{\ge r+2})\phi(\de H_{\mathbf j}'(s))\prod_{n+1-k\notin A_*}\phi((\mathcal D_kP_{\mathbf j})(t(s))/L_k)\prod_{n+1-k\in A_*}\psi((\mathcal D_kP_{\mathbf j})(t(s))/L_k)
\end{equation}

At a critical point in this support,
\[
j_{r+1}+H_{\mathbf j}'(s)=0
\quad\Longrightarrow\quad
\begin{cases}
|j_{r+1}|=|H_{\mathbf j}'(s)|\lesssim\delta^{-1},\\[1mm]
(\mathcal D_{r+1}P_{\mathbf j})(t(s))=0.
\end{cases}
\]
Thus all the defining conditions of $\widetilde{\mathcal T}^{(r+1)}$
hold with witness $t_*=t(s)$:
\[
\left\{
\begin{gathered}
a(s,\mathbf j_{\ge r+2})\ne0,\\
j_{r+1}+H_{\mathbf j}'(s)=0
\end{gathered}
\right.
\quad\Longrightarrow\quad
\mathbf j_{\ge r+1}\in\widetilde{\mathcal T}^{(r+1)}.
\]

In the $s$ coordinate, the normalized derivatives of order at
least two are lower-triangular combinations of ordinary derivatives
of orders at least two, with nonzero leading coefficients.
Using $L_{r+2}\le\cdots\le L_n=\delta^{-1}$, we obtain, on the support of $a$,
\begin{equation}\label{eq:bp52-derivative-bounds}
|H_{\mathbf j}''(s)|\asymp_\gamma L_{r+2},\qquad
|H_{\mathbf j}^{(k)}(s)|\lesssim_\gamma L_{r+k}
\quad(3\le k\le n-r),\qquad
|\mathbf j_{\ge r+2}|\lesssim\delta^{-1}.
\end{equation}
Higher derivatives are $O_\gamma(\delta^{-1})$.

In the next proposition,
all sums over critical points (i.e., $s$ such that $j_{r+1}+H_{\mathbf j}'(s)=0$) are restricted to points in the support of
$a$. For each $\mathbf j_{\ge r+1}$, there are at most $n-r-1$ such points:
at each one, $\mathcal D_{r+2}P_{\mathbf j}\ne0$, and the proof of item~1
of Lemma~\ref{lem:n-fibre-counting} bounds the number of zeros of
$\mathcal D_{r+1}P_{\mathbf j}$. The coefficient cutoffs also leave only
finitely many integer vectors, so all sums over vectors and critical points
are finite. The next proposition reduces the count to exponential sums
over these points.

\begin{proposition}\label{prop:bp52-value-poisson}
Let $r\in\mathcal R_{\boldsymbol\lambda}$ satisfy $r\le n-3$, $L_{r+2}\ge1$, and
$1\le K\le\widetilde{L}_r^{-1}$.
For the function $a$ chosen above, we have
\begin{equation}\label{eq:bp52-value-poisson}
\begin{aligned}
\#\mathcal T^{(r)}
\lesssim_\gamma{}\frac{R^{O(\beta)}}K\sum_{|d|\lesssim K}
\Biggl|\sum_{\mathbf j_{\ge r+1}\in\mathbb Z^{n-r}}
\sum_{\substack{s\in J_0\\j_{r+1}+H_{\mathbf j}'(s)=0}}
 a(s,\mathbf j_{\ge r+2}) e\bigl(d(j_{r+1}s+H_{\mathbf j}(s))\bigr)\Biggr|.
\end{aligned}
\end{equation}
\end{proposition}

\begin{proof}
We use the periodized Poisson summation identity underlying
\cite[Lemma~3.1(ii)]{CSST2026}, as in the critical-value expansion in
\cite[Lemma~7.2]{CSST2026}.
Here we choose a majorant with compactly supported Fourier transform,
so that the Fourier expansion is finite and has no truncation error.
Choose an even, nonnegative $b\in\mathcal S(\mathbb R)$ such that
\[
b(x)\ge1\quad(|x|\le1),\qquad
\operatorname{supp}\widehat b\subset[-1,1],\qquad
\|\widehat b\|_\infty\le\|b\|_1\lesssim1.
\]
Such a function is obtained by taking the squared modulus of the inverse
Fourier transform of a nonnegative smooth function supported in a
sufficiently small interval about zero, and multiplying by a suitable
constant. This choice is independent of all scale parameters.
With $\widehat b(\xi)=\int_{\mathbb R}b(x)e(-x\xi)\,dx$, Poisson summation gives,
for every $v\in\mathbb R$, $m\in\mathbb Z$, and $K\ge1$,
\[
\sum_{j_r\in\mathbb Z}b\bigl(K(j_r+v)-m\bigr)
=\frac1K\sum_{d\in\mathbb Z}\widehat b(d/K)e\bigl(dv-dm/K\bigr)
=\frac1K\sum_{|d|\le K}\widehat b(d/K)e\bigl(dv-dm/K\bigr).
\]
The series on the left converges absolutely, and the last sum is finite.

Every vector counted by $\mathcal T^{(r)}$ has a critical point $s\in I^*$
at which $a\ge1$. Choose $C\ge1$, depending only on the curve, so that
at each such point
\[
|j_r+j_{r+1}s+H_{\mathbf j}(s)|
\le C R^{C\beta}\widetilde L_r.
\]
Since $K\widetilde L_r\le1$, every $x$ satisfying
$|x|\le C R^{C\beta}\widetilde L_r$ obeys
\[
1\le\sum_{\substack{m\in\mathbb Z\\|m|\le C R^{C\beta}+1}}b(Kx-m).
\]
Indeed, one may take $m$ to be an integer nearest to $Kx$, for which
$|Kx-m|\le1/2$ and $|m|\le C R^{C\beta}+1$.
Dropping the coefficient cutoff on $j_r$ and using $a\ge0$ therefore yields
\[
\begin{aligned}
\#\mathcal T^{(r)}\le{}&
\sum_{\substack{m\in\mathbb Z\\|m|\le C R^{C\beta}+1}}
\sum_{\mathbf j_{\ge r+1}\in\mathbb Z^{n-r}}
\sum_{\substack{s\in J_0\\j_{r+1}+H_{\mathbf j}'(s)=0}}
 a(s,\mathbf j_{\ge r+2})\\
&\qquad\times\sum_{j_r\in\mathbb Z}
 b\bigl(K(j_r+j_{r+1}s+H_{\mathbf j}(s))-m\bigr).
\end{aligned}
\]
Apply the Poisson identity with $v=j_{r+1}s+H_{\mathbf j}(s)$.
The factor $\widehat b(d/K)e(-dm/K)$ is independent of the coefficient
vector and the critical point. Thus we sum over these vectors and points
before taking absolute values. Since there are $O_\gamma(R^{C\beta})$
choices of $m$ and $|\widehat b(d/K)e(-dm/K)|\lesssim1$, we obtain
\[
\#\mathcal T^{(r)}
\lesssim_\gamma\frac{R^{O(\beta)}}K\sum_{|d|\le K}
\Biggl|\sum_{\mathbf j_{\ge r+1}\in\mathbb Z^{n-r}}
\sum_{\substack{s\in J_0\\j_{r+1}+H_{\mathbf j}'(s)=0}}
 a(s,\mathbf j_{\ge r+2})
 e\bigl(d(j_{r+1}s+H_{\mathbf j}(s))\bigr)\Biggr|.
\]
This proves \eqref{eq:bp52-value-poisson}.
\end{proof}

We now estimate the exponential sums in \eqref{eq:bp52-value-poisson}.
After Poisson summation in $j_r$, the remaining coefficient vector is
$\mathbf j_{\ge r+1}$. The support implication above gives
\[
\mathbf j_{\ge r+1}\notin\widetilde{\mathcal T}^{(r+1)},\qquad
j_{r+1}+H_{\mathbf j}'(s)=0
\quad\Longrightarrow\quad a(s,\mathbf j_{\ge r+2})=0.
\]
By the proof of item~1 of Lemma~\ref{lem:n-fibre-counting}, for each
$\mathbf j_{\ge r+1}\in\widetilde{\mathcal T}^{(r+1)}$,
\[
\#\left\{s\in J_0:
 j_{r+1}+H_{\mathbf j}'(s)=0,\quad
 a(s,\mathbf j_{\ge r+2})\ne0\right\}
\le n-r-1.
\]
Consequently, for every $d\in\mathbb Z$,
\[
\begin{aligned}
&\left|
\sum_{\mathbf j_{\ge r+1}\in\mathbb Z^{n-r}}
\sum_{\substack{s\in J_0\\j_{r+1}+H_{\mathbf j}'(s)=0}}
a(s,\mathbf j_{\ge r+2})
e\bigl(d(j_{r+1}s+H_{\mathbf j}(s))\bigr)
\right|\\
&\quad=
\left|
\sum_{\mathbf j_{\ge r+1}\in\widetilde{\mathcal T}^{(r+1)}}
\sum_{\substack{s\in J_0\\j_{r+1}+H_{\mathbf j}'(s)=0}}
a(s,\mathbf j_{\ge r+2})
e\bigl(d(j_{r+1}s+H_{\mathbf j}(s))\bigr)
\right|\\
&\quad\le
\sum_{\mathbf j_{\ge r+1}\in\widetilde{\mathcal T}^{(r+1)}}
\underbrace{\sum_{\substack{s\in J_0\\j_{r+1}+H_{\mathbf j}'(s)=0}}
a(s,\mathbf j_{\ge r+2})}_{\lesssim_\gamma 1}
\lesssim_\gamma\#\widetilde{\mathcal T}^{(r+1)}.
\end{aligned}
\]
We call the frequencies $|d|\lesssim1+L_{r+3}^2/L_{r+2}^3$
the \emph{low frequencies}. There are
$O(1+L_{r+3}^2/L_{r+2}^3)$ such integers $d$.
Summing the preceding bound over these frequencies and multiplying by the
factor $R^{O(\beta)}/K$ in \eqref{eq:bp52-value-poisson}, we obtain a contribution of at most
\[
R^{O(\beta)}\frac{1+L_{r+3}^2/L_{r+2}^3}{K}\#\widetilde{\mathcal T}^{(r+1)}.
\]
Negative frequencies are conjugates of positive ones. After treating the
low frequencies, it therefore suffices to prove the following inequality:
\[
\begin{aligned}
&\sum_{1+L_{r+3}^2/L_{r+2}^3\lesssim d\lesssim K}
\Biggl|
\sum_{\mathbf j_{\ge r+1}\in\mathbb Z^{n-r}}
\sum_{\substack{s\in J_0\\j_{r+1}+H_{\mathbf j}'(s)=0}}
a(s,\mathbf j_{\ge r+2})
e\bigl(d(j_{r+1}s+H_{\mathbf j}(s))\bigr)
\Biggr|\\
&\qquad\lesssim_\gamma
1+\sqrt{L_{r+2}}\left(\prod_{k=r+2}^nL_k\right)
\sum_{\substack{1+L_{r+3}^2/L_{r+2}^3\lesssim D\lesssim K\\
D\text{ dyadic}}}
D^{-1/2}\mathfrak R_{\mathcal C_{r,n-r-1}}(L_{r+2}^{-1},D).
\end{aligned}
\]
Indeed, multiplying this bound by $R^{O(\beta)}/K$ gives the remaining
contribution in Proposition~\ref{lem:bp36-B}. The constant term is absorbed
by the low-frequency bound, since $\#\widetilde{\mathcal T}^{(r+1)}\ge1$.

\bigskip

\subsection*{Step 2. Poisson summation in \texorpdfstring{$j_{r+1}$}{j(r+1)}}
For $d\ge1$ and $p\in\mathbb Z$, write $s_0=p/d$ and define
\begin{equation}\label{eq:bp52-first-poisson}
I_{d,p}(\mathbf j_{\ge r+2})
=\int a(s,\mathbf j_{\ge r+2})\cdot |H_{\mathbf j}''(s)|\cdot
e\bigl(d[H_{\mathbf j}(s)-(s-s_0)H_{\mathbf j}'(s)]\bigr)\,ds.
\end{equation}

\begin{proposition}\label{prop:bp52-first-poisson}
For every integer $d\ge1$, 
\[
\begin{aligned}
\sum_{\mathbf j_{\ge r+1}\in\mathbb Z^{n-r}}
\sum_{\substack{s\in J_0,\ a(s,\mathbf j_{\ge r+2})\ne0\\j_{r+1}+H_{\mathbf j}'(s)=0}}
a(s,\mathbf j_{\ge r+2})
e\bigl(d(j_{r+1}s+H_{\mathbf j}(s))\bigr)=\sum_{p\in\mathbb Z}
\sum_{\mathbf j_{\ge r+2}\in\mathbb Z^{n-r-1}}
I_{d,p}(\mathbf j_{\ge r+2}).
\end{aligned}
\]

\end{proposition}

\begin{proof}
Fix a choice of $\mathbf j_{\ge r+2}$.
By construction, $a(\cdot,\mathbf{j}_{\ge r+2})\in C_c^\infty(\operatorname{int}J_0)$, and
\eqref{eq:bp52-derivative-bounds} gives $|H_{\mathbf j}''|\asymp_\gamma L_{r+2}$
on its support. For $x\in\mathbb R$, let
\[
F_{\mathbf j}(x):=
\sum_{\substack{s\in J_0,\ a(s,\mathbf{j}_{\ge r+2})\neq 0\\x+H_{\mathbf j}'(s)=0}}
a(s,\mathbf{j}_{\ge r+2})e\bigl(d(xs+H_{\mathbf j}(s))\bigr).
\]
Poisson summation gives
$\sum_{j_{r+1}\in\mathbb Z}F_{\mathbf j}(j_{r+1})
=\sum_{p\in\mathbb Z}\widehat F_{\mathbf j}(p)$. We partition the domain of $s$ into finitely many branches so that on each branch, $s\mapsto-H_{\mathbf j}'(s)$ is invertible. On each branch, the
change of variables $x=-H_{\mathbf j}'(s)$, with $|dx|=|H_{\mathbf j}''(s)|\,ds$, yields
\[
\begin{aligned}
\widehat F_{\mathbf j}(p)
&=\int\sum_{\substack{s\in J_0,\ a(s,\mathbf{j}_{\ge r+2})\neq 0\\x+H_{\mathbf j}'(s)=0}}
a(s,\mathbf{j}_{\ge r+2})e\bigl(d(xs+H_{\mathbf j}(s))-px\bigr)dx\\
&=\int a(s,\mathbf{j}_{\ge r+2})\cdot|H_{\mathbf j}''(s)|\cdot
 e\bigl(d(H_{\mathbf j}(s)-sH_{\mathbf j}'(s))+pH_{\mathbf j}'(s)\bigr)\,ds\\
&=I_{d,p}(\mathbf j_{\ge r+2}),
\end{aligned}
\]
where the last equality uses $s_0=p/d$.
Summing over $p, \mathbf j_{\ge r+2}$ proves the
identity.
\end{proof}

The phase inside the square brackets in \eqref{eq:bp52-first-poisson}
has derivative $-(s-s_0)H_{\mathbf j}''(s)$.
Thus its only possible critical point on the support of $a$ is $s=s_0$.
When present, its critical value is $H_{\mathbf j}(s_0)$ and its second
derivative is $-H_{\mathbf j}''(s_0)$.
These identities explain the two factors extracted in the next proposition.
They are the phase identities underlying \cite[Lemma 7.5]{CSST2026}.

\bigskip

\subsection*{Step 3. Uniform stationary phase}
We first estimate the integrals for which $s_0=p/d\in J_1$.
Recalling \eqref{Hj}, one has
\[H_{\mathbf j}^{(k)}(s)=\sum_{\nu=r+2}^n j_\nu\partial_s^k[\mathcal D_r\gamma_\nu(t(s))].\]
Here $\partial_s$ differentiates the
full composition with $t=t(s)$.

For fixed $s_0$, introduce the normalized derivative coordinates
\begin{equation}\label{derivativecoord}
    z_k=\frac{H_{\mathbf j}^{(k)}(s_0)}{L_{r+k}}
\quad(2\le k\le n-r),\qquad
\mathbf z=(z_2,\ldots,z_{n-r}),
\end{equation}

the diagonal matrix
\[\mathsf D=\text{diag}(L_{r+2},\dots,L_n), \]
and the matrix
\[
\mathsf M=\left(\left.\partial_s^k\bigl[\mathcal D_r\gamma_\nu(t(s))\bigr]
\right|_{s=s_0}\right)_{\substack{ 2\le k\le n-r\\ r+2\le\nu\le n}}.
\]
One has
\begin{equation}\label{changeofcoord}
     \mathsf D\mathbf z=\mathsf M\mathbf{j}_{\ge r+2}. 
\end{equation} 
Here, we view $\mathbf z$ and $\mj_{\ge r+2}$ as column vectors.

By the nondegeneracy of $\mathcal C_{r,n-r}$
in Proposition~\ref{prop:n-projected-curves}, $\mathsf M$ is  invertible with $\mathsf M^{-1}$ uniformly bounded on $J_1$.
Thus the vector $\mathbf j_{\ge r+2}$ is a linear function of $\mathbf z$:
\[ \mathbf j_{\ge r+2}=\mathsf M^{-1}\mathsf D\mathbf{z}. \]
and
\begin{equation}\label{bounddj}
|\partial_{z_k}\mathbf j_{\ge r+2}|=|\mathsf M^{-1}\mathsf D(\partial_{z_k}\mathbf z)|\lesssim_\gamma L_{r+k}
\quad(2\le k\le n-r).
\end{equation}

The next proposition separates the oscillation at $s_0$ from a smooth
amplitude whose bounds are uniform in the parameters.

\begin{proposition}\label{prop:bp52-stationary-phase}
Let $d\gtrsim1+L_{r+3}^2/L_{r+2}^3$ and $p\in\mathbb Z$ satisfy
$s_0=p/d\in J_1$. There is a smooth function $B_{d,s_0}$ such that
\begin{equation}\label{eq:bp52-amplitude}
I_{d,p}(\mathbf j_{\ge r+2})
=\sqrt{\frac{L_{r+2}}{d}}\,
e\bigl(dH_{\mathbf j}(s_0)\bigr)B_{d,s_0}(\mathbf z),
\end{equation}
where, for every multi-index $\alpha$ and every $L>0$,
\begin{equation}\label{eq:bp52-amplitude-bounds}
\sup_{\mathbf z}(1+|\mathbf z|)^L
|\partial_{\mathbf z}^{\alpha}B_{d,s_0}(\mathbf z)|
\lesssim_{\alpha,L,\gamma}1
\end{equation}
uniformly in $s_0,d,\delta$ and the admissible scales.

\end{proposition}

\begin{proof}
Let $s_0=p/d$ be fixed.
We regard $\mathbf j_{\ge r+2}$ as the linear function of $\mathbf z$
determined above. We use the change of variables
\[
s=s_0+\frac{L_{r+2}}{L_{r+3}}y
\]
and write the normalized phase and amplitude as

\begin{align}
\label{taylorphi}\Phi_{\mathbf z}(y)
&:=\frac{L_{r+3}^2}{L_{r+2}^3}
\bigl(H_{\mathbf j}(s)-(s-s_0)H_{\mathbf j}'(s)-H_{\mathbf j}(s_0)\bigr)=-\frac1{L_{r+2}}\int_0^y
vH_{\mathbf j}''\!\left(s_0+\frac{L_{r+2}}{L_{r+3}}v\right)\,dv,\\
\widetilde a(y,\mathbf z)
&:=a(s,\mathbf j_{\ge r+2})\frac{|H_{\mathbf j}''(s)|}{L_{r+2}}.
\end{align}

Then \eqref{eq:bp52-first-poisson} becomes
\begin{equation}\label{Idp}
    I_{d,p}(\mathbf j_{\ge r+2})
=\frac{L_{r+2}^2}{L_{r+3}}e\bigl(dH_{\mathbf j}(s_0)\bigr)
\int\widetilde a(y,\mathbf z)
 e\!\left(\frac{dL_{r+2}^3}{L_{r+3}^2}\Phi_{\mathbf z}(y)\right)\,dy.
\end{equation}

In particular, \eqref{eq:bp52-amplitude} holds exactly with
\[
B_{d,s_0}(\mathbf z)
=\left(\frac{dL_{r+2}^3}{L_{r+3}^2}\right)^{1/2}
\int\widetilde a(y,\mathbf z)
 e\!\left(\frac{dL_{r+2}^3}{L_{r+3}^2}\Phi_{\mathbf z}(y)\right)\,dy.
\]
We prove \eqref{eq:bp52-amplitude-bounds} for this expression.
\medskip

$\bullet$ \textbf{The contribution near $y=0$.}
Choose a smooth cutoff $\chi$ equal to one on $|y|\le c$ and supported
in $|y|\le2c$, where $c>0$ is sufficiently small in terms of the curve.

Direct computation gives
\begin{align}
\Phi_{\mathbf z}(0)&=\partial_y\Phi_{\mathbf z}(0)=0,\\ \label{dphi}
\partial_y\Phi_{\mathbf z}(y)&=-y\frac{H_{\mathbf j}''(s)}{L_{r+2}},\\ \label{d2phi}
\partial_y^2\Phi_{\mathbf z}(y)
&=-\frac{H_{\mathbf j}''(s)}{L_{r+2}}
-y\frac{H_{\mathbf j}'''(s)}{L_{r+3}}.
\end{align}

The numbers $L_i$ are nondecreasing, and the ratios $L_{i+1}/L_i$ are nonincreasing. Thus
\begin{equation}\label{eq:bp52-scale-control}
\begin{gathered}
1\le L_{r+2}\le\cdots\le L_n=\delta^{-1},\\
\frac{L_{r+k}}{L_{r+2}}
=\prod_{i=r+2}^{r+k-1}\frac{L_{i+1}}{L_i}
\le\left(\frac{L_{r+3}}{L_{r+2}}\right)^{k-2}
\qquad(2\le k\le n-r).
\end{gathered}
\end{equation}
Taking $k=n-r$ gives
\begin{equation}\label{eq:bp52-last-scale}
\delta^{-1}\left(\frac{L_{r+2}}{L_{r+3}}\right)^{n-r-2}
\le L_{r+2}.
\end{equation}
For $2\le k\le l\le n-r$, the same ratio comparison gives
\begin{equation}\label{twoscalingineq}
    \frac{L_{r+l}}{L_{r+k}}
\left(\frac{L_{r+2}}{L_{r+3}}\right)^{l-k}\le1,
\qquad
\frac{\delta^{-1}}{L_{r+k}}
\left(\frac{L_{r+2}}{L_{r+3}}\right)^{n-r-k}\le1.
\end{equation}

For $2\le k\le n-r$, Taylor-expand the function
$H_{\mathbf j}^{(k)}$ about $s$, evaluating it at $s_0$:
\[
H_{\mathbf j}^{(k)}(s_0)
=\sum_{m=0}^{n-r-k}\frac{(s_0-s)^m}{m!}
 H_{\mathbf j}^{(k+m)}(s)
+O_\gamma\!\left(\delta^{-1}|s_0-s|^{n-r-k}\right).
\]
Here the remainder uses the  $O_\gamma(\delta^{-1})$ bound for
higher derivatives. If $\widetilde a(y,\mathbf z)\ne0$, then
\eqref{eq:bp52-derivative-bounds} holds at $s$. Dividing the last display
by $L_{r+k}$, using
$z_k=H_{\mathbf j}^{(k)}(s_0)/L_{r+k}$ and $s_0-s=-\frac{L_{r+2}}{L_{r+3}}y$, and applying \eqref{twoscalingineq}, we obtain
\begin{equation}\label{boundzk}
    |z_k|
\lesssim_\gamma
\sum_{m=0}^{n-r-k}\frac{L_{r+k+m}}{L_{r+k}}\left( \frac{L_{r+2}}{L_{r+3}} \right)^m|y|^m+\frac{\de^{-1}}{L_{r+k}}\left(\frac{L_{r+2}}{L_{r+3}}\right)^{n-r-k}|y|^{n-r-k}
\lesssim_\gamma1
\qquad (|y|\le2c).
\end{equation}

Consequently,
\begin{equation}\label{zbounded}
    \chi(y)\widetilde a(y,\mathbf z)\ne0
\quad\Longrightarrow\quad |\mathbf z|\lesssim_\gamma1.
\end{equation}

\eqref{dphi} shows that $y=0$ is a critical point of the phase function.
By \eqref{eq:bp52-derivative-bounds}, $|H_{\mathbf j}''(s)|\asymp_\gamma L_{r+2}$ and $|H_{\mathbf j}'''(s)|\lesssim L_{r+3}$ on this support. If $c$ is chosen sufficiently small,
by \eqref{d2phi}, one sees that
\begin{equation}\label{uniformnondeg}
    |\partial_y^2\Phi_{\mathbf z}(y)|\asymp_\gamma1,
\end{equation} 
on the support of $\chi\widetilde{a}$.
Hence $y=0$ is the only stationary point of the phase function $\Phi_{\mathbf z}(y)$ on the support of $\chi\widetilde{a}$, and it is uniformly nondegenerate in the sense of \eqref{uniformnondeg}.

\medskip
The next goal is to prove
\begin{equation}\label{boundtildea}
|\partial_y^q\partial_{\mathbf z}^{\alpha}\widetilde a(y,\mathbf z)| \lesssim_{q,\alpha,\gamma}1,
\end{equation}
when $(y,\mathbf z)$ lies in the support of $\chi \widetilde{a}$. For simplicity, we write \[h:=\frac{L_{r+2}}{L_{r+3}}(\le 1).\] 
Set
\[
b(y,\mathbf z)=\frac{H_{\mathbf j}''(s_0+hy)}{L_{r+2}},
\]
so that $\widetilde a=a|b|$. Here, $\mathbf j_{\ge r+2}=\mathbf j_{\ge r+2}(\mathbf z)$ is a function of $\mathbf z$ as in \eqref{changeofcoord}. We also remark $b\neq 0$ in the support of $\chi\wt a$, so $\partial_y^q\partial_{\mathbf z}^\alpha |b|=\pm \partial_y^q\partial_{\mathbf z}^\alpha b$.

Recall that the majorant $a$ is chosen in \eqref{defa} as a finite product of fixed smooth cutoffs in the normalized
variables
\[
s,\qquad \delta\mathbf j_{\ge r+2},\qquad
\delta H_{\mathbf j}'(s),\qquad
\frac{(\mathcal D_{r+\ell}P_{\mathbf j})(t(s))}{L_{r+\ell}}
\quad(2\le\ell\le n-r),
\]
so we just need to prove $O(1)$ upper bounds when $\partial_y^q\partial_{\mathbf z}^\alpha$ is applied to these variables and applied to $b$. 
\medskip

\begin{enumerate}
    \item It is easy to see $|\partial_y^q\partial_{\mathbf z}^\alpha s|\le h^q\le 1$, since $s=s_0+hy$ and $s$ is independent of $\mathbf z$.\\

    \item Since $\mathbf j_{\ge r+2}$ is linear in $\mathbf z$ and is independent of $y$, only $q=0$ and $|\alpha|\le 1$ contribute. By \eqref{bounddj}, one has $|\partial_y^q\partial_{\mathbf z}^\alpha (\de\mathbf j_{\ge r+2})|\lesssim \de\max_j L_j\le 1$.\\

    \item 
    A direct computation gives
    $|\partial_y^q\partial_{\mathbf z}^\alpha (\de H_{\mathbf j}'(s))|\lesssim \de (|\mj_{\ge r+2}|+|\partial_{\mathbf z}\mj_{\ge r+2}| )\lesssim 1  $. Here, we only need to consider $|\alpha|\le 1$ since $H_\mj'$ is linear in $\mj$ and hence in $\mathbf z$. \\

     \item Denote $F_\ell(s,\mathbf z):=(\mathcal D_{r+\ell}P_{\mathbf j})(t(s))$. We want to prove
     \begin{equation}\label{boundFl}
         | \partial_y^q\partial_{z_k}^\alpha F_\ell |\lesssim L_{r+\ell}
     \end{equation}
     
     By \eqref{eq:n-normalized-derivatives},
     \[ \frac{d}{ds}F_\ell(s,\mathbf z)=c_\ell(s)F_{\ell+1}(s,\mathbf z)\]
     with $|c_\ell|\asymp 1$. Also, \eqref{Hj'} implies
     \[ F_1=j_{r+1}+H'_{\mathbf{j}}, \]
     so
     \[ H''_{\mathbf j}=\frac{d}{ds}F_1=c_1 F_2. \]
     Consequently,
     \[ F_2=c_1^{-1}H''_{\mathbf j},\ \ \ F_3=c_2^{-1}\frac{d}{ds}F_2=c_2^{-1}(c_1^{-1}H''_{\mathbf j})'. \]
     
Iterating this process yields 
\begin{equation}\label{Fl}
    F_\ell(s,\mathbf z)=\sum_{v=2}^{\ell}a_{\ell,v}(s)H^{(v)}_{\mathbf j}(s), 
\end{equation} 
where all $a_{\ell,v}$ are bounded. In other words, $F_\ell$ is a lower-triangular combination of $H_{\mathbf j}^{(v)}$ ($2\le v\le \ell$). Since $F_\ell$ is linear in $\mathbf j_{\ge r+2}$ and hence linear in $\mathbf z$, we just need to show
\[ |\partial_y^q F_\ell|\lesssim L_{r+\ell},\ \ \  |\partial_y^q \partial_{z_k} F_\ell|\lesssim L_{r+\ell}. \]

To prove the first one, we observe that for $2\le v\le \ell$,
\begin{equation}\label{derivativebounddH0}
    |\partial_y^qH^{(v)}_{\mathbf j}(s)|=|h^q H^{(v+q)}_{\mathbf j}(s)|\lesssim h^qL_{\min\{r+v+q,n\}}\lesssim L_{r+v}. 
\end{equation} 
The last inequality uses \eqref{twoscalingineq}. Using \eqref{Fl}, we see that $|\partial_y^qF_\ell|\lesssim L_{r+\ell}$.
Next, we prove the second one. It suffices to prove
\begin{equation}\label{derivativeboundH}
|\partial_y^q\partial_{z_k}H^{(v)}_{\mathbf j}(s)|\lesssim L_{r+v}.
\end{equation}

Recalling \eqref{eq:bp52-derivative-bounds} and\eqref{bounddj} that $|\mj_{\ge r+2}|\lesssim \de^{-1}$ and $|\partial_{z_k}\mj_{\ge r+2}|\lesssim \de^{-1}$ on the support of $\chi\widetilde{a}$, by the definition \eqref{Hj}, one sees that
\begin{equation}\label{trivialbound}
|\partial_y^q\partial_{z_k}H^{(v)}_\mj(s)|\lesssim |\mj_{\ge r+2}|+| \partial_{z_k}\mj_{\ge r+2}|\lesssim \de^{-1}.
\end{equation}

When $v=n-r$, \eqref{derivativeboundH} holds, so we just need to consider $v\le n-r-1$.
By Taylor's expansion,
\begin{equation}\label{taylorHv}
     H^{(v)}_{\mathbf j}(s)=\sum_{m=0}^{n-r-v-1}H^{(v+m)}_\mj(s_0)\frac{(hy)^m}{m!}+\frac{(hy)^{n-r-v}}{(n-r-v-1)!}\int_{0}^1(1-\theta)^{n-r-v-1}H^{(n-r)}_\mj(s_0+\theta hy)d\theta. 
\end{equation}

First consider when $\partial_y^q\partial_{z_k}$ acts on the first term. Recalling that $z_k=\frac{H^{(k)}(s_0)}{L_{r+k}}$, we have
\[ \partial_y^q\partial_{z_k}(\text{First term})\lesssim h^{m}L_{r+k}(1+|y|)^{(m-q)_+}. \]
with $v+m=k$. By \eqref{twoscalingineq}, $h^{k-v}L_{r+k}\le L_{r+v}$, and noting $|y|\lesssim 1$, this finishes
the estimate for the first term.

For the remainder term, using \eqref{trivialbound} we have
\[ \partial_y^q\partial_{z_k}(\text{Remainder term})\lesssim h^{n-r-v}\de^{-1} (1+|y|)^{n-r-v}\lesssim L_{r+v}. \]
The last inequality is by \eqref{twoscalingineq}.\\

\item To prove $|\partial_y^q\partial_{z_k}b|\lesssim 1$, we just note that this is a special case of \eqref{derivativebounddH0} and \eqref{derivativeboundH} with $v=2$.

\end{enumerate}

This finishes the proof of \eqref{boundtildea}.
\medskip

Finally, noting
$\partial_y\Phi_{\mathbf z}=-yb$, so the bounds
for $b$ also give the required bounds for the phase function. Hence we obtain
\begin{equation}
|\partial_y^q\partial_{\mathbf z}^{\alpha}\Phi_{\mathbf z}(y)|
+|\partial_y^q\partial_{\mathbf z}^{\alpha}\widetilde a(y,\mathbf z)|
\lesssim_{q,\alpha,\gamma}1
\end{equation}
on the bounded parameter region contributing to $\chi\widetilde a$.

We now apply Lemma~\ref{lem:stationary-phase-parameters} with
\[
\Lambda=\frac{dL_{r+2}^3}{L_{r+3}^2},
\qquad \phi(y;\mathbf z)=\Phi_{\mathbf z}(y),
\qquad w(\Lambda;y,\mathbf z)=\chi(y)\widetilde a(y,\mathbf z).
\]
The assumed lower bound
$d\gtrsim1+L_{r+3}^2/L_{r+2}^3$ gives $\Lambda\gtrsim1$.
\eqref{uniformnondeg} verifies uniqueness and nondegeneracy of the
critical point. The amplitude is independent of
$\Lambda$, so its positive $\Lambda$-derivatives vanish. Lemma \ref{lem:stationary-phase-parameters} yields
\begin{equation}
\int\chi(y)\widetilde a(y,\mathbf z)
 e\!\left(\Lambda\Phi_{\mathbf z}(y)\right)\,dy=\Lambda^{-1/2}e(\Lambda\Phi_{\mathbf z}(0))w_1(\Lambda;\mathbf z).
\end{equation}
Since $\Phi_{\mathbf z}(0)=0$ and $w_1(\Lambda;\mathbf z)$ have bounded $\mathbf z$-derivatives, we have
\begin{equation}
\sup_{\mathbf z}(1+|\mathbf z|)^L
\left|\partial_{\mathbf z}^{\alpha}
\int\chi(y)\widetilde a(y,\mathbf z)
 e\!\left(\Lambda\Phi_{\mathbf z}(y)\right)\,dy\right|
\lesssim_{\alpha,L,\gamma}\Lambda^{-1/2}
=\left(\frac{dL_{r+2}^3}{L_{r+3}^2}\right)^{-1/2}.
\end{equation}
Here $(1+|\mathbf z|)^L$ is harmless because $|\mathbf z|\lesssim 1$ by \eqref{zbounded}.  Thus the above estimate gives the required bounds for \eqref{eq:bp52-amplitude-bounds} when the $y$-integral is restricted to the cutoff by $\chi$. 
\medskip

$\bullet$ \textbf{The contribution away from $y=0$.}
Write $1-\chi=\sum_{\nu\ge0}\chi_\nu$, where
\[
\operatorname{supp}\chi_\nu\subset\{|y|\asymp2^\nu c\},
\qquad |\partial_y^q\chi_\nu(y)|\lesssim_{q,\gamma}2^{-\nu q}.
\]

We record two useful estimates. Assuming $2\le v\le n-r$, then
\begin{equation}\label{usefulestimate}
|\partial_y^qH^{(v)}_\mj(s)|\lesssim L_{r+v},\ \ \ |\partial_y^q\partial_{z_k} H^{(v)}_\mj(s)|\lesssim L_{r+v}(1+|y|)^{n-r-2}.
\end{equation}
The first one is just \eqref{derivativebounddH0}; the second one follows the same proof as \eqref{derivativeboundH} in which we also retain $y$ variable that contributes an additional factor $\lesssim (1+|y|)^{n-r-v}\le (1+|y|)^{n-r-2}$.

For $q\ge2$, direct differentiation gives
\[
\partial_y^q\Phi_{\mathbf z}(y)
=-\frac1{L_{r+2}}\left[
(q-1)\left(\frac{L_{r+2}}{L_{r+3}}\right)^{q-2}H_{\mathbf j}^{(q)}(s)
+y\left(\frac{L_{r+2}}{L_{r+3}}\right)^{q-1}H_{\mathbf j}^{(q+1)}(s)
\right].
\]
Hence, on the support of $\chi_\nu\widetilde a$,
\[
|\partial_y\Phi_{\mathbf z}|\asymp_\gamma2^\nu,
\qquad |\partial_y^q\Phi_{\mathbf z}|\lesssim_{q,\gamma}2^\nu\quad(q\ge2),
\qquad
\left|\partial_y^q\frac1{\partial_y\Phi_{\mathbf z}}\right|
\lesssim_{q,\gamma}2^{-\nu}\quad(q\ge0).
\]
To control the size of $\mathbf z$, expand each $H_{\mathbf j}^{(k)}(s_0)$
about $s$ as in \eqref{boundzk}. For $2\le k\le n-r$,
\[
|z_k|\lesssim_\gamma
\sum_{m=0}^{n-r-k}\frac{L_{r+k+m}}{L_{r+k}}
\left(\frac{L_{r+2}}{L_{r+3}}\right)^{m}|y|^{m}
+\frac{\delta^{-1}}{L_{r+k}}
\left(\frac{L_{r+2}}{L_{r+3}}\right)^{n-r-k}|y|^{n-r-k}
\lesssim_\gamma(1+|y|)^{n-r-k}.
\]
In particular,
\begin{equation}\label{boundofz}
    \chi_\nu(y)\widetilde a(y,\mathbf z)\ne0
\quad\Longrightarrow\quad
1+|\mathbf z|\lesssim_\gamma2^{\nu(n-r-2)}.
\end{equation}

The next goal is to prove
\begin{equation}
    |\partial_y^q\partial_{\mathbf z}^\alpha\widetilde{a}(y,\mathbf z)|\lesssim_{q,\alpha,\gamma} 2^{\nu(n-r-2)|\alpha|}. 
\end{equation}
The proof follows the same strategy when proving \eqref{boundtildea}; the only additional thing is to take consideration the effect of the $y$ variable.
Using the previous notation, one has $\wt a=a|b|$. 
One needs to estimate the action of $\partial_y^q\partial_{\mathbf z}^\alpha$ on
\[
s,\qquad \delta\mathbf j_{\ge r+2},\qquad
\delta H_{\mathbf j}'(s),\qquad
\frac{(\mathcal D_{r+\ell}P_{\mathbf j})(t(s))}{L_{r+\ell}}
\quad(2\le\ell\le n-r),\qquad b.
\]
Follow the same proof as in the previous part, we see that the action of $\partial_y^q\partial_{\mathbf z}^\alpha$ on $s, \delta\mj_{\ge r+2}$ and $\de H_\mj'(s)$ results in an $O(1)$ bound. Following the proof strategy of \eqref{boundFl} together with \eqref{usefulestimate}, one has
\[ |\partial_y^q\partial_{\mathbf z}^\alpha\frac{(\mathcal D_{r+\ell}P_{\mathbf j})(t(s))}{L_{r+\ell}}|\lesssim (1+|y|)^{n-r-2}. \]
From \eqref{defa}, by using Leibniz rule and noting $|y|\sim 2^\nu$, one has
\[ |\partial_y^q\partial_{\mathbf z}^\alpha a|\lesssim 2^{\nu (n-r-2)|\alpha| } \]
in the support of $\chi_\nu \wt a$. One can similarly prove the bound
\[ |\partial_y^q\partial_{\mathbf z}^\alpha b|\lesssim 2^{\nu(n-r-2)|\alpha|}. \]
Hence, using $\widetilde a=a |b|$ and $b\neq 0$ in the support of $\chi_\nu \wt a$, one shows that
\[
|\partial_y^q\partial_{\mathbf z}^{\alpha}
 \widetilde a(y,\mathbf z)|
\lesssim_{q,\alpha,\gamma}
2^{\nu(n-r-2)|\alpha|}.
\]

One can also show
\begin{align}
|\partial_y^q\partial_{z_k}\Phi_{\mathbf z}(y)|
&\lesssim_{q,\gamma}2^{\nu(n-r)}.
\end{align}
(Since $\Phi_{\mathbf z}$ is linear in $\mathbf z$, all
of its $\mathbf z$-derivatives of order at least two vanish.) When $q\ge 1$, the estimate follows from $\partial_y\Phi_{\mathbf z}=-yb$ and the derivative bounds on $b$. When $q=0$, we use the expression \eqref{taylorphi} to see that it suffices to prove for any $|s'-s|\lesssim |hy|$ that
\begin{equation}\label{want}
|\partial_{z_k}H_{\mj}''(s')|\lesssim L_{r+2} |y|^{n-r-2}. 
\end{equation} 
Using the Taylor's expansion \eqref{taylorHv} with $v=2$, the follow-up estimates give the desired \eqref{want}.

Set
\[
\Lambda:=\frac{dL_{r+2}^3}{L_{r+3}^2}\gtrsim1.
\]
Leibniz rule gives
\begin{equation}\label{eqqq1}
    \partial_{\mathbf z}^{\alpha}
 \bigl((\chi_\nu\wt a) e(\Lambda\Phi_{\mathbf z})\bigr)
=e(\Lambda\Phi_{\mathbf z})
 \sum_{\beta\le\alpha}
 \binom{\alpha}{\beta}
 (2\pi i\Lambda)^{|\beta|}
 \partial_{\mathbf z}^{\alpha-\beta}(\chi_\nu \wt a)
 \prod_{k=2}^{n-r}
 \bigl(\partial_{z_k}\Phi_{\mathbf z}\bigr)^{\beta_k}.
\end{equation}

By
the preceding bounds, the Leibniz rule in $y$, and the estimates for
$\chi_\nu$, for every $q\ge0$ we have
\begin{align}\label{eqqq2}
\left|\partial_y^q\left[
 \partial_{\mathbf z}^{\alpha-\beta}(\chi_\nu \wt a)
 \prod_{k=2}^{n-r}
 \bigl(\partial_{z_k}\Phi_{\mathbf z}\bigr)^{\beta_k}
 \right]\right|
\qquad\lesssim  
2^{\nu(n-r)|\alpha-\beta|}
2^{\nu(n-r)|\beta|}
=2^{\nu(n-r)|\alpha|}.
\end{align}
Indeed, derivatives falling on $\chi_\nu$ only improve this estimate,
while every $y$-derivative of a factor
$\partial_{z_k}\Phi_{\mathbf z}$ satisfies the same bound as the factor
itself.

For fixed $\mathbf z$, introduce the formal adjoint of the
nonstationary-phase operator,
\[
T_{\mathbf z}f:=-\partial_y\!\left(
 \frac{f}{\partial_y\Phi_{\mathbf z}}\right).
\]
The identity
\[
e(\Lambda\Phi_{\mathbf z})
=\frac{1}{2\pi i\Lambda\,\partial_y\Phi_{\mathbf z}}
 \partial_y e(\Lambda\Phi_{\mathbf z})
\]
and $N$ integrations by parts give, with no boundary terms,
\[
\int f(y)e(\Lambda\Phi_{\mathbf z}(y))\,dy
=(2\pi i\Lambda)^{-N}
 \int T_{\mathbf z}^{N}f(y)
 e(\Lambda\Phi_{\mathbf z}(y))\,dy.
\]
Each application of $T_{\mathbf z}$ gains one factor $2^{-\nu}$.  More
precisely, the reciprocal-phase bounds above imply
\[
\begin{aligned}
|\partial_y^q(T_{\mathbf z}f)|
&\le C_q\sum_{j=0}^{q+1}|\partial_y^j f|
 \left|\partial_y^{q+1-j}
       \frac1{\partial_y\Phi_{\mathbf z}}\right|\\
&\lesssim_{q,\gamma}2^{-\nu}
 \max_{0\le j\le q+1}|\partial_y^j f|.
\end{aligned}
\]
Iterating this inequality yields
\[|T_{\mathbf z}^{N}f|
\lesssim_{N,\gamma}2^{-N\nu}
 \max_{0\le q\le N}|\partial_y^qf|.\]
 Hence, for $f$ supported in $\operatorname{supp}\chi_\nu$, whose length is
$O_\gamma(2^\nu)$, we obtain
\begin{equation}\label{eqqq3}
   |\int f(y)e(\Lambda\Phi_{\mathbf z}(y))\,dy|
\lesssim_{N,\gamma} \Lambda^{-N}2^{(1-N)\nu}
\max_{0\le q\le N}\|\partial_y^qf\|_{L^\infty_y}.
\end{equation}

Combining \eqref{eqqq1},\eqref{eqqq2},\eqref{eqqq3}, we obtain
\[
\left|\partial_{\mathbf z}^{\alpha}
\int\chi_\nu(y)\widetilde a(y,\mathbf z)
 e\!\left(\Lambda\Phi_{\mathbf z}(y)\right)\,dy\right|\lesssim_{N,\alpha,\gamma}
\Lambda^{|\alpha|-N}2^{-N\nu}2^\nu
2^{\nu(n-r)|\alpha|}.
\]

It remains to insert the weight in $\mathbf z$. By \eqref{boundofz},
$1+|\mathbf z|\lesssim_\gamma2^{\nu(n-r-2)}$.  Hence
\[
\begin{aligned}
&\sup_{\mathbf z}(1+|\mathbf z|)^L
\left|\partial_{\mathbf z}^{\alpha}\left[
\Lambda^{1/2}
\int\chi_\nu(y)\widetilde a(y,\mathbf z)
 e\!\left(\Lambda\Phi_{\mathbf z}(y)\right)\,dy
\right]\right|\\
&\qquad\lesssim_{N,\alpha,L,\gamma}
\Lambda^{|\alpha|+\frac12-N}2^{-N\nu}2^\nu
2^{\nu(n-r)(|\alpha|+L)}.
\end{aligned}
\]
Choosing  $N$ large enough, we get
\[
\begin{aligned}
\sup_{\mathbf z}(1+|\mathbf z|)^L
\left|\partial_{\mathbf z}^{\alpha}\left[
\left(\frac{dL_{r+2}^3}{L_{r+3}^2}\right)^{1/2}
\int\chi_\nu(y)\widetilde a(y,\mathbf z)
 e\!\left(\frac{dL_{r+2}^3}{L_{r+3}^2}
 \Phi_{\mathbf z}(y)\right)\,dy
\right]\right|
\lesssim_{\alpha,L,\gamma}2^{-\nu}.
\end{aligned}
\]
Summing over $\nu\ge0$ and adding the contribution near $y=0$ proves
\eqref{eq:bp52-amplitude-bounds}.
\end{proof}

\bigskip

\subsection*{Step 4. Parameters outside \texorpdfstring{$J_1$}{J1}}
The parameters $p/d\notin J_1$ give a small error, as the next proposition shows.

\begin{proposition}\label{prop:bp52-exterior}
For every $L>0$, we have
\begin{equation}\label{eq:bp52-exterior}
\sum_{d\gtrsim1+L_{r+3}^2/L_{r+2}^3}
\ \sum_{p:p/d\notin J_1}
\ \sum_{\mathbf j_{\ge r+2}}
|I_{d,p}(\mathbf j_{\ge r+2})|
\lesssim_{L,\gamma}\delta^L
\end{equation}
\end{proposition}

\begin{proof}
Fix $L>0$ and choose an integer $N>3+n(L+n+3)$.
If $s_0=p/d\notin J_1$, its distance from the support of the
cutoff is bounded below by a positive constant. \eqref{dphi} shows that
\[ |\partial_y\Phi_{\mathbf z}(y)|\sim |y|=\frac{L_{r+3}}{L_{r+2}}|s-s_0|\sim \frac{L_{r+3}}{L_{r+2}}(1+|s_0|). \]
Integrating by parts $N$ times in \eqref{Idp} gives
\[
|I_{d,p}(\mathbf j_{\ge r+2})|
\lesssim_{N,\gamma}L_{r+2}
\left(\frac{dL_{r+2}^2}{L_{r+3}}\right)^{-N}
(1+|p|/d)^{-N}.
\]
By \eqref{twoscalingineq}, one has
\[ \de^{-1}\le \left(\frac{L_{r+3}}{L_{r+2}}\right)^{n-r-1}\frac{L^2_{r+2}}{L_{r+3}}. \]
Therefore,
\[
\frac{dL_{r+2}^2}{L_{r+3}}
\gtrsim\max\left\{
\frac{L_{r+2}^2}{L_{r+3}},
\frac{L_{r+3}}{L_{r+2}}\right\}
\gtrsim\delta^{-1/(n-r)}.
\]
Since $\sum_{p\in\mathbb Z}(1+|p|/d)^{-N}\lesssim_N d$ and there are
$O_\gamma(\delta^{-n})$ admissible tails $\mathbf j_{\ge r+2}$,
we can bound the left hand side of \eqref{eq:bp52-exterior} by
\[
\begin{aligned}
&\delta^{-n}L_{r+2}
\sum_{d\gtrsim1+L_{r+3}^2/L_{r+2}^3}
d\left(\frac{dL_{r+2}^2}{L_{r+3}}\right)^{-N}\\
&\qquad\lesssim_{N,\gamma}
\delta^{-n+(N-3)/(n-r)}L_{r+2}
\left(\frac{L_{r+3}}{L_{r+2}^2}\right)^3
\sum_{d\ge1}d^{-2}\\
&\qquad\lesssim_{N,\gamma}
\delta^{(N-3)/(n-r)-n-3}
\lesssim_{L,\gamma}\delta^L.
\end{aligned}
\]
Here we used $L_{r+2}\ge1$ and $L_{r+3}\le\delta^{-1}$ to obtain
$L_{r+2}(L_{r+3}/L_{r+2}^2)^3\le\delta^{-3}$, and then used the choice of $N$.
\end{proof}

After multiplication by $R^{O(\beta)}/K$, this error is absorbed by
$R^{O(\beta)}\#\widetilde{\mathcal T}^{(r+1)}/K$, since
$\#\widetilde{\mathcal T}^{(r+1)}\ge1$ and $\delta\le1$.
It remains to prove the following inequality:
\[
\begin{aligned}
&\sum_{1+L_{r+3}^2/L_{r+2}^3\lesssim d\lesssim K}
\ \sum_{\substack{p\in\mathbb Z\\p/d\in J_1}}
\left|\sum_{\mathbf j_{\ge r+2}\in\mathbb Z^{n-r-1}}
I_{d,p}(\mathbf j_{\ge r+2})\right|\\
&\qquad\lesssim_\gamma
\sqrt{L_{r+2}}\left(\prod_{k=r+2}^nL_k\right)
\sum_{\substack{1+L_{r+3}^2/L_{r+2}^3\lesssim D\lesssim K\\
D\text{ dyadic}}}
D^{-1/2}\mathfrak R_{\mathcal C_{r,n-r-1}}(L_{r+2}^{-1},D).
\end{aligned}
\]

\bigskip

\subsection*{Step 5. Poisson summation in the remaining coefficients}
Here we use Proposition~\ref{prop:bp52-stationary-phase}
to estimate the sum over the remaining coefficients.

\begin{proposition}\label{prop:bp52-second-poisson}
Let $d\gtrsim1+L_{r+3}^2/L_{r+2}^3$ and $p\in\mathbb Z$ satisfy
$p/d\in J_1$. For every $L>0$, we have
\begin{equation}\label{eq:bp52-second-poisson}
\left|\sum_{\mathbf j_{\ge r+2}\in\mathbb Z^{n-r-1}}
I_{d,p}(\mathbf j_{\ge r+2})\right|
\lesssim_{L,\gamma}
\left(\prod_{k=r+2}^nL_k\right)
\sqrt{\frac{L_{r+2}}d}
\sum_{\mathbf a\in\mathbb Z^{n-r-1}}
\left(1+L_{r+2}
\left|\mathbf a-d\bigl(\mathcal D_r\gamma_\nu(t(p/d))\bigr)_{\nu=r+2}^n\right|\right)^{-L}.
\end{equation}
\end{proposition}

\begin{proof}
Write $s_0=p/d$ and recall the following notation.
\[
\mathsf M=
\left(\left.\partial_s^k\bigl[\mathcal D_r\gamma_\nu(t(s))\bigr]
\right|_{s=s_0}\right)_{
\substack{2\le k\le n-r\\ r+2\le\nu\le n }}.
\]
The nondegeneracy of $\mathcal C_{r,n-r}$ in
Proposition~\ref{prop:n-projected-curves} gives
\[
|\det\mathsf M|\asymp_\gamma1,
\qquad
\|\mathsf M\|+\|\mathsf M^{-1}\|\lesssim_\gamma1.
\]
The normalized derivative coordinates from \eqref{derivativecoord} give
\[
\mathbf z(\mj_{\ge r+2})
=\operatorname{diag}(L_{r+2},\ldots,L_n)^{-1}\mathsf M\mathbf \mj_{\ge r+2}.
\]
In particular, \eqref{eq:bp52-amplitude} reads
\[
I_{d,p}(\mathbf j_{\ge r+2})
=\sqrt{\frac{L_{r+2}}d}\,
 e\!\left(d\sum_{\nu=r+2}^n j_\nu \mathcal D_r\gamma_\nu(t(s_0))\right)
 B_{d,s_0}\bigl(\mathbf z(\mathbf j_{\ge r+2})\bigr).
\]

Poisson summation in the original integer coordinates gives
\[
\sum_{\mathbf j_{\ge r+2}\in\mathbb Z^{n-r-1}}
 I_{d,p}(\mathbf j_{\ge r+2})
=\sqrt{\frac{L_{r+2}}d}
\sum_{\mathbf a\in\mathbb Z^{n-r-1}}
\int_{\mathbb R^{n-r-1}}
 B_{d,s_0}\bigl(\mathbf z(\mathbf x)\bigr)
\cdot e\!\left(
-\bigl[\mathbf a-d\bigl(\mathcal D_r\gamma_\nu(t(s_0))\bigr)_{\nu=r+2}^n\bigr]
\cdot\mathbf x\right)\,d\mathbf x.
\]
To estimate these Fourier integrals, use
\[
\mathbf x=\mathsf M^{-1}
\operatorname{diag}(L_{r+2},\ldots,L_n)\mathbf z,
\qquad
 d\mathbf x=\frac{\prod_{k=r+2}^nL_k}{|\det\mathsf M|}\,d\mathbf z(\mathbf x).
\]
This Jacobian produces the factor $\prod_{k=r+2}^nL_k$.
Moreover,
\[
\boldsymbol\eta\cdot\mathbf x
=\left(\operatorname{diag}(L_{r+2},\ldots,L_n)
\mathsf M^{-T}\boldsymbol\eta\right)\cdot\mathbf z.
\]
Consequently,
\[
\begin{aligned}
\int_{\mathbb R^{n-r-1}}
 B_{d,s_0}\bigl(\mathbf z(\mathbf x)\bigr)
 e(-\boldsymbol\eta\cdot\mathbf x)\,d\mathbf x=
\frac{\prod_{k=r+2}^nL_k}{|\det\mathsf M|}
\widehat B_{d,s_0}\!\left(
\operatorname{diag}(L_{r+2},\ldots,L_n)
\mathsf M^{-T}\boldsymbol\eta\right).
\end{aligned}
\]
The derivative bounds in \eqref{eq:bp52-amplitude-bounds} imply, by
integration by parts, that
\[
\left|\widehat B_{d,s_0}(\boldsymbol\xi)\right|
=\left|\int_{\mathbb R^{n-r-1}}B_{d,s_0}(\mathbf z)
 e(-\boldsymbol\xi\cdot\mathbf z)\,d\mathbf z\right|
\lesssim_{L,\gamma}(1+|\boldsymbol\xi|)^{-L}.
\]
Since $L_{r+2}\le\cdots\le L_n$, we also have
\[
\begin{aligned}
\left|\operatorname{diag}(L_{r+2},\ldots,L_n)
\mathsf M^{-T}\boldsymbol\eta\right|^2
=\sum_{k=2}^{n-r}L_{r+k}^2
 \left|(\mathsf M^{-T}\boldsymbol\eta)_k\right|^2\ge L_{r+2}^2|\mathsf M^{-T}\boldsymbol\eta|^2
\gtrsim_\gamma L_{r+2}^2|\boldsymbol\eta|^2.
\end{aligned}
\]
Thus the smallest scale $L_{r+2}$ controls the decay in every direction,
and the Fourier integral is bounded by
\[
\left|\int_{\mathbb R^{n-r-1}}
 B_{d,s_0}\bigl(\mathbf z(\mathbf x)\bigr)
 e(-\boldsymbol\eta\cdot\mathbf x)\,d\mathbf x\right|
\lesssim_{L,\gamma}
\left(\prod_{k=r+2}^nL_k\right)
(1+L_{r+2}|\boldsymbol\eta|)^{-L}.
\]
Substituting
$\boldsymbol\eta=\mathbf a-d\bigl(\mathcal D_r\gamma_\nu(t(s_0))\bigr)_{\nu=r+2}^n$
in the Poisson formula and summing over $\mathbf a$ gives
\[
\left|\sum_{\mathbf j_{\ge r+2}\in\mathbb Z^{n-r-1}}
 I_{d,p}(\mathbf j_{\ge r+2})\right|
\lesssim_{L,\gamma}{}
\sqrt{\frac{L_{r+2}}d}
\left(\prod_{k=r+2}^nL_k\right)
\sum_{\mathbf a\in\mathbb Z^{n-r-1}}
\left(1+L_{r+2}
\left|\mathbf a-d\bigl(\mathcal D_r\gamma_\nu(t(s_0))\bigr)_{\nu=r+2}^n\right|\right)^{-L}.
\]
Since $s_0=p/d$, this is \eqref{eq:bp52-second-poisson}.
\end{proof}

The normalized derivative coordinates were used to estimate this
Fourier transform; both Poisson summations use the original
integer coefficients.

Choose $L\ge A$, where $A$ is the exponent in the definition of
$\mathfrak R$, and let $D$ be dyadic with
$1+L_{r+3}^2/L_{r+2}^3\lesssim D\lesssim K$.
Applying Proposition~\ref{prop:bp52-second-poisson} and using
$d^{-1/2}\lesssim D^{-1/2}$ for $D/2<d\le D$, we obtain
\[
\begin{aligned}
&\sum_{\substack{D/2<d\le D\\d\gtrsim1+L_{r+3}^2/L_{r+2}^3}}
\ \sum_{\substack{p\in\mathbb Z\\p/d\in J_1}}
\left|\sum_{\mathbf j_{\ge r+2}\in\mathbb Z^{n-r-1}}
I_{d,p}(\mathbf j_{\ge r+2})\right|\\
&\quad\lesssim_{L,\gamma}
\sqrt{\frac{L_{r+2}}D}\left(\prod_{k=r+2}^nL_k\right)\\
&\qquad\times
\sum_{\substack{D/2<d\le D\\d\gtrsim1+L_{r+3}^2/L_{r+2}^3}}
\ \sum_{\substack{p\in\mathbb Z\\p/d\in J_1}}
\ \sum_{\mathbf a\in\mathbb Z^{n-r-1}}
\left(1+L_{r+2}
\left|\mathbf a-d\bigl(\mathcal D_r\gamma_\nu(t(p/d))\bigr)_{\nu=r+2}^n\right|\right)^{-L}.
\end{aligned}
\]
The prefactor already has the required form. It therefore suffices to prove,
for every such $D$,
\[
\begin{aligned}
&\sum_{\substack{D/2<d\le D\\d\gtrsim1+L_{r+3}^2/L_{r+2}^3}}
\ \sum_{\substack{p\in\mathbb Z\\p/d\in J_1}}
\ \sum_{\mathbf a\in\mathbb Z^{n-r-1}}
\left(1+L_{r+2}
\left|\mathbf a-d\bigl(\mathcal D_r\gamma_\nu(t(p/d))\bigr)_{\nu=r+2}^n\right|\right)^{-L}\\
&\qquad\lesssim_{L,\gamma}
\mathfrak R_{\mathcal C_{r,n-r-1}}(L_{r+2}^{-1},D).
\end{aligned}
\]
Summing the resulting estimates over $D$ gives the inequality at the end of \textbf{Step~4}.

\bigskip

\subsection*{Step 6. The projected counting function}
We prove the inequality at the end of Step~5 by counting the integer vectors
in each distance range.

\begin{proposition}\label{prop:bp52-projected-counting}
For every dyadic $D$ with $1+L_{r+3}^2/L_{r+2}^3\lesssim D\lesssim K$
and every integer $\nu\ge0$, we have
\[
\begin{aligned}
&\sum_{\substack{D/2<d\le D\\d\gtrsim1+L_{r+3}^2/L_{r+2}^3}}
\ \sum_{\substack{p\in\mathbb Z\\p/d\in J_1}}
\#\left\{\mathbf a\in\mathbb Z^{n-r-1}:
\left|\mathbf a-d\bigl(\mathcal D_r\gamma_k(t(p/d))\bigr)_{k=r+2}^n\right|
\le2^\nu L_{r+2}^{-1}\right\}\\
&\qquad\lesssim_\gamma
(1+2^\nu L_{r+2}^{-1})
\mathcal R_{\mathcal C_{r,n-r-1}}(C2^\nu L_{r+2}^{-1},D).
\end{aligned}
\]
\end{proposition}

\begin{proof}
For each tuple counted on the left,
\[
\frac pd\in J_1,\qquad
\left|\mathbf a-d\bigl(\mathcal D_r\gamma_k(t(p/d))\bigr)_{k=r+2}^n\right|
\le2^\nu L_{r+2}^{-1}.
\]
Omitting $a_n$ and dividing by $d$ gives
\[
\operatorname{dist}\!\left(
\left(\frac pd,\frac{a_{r+2}}d,\ldots,\frac{a_{n-1}}d\right),
\mathcal C_{r,n-r-1}\right)
\le\frac{C2^\nu L_{r+2}^{-1}}d,
\]
since
\[
\mathcal C_{r,n-r-1}
=\bigl\{(s,\mathcal D_r\gamma_{r+2}(t(s)),\ldots,
\mathcal D_r\gamma_{n-1}(t(s))):s\in J_1\bigr\}.
\]
Thus there are at most
$\mathcal R_{\mathcal C_{r,n-r-1}}(C2^\nu L_{r+2}^{-1},D)$
choices for $(p,a_{r+2},\ldots,a_{n-1},d)$.
For each such choice, the remaining coordinate satisfies
\[
a_n\in\mathbb Z\cap
\left[d\mathcal D_r\gamma_n(t(p/d))-2^\nu L_{r+2}^{-1},\,
      d\mathcal D_r\gamma_n(t(p/d))+2^\nu L_{r+2}^{-1}\right],
\]
so it has $O(1+2^\nu L_{r+2}^{-1})$ possibilities.
Multiplying these two bounds proves the proposition. 
\end{proof}
\begin{remark}
    {\rm
    The above counting just ignored the $a_n$. Intuitively, we call it a projection argument since the counting problem becomes 1 dimensional lower.
    }
\end{remark}

Decomposing the distance into dyadic ranges and applying
Proposition~\ref{prop:bp52-projected-counting}, we obtain
\[
\begin{aligned}
&\sum_{\substack{D/2<d\le D\\d\gtrsim1+L_{r+3}^2/L_{r+2}^3}}
\ \sum_{\substack{p\in\mathbb Z\\p/d\in J_1}}
\ \sum_{\mathbf a\in\mathbb Z^{n-r-1}}
\left(1+L_{r+2}
\left|\mathbf a-d\bigl(\mathcal D_r\gamma_k(t(p/d))\bigr)_{k=r+2}^n\right|\right)^{-L}\\
&\quad\lesssim_{L,\gamma}
\sum_{\nu\ge0}2^{-L\nu}(1+2^\nu L_{r+2}^{-1})
\mathcal R_{\mathcal C_{r,n-r-1}}(C2^\nu L_{r+2}^{-1},D)\\
&\quad\le
\sum_{\nu\ge0}2^{-A\nu}(1+2^\nu L_{r+2}^{-1})
\mathcal R_{\mathcal C_{r,n-r-1}}(C2^\nu L_{r+2}^{-1},D)=
\mathfrak R_{\mathcal C_{r,n-r-1}}(L_{r+2}^{-1},D),
\end{aligned}
\]
where the second inequality uses $L\ge A$ and the last equality is the
definition of $\mathfrak R$.
This proves the inequality at the end of \textbf{Step~5}.
The reductions in Steps~1--5 complete the proof of
Proposition~\ref{lem:bp36-B}.

\section{Proof of Proposition~\ref{cor:intro-moment-count}}
\label{sec:sharpness}

In this section, we give a proof of Proposition~\ref{cor:intro-moment-count}.
Recall that
\[
 \gamma_n(t):=(t,t^2,\ldots,t^n),\qquad
 \Gamma_n=\gamma_n([0,1]).
\]

\subsection{Volume bound}
We first prove the volume lower bound
\[
\mathcal R_{\Gamma_n}(\delta,Q)\gtrsim_n\delta^{n-1}Q^2
\]
for $0<\delta\leq1$. 
The lower bound of Beresnevich--Vaughan--Velani--Zorin
\cite[Corollary~1.2]{BVVZ2021}, 
gives this estimate when $\delta\gtrsim Q^{-1/(n-1)}$.
For $0<\delta \lesssim Q^{-1/(n-1)}$, the pairs $(0,q)$ instead give
\[
\delta^{n-1}Q^2\lesssim_n Q
\lesssim\#\{q\in\mathbb N:Q/2<q\leq Q\}
\leq \mathcal R_{\Gamma_n}(\delta,Q).
\]
It remains to obtain the lower bounds for the other terms in
Proposition~\ref{cor:intro-moment-count}. \medskip

\bigskip

\subsection{Contribution of the origin}

Since $\gamma_n(0)=0$, every integer $Q/2<q\leq Q$ satisfies
\[
\operatorname{dist}(0,q\Gamma_n)
\leq |q\gamma_n(0)|=0\leq\delta.
\]
Thus the pairs $(0,q)$ give
\begin{equation}\label{eq:sharp-Q-lower}
\begin{aligned}
\mathcal R_{\Gamma_n}(\delta,Q)
\geq\#\{q\in\mathbb N:Q/2<q\leq Q\}\gtrsim Q.
\end{aligned}
\end{equation}
This gives the term $Q$ in Proposition~\ref{cor:intro-moment-count}. \medskip

\bigskip

\subsection{Local flatness}
We next fix $1\leq k\leq n-1$. We prove that
\[
\mathcal R_{\Gamma_n}(\delta,Q)\gtrsim_n
Q(\delta^kQ)^{k/(k+1)}
=\delta^{k^2/(k+1)}Q^{(2k+1)/(k+1)}.
\]
If $\delta^kQ$ is bounded by a sufficiently large constant depending only
on $n$, then
\[
\delta^{k^2/(k+1)}Q^{(2k+1)/(k+1)}
=Q(\delta^kQ)^{k/(k+1)}\lesssim_n Q,
\]
so \eqref{eq:sharp-Q-lower} gives the desired lower bound.
We may therefore assume that $\delta^kQ$ is larger than this constant.  Choose $c_0=c_0(n)>0$ sufficiently small.
Let $q$ be an integer with $Q/2<q\leq Q$. First choose $p_k\in\mathbb Z$ with
\begin{equation}\label{eq:sharp-pk-range}
2^{-k}q\left(\frac{c_0\delta}{Q}\right)^{k/(k+1)}
\leq p_k\leq
\left(\frac23\right)^kq\left(\frac{c_0\delta}{Q}\right)^{k/(k+1)}.
\end{equation}
The interval has length comparable to
\[
q\left(\frac{c_0\delta}{Q}\right)^{k/(k+1)}
\asymp_n(\delta^kQ)^{1/(k+1)},
\]
so it contains $\gtrsim_n(\delta^kQ)^{1/(k+1)}$ integers. 
For $2\leq j\leq k$, define
\begin{equation}\label{eq:sharp-Fjq}
F(x)=q^{-1/(j-1)}x^{j/(j-1)},\qquad x>0.
\end{equation}
Here and below, the dependence of $F$ on $j$ and $q$ is suppressed.
We choose
$p_{k-1},\ldots,p_1$ successively. To be specific, after $p_j$ has been chosen, choose every integer $p_{j-1}$ satisfying
\begin{equation}\label{eq:sharp-inverse-interval}
p_{j-1}\in F^{-1}
\bigl([p_j-c_0\delta,p_j+c_0\delta]\bigr).
\end{equation}
When $k=1$, there are no further choices, and
\eqref{eq:sharp-pk-range} directly specifies $p_1$.

\begin{lemma}\label{lem:sharp-integer-choices}
Under the assumptions above, for each $2\leq j\leq k$ and every $p_j$
arising from the choices in \eqref{eq:sharp-pk-range} and
\eqref{eq:sharp-inverse-interval},
\[
\#\Bigl(\mathbb Z\cap F^{-1}
\bigl([p_j-c_0\delta,p_j+c_0\delta]\bigr)\Bigr)
\gtrsim_n(\delta^kQ)^{1/(k+1)}.
\]
Moreover, writing $t_j=(p_j/q)^{1/j}$, every coordinate chosen in this
construction satisfies
\[
\frac13\left(\frac{c_0\delta}{Q}\right)^{1/(k+1)}
\leq t_j\leq
\frac34\left(\frac{c_0\delta}{Q}\right)^{1/(k+1)}
\qquad(1\leq j\leq k).
\]
\end{lemma}

We postpone the proof of this lemma until the end of this section.

By Lemma~\ref{lem:sharp-integer-choices} and the number of choices for $p_k$,
we obtain
$
\gtrsim_n(\delta^kQ)^{k/(k+1)}
$
distinct tuples $(p_1,\ldots,p_k)$ for each denominator $q$.
For any tuple $(p_1,\ldots,p_k)$ constructed above, set
\[
t=\frac{p_1}{q},\qquad
p=(p_1,\ldots,p_k,0,\ldots,0)\in\mathbb Z^n.
\]
We will show that
\[
|q\gamma_n(t)-p|
=\bigl|(qt,qt^2,\ldots,qt^n)-(p_1,\ldots,p_k,0,\ldots,0)\bigr|
\leq\delta.
\]
Assuming this bound, each constructed pair satisfies
\[
1\leq q\leq Q,\qquad
\operatorname{dist}(p/q,\Gamma_n)
\leq\frac{|q\gamma_n(t)-p|}{q}\leq\frac{\delta}{q},
\]
and is therefore counted by $\mathcal R_{\Gamma_n}(\delta,Q)$. Thus each $q$ contributes
$\gtrsim_n(\delta^kQ)^{k/(k+1)}$ pairs. Summing over the
$\asymp Q$ integers $q\in(Q/2,Q]$, we obtain
\[
\mathcal R_{\Gamma_n}(\delta,Q)\gtrsim_n
Q(\delta^kQ)^{k/(k+1)}
=\delta^{k^2/(k+1)}Q^{(2k+1)/(k+1)}.
\]
Together with the volume lower bound proved above and
\eqref{eq:sharp-Q-lower}, these estimates for $1\leq k\leq n-1$
yield the claimed lower bound in Proposition~\ref{cor:intro-moment-count}.

We now prove the distance bound above. It suffices to show that
\[
|qt^j-p_j|\lesssim_n c_0\delta\quad(1\leq j\leq k),
\qquad
qt^{k+1}\lesssim_n c_0\delta.
\]

For the second bound, Lemma~\ref{lem:sharp-integer-choices} gives
\[
0<t=t_1\lesssim_n\left(\frac{c_0\delta}{Q}\right)^{1/(k+1)}.
\]
Hence, using $q\asymp Q$, we obtain
\[
qt^{k+1}\lesssim_n\frac qQ c_0\delta\lesssim_n c_0\delta.
\]
We now prove the first bound, for $1\leq j\leq k$.
Since $q\asymp Q$ and $p_1,p_{j-1}\lesssim_n Q$, we have $t=p_1/q\lesssim_n1$.
Thus $x\lesssim_n Q$ for every $x$ between $qt^{j-1}$ and $p_{j-1}$, and
\[
F'(x)=\frac{j}{j-1}\left(\frac{x}{q}\right)^{1/(j-1)}\lesssim_n1.
\]
For $2\leq j\leq k$, the mean value theorem and
\eqref{eq:sharp-inverse-interval} give
\begin{equation}\label{eq:sharp-error-recursion}
\begin{aligned}
|qt^j-p_j|
&=|F(qt^{j-1})-p_j|\\
&\leq |F(qt^{j-1})-F(p_{j-1})|
+\underbrace{|F(p_{j-1})-p_j|}_{\leq c_0\delta}\lesssim_n |qt^{j-1}-p_{j-1}|+c_0\delta.
\end{aligned}
\end{equation}
Since $qt=p_1$ and $k\leq n-1$, iterating this estimate at most $n-2$ times gives
\[
|qt^j-p_j|\lesssim_n c_0\delta
\qquad(1\leq j\leq k).
\]
The two bounds give the required distance estimate when $c_0(n)$ is
sufficiently small, completing the proof of
Proposition~\ref{cor:intro-moment-count}.

\begin{proof}[Proof of Lemma~\ref{lem:sharp-integer-choices}]
We first prove the bounds for $t_j$ by descending induction,
from $j=k$ to $j=1$.
By \eqref{eq:sharp-pk-range},
\[
\frac12\left(\frac{c_0\delta}{Q}\right)^{1/(k+1)}
\leq t_k\leq
\frac23\left(\frac{c_0\delta}{Q}\right)^{1/(k+1)},
\]
so the case $j=k$ holds. Let $2\leq j\leq k$ and assume the stated bounds
hold for $t_j,\ldots,t_k$. We show that they also hold for $t_{j-1}$
whenever $p_{j-1}$ is chosen according to \eqref{eq:sharp-inverse-interval}.
Since $j\leq k\leq n-1$, $q\asymp Q$, and $0<c_0\delta\leq1$,
\[
\frac{p_j}{c_0\delta}
=\frac{qt_j^j}{c_0\delta}
\gtrsim_n\left(\frac{Q}{c_0\delta}\right)^{\frac{k+1-j}{k+1}}
\geq Q^{1/n},
\]
where $(k+1-j)/(k+1)\geq1/n$.
Taking $Q_n$ sufficiently large, uniformly in $\delta,j,k$, we obtain
\[
0<p_j-c_0\delta\asymp p_j+c_0\delta\asymp p_j.
\]
For every such choice of $p_{j-1}$, \eqref{eq:sharp-inverse-interval}
gives $|F(p_{j-1})-p_j|\leq c_0\delta$.
By computations, the definition of $F$ gives $F(p_{j-1})=qt_{j-1}^j$.
Dividing $|F(p_{j-1})-p_j|\leq c_0\delta$ by $q$ gives
\begin{equation}\label{eq:sharp-tj-stability}
\begin{aligned}
|t_{j-1}^j-t_j^j|=\frac{|F(p_{j-1})-p_j|}{q}\leq\frac{c_0\delta}{q}.
\end{aligned}
\end{equation}
We claim that
$
|t_{j-1}-t_j|
\lesssim_n\left(\frac{c_0\delta}{Q}\right)^{2/(k+1)}.
$
To prove the claim, use $t_{j-1},t_j>0$ to write
\[
\begin{aligned}
|t_{j-1}^j-t_j^j|
=|t_{j-1}-t_j|
\sum_{\ell=0}^{j-1}t_{j-1}^{\ell}t_j^{j-1-\ell}\geq |t_{j-1}-t_j|\,t_j^{j-1}.
\end{aligned}
\]
Combining this with \eqref{eq:sharp-tj-stability}
gives $|t_{j-1}-t_j|\leq c_0\delta/(qt_j^{j-1})$.
Next, using $q\asymp Q$ and the assumed lower bound for $t_j$, we obtain
\[
\begin{aligned}
\frac{c_0\delta}{qt_j^{j-1}}
\lesssim_n\frac{c_0\delta}{Q}
\left(\frac{c_0\delta}{Q}\right)^{-(j-1)/(k+1)}=\left(\frac{c_0\delta}{Q}\right)^{(k-j+2)/(k+1)}
\leq\left(\frac{c_0\delta}{Q}\right)^{2/(k+1)}.
\end{aligned}
\]
The last inequality uses $j\leq k$ and $c_0\delta/Q\leq1$.
This proves the claim.

To complete the induction step, sum the bounds from the claim:
\[
\begin{aligned}
|t_{j-1}-t_k|
&\leq\sum_{\ell=j-1}^{k-1}|t_\ell-t_{\ell+1}|\\
&\lesssim_n\left(\frac{c_0\delta}{Q}\right)^{2/(k+1)}
\leq\left(\frac{c_0}{Q}\right)^{1/n}
\left(\frac{c_0\delta}{Q}\right)^{1/(k+1)}.
\end{aligned}
\]
Here we used $k\leq n-1$ and
$(c_0\delta/Q)^{1/(k+1)}\leq(c_0/Q)^{1/n}$. This gives the asserted bounds for $t_{j-1}$
and completes the induction.

For every $x\in F^{-1}\bigl([p_j-c_0\delta,p_j+c_0\delta]\bigr)$, we have
\[
x\asymp_n q\left(\frac{c_0\delta}{Q}\right)^{(j-1)/(k+1)},
\]
and therefore
\begin{equation}\label{eq:sharp-F-derivative}
F'(x)
=\frac{j}{j-1}\left(\frac{x}{q}\right)^{1/(j-1)}
\asymp_n\left(\frac{c_0\delta}{Q}\right)^{1/(k+1)}.
\end{equation}
Since $F$ is strictly increasing, the interval in
\eqref{eq:sharp-inverse-interval} has endpoints
$F^{-1}(p_j-c_0\delta)$ and $F^{-1}(p_j+c_0\delta)$.
The mean value theorem gives a point
$\xi\in F^{-1}\bigl([p_j-c_0\delta,p_j+c_0\delta]\bigr)$ such that
\[
F'(\xi)\bigl(F^{-1}(p_j+c_0\delta)-F^{-1}(p_j-c_0\delta)\bigr)
\asymp c_0\delta.
\]
Using \eqref{eq:sharp-F-derivative}, we therefore obtain
\begin{equation}\label{eq:sharp-I-length}
\begin{aligned}
\left|F^{-1}\bigl([p_j-c_0\delta,p_j+c_0\delta]\bigr)\right|
&\asymp\frac{c_0\delta}{F'(\xi)}
\asymp_n(\delta^kQ)^{1/(k+1)}.
\end{aligned}
\end{equation}
Our assumption on $\delta^kQ$ ensures that each interval contains
$\gtrsim_n(\delta^kQ)^{1/(k+1)}$ integers, proving the counting estimate.
\end{proof}

\appendix
\section{Stationary phase method}

We recall \cite[Theorem~3.6]{CSST2026} here.

\begin{lemma}[Stationary phase with parameters]
\label{lem:stationary-phase-parameters}
Let $Y\subset\mathbb R$ and $Z\subset\mathbb R^{d_0}$ be open sets of
diameter at most $10$, and let
$\phi\in C^\infty(Y\times Z;\mathbb R)$. Suppose that there is a smooth
map $y_*:Z\to Y$ such that, for every $\mathbf z\in Z$, the point
$y_*(\mathbf z)$ is the unique critical point of
$y\mapsto\phi(y;\mathbf z)$ in $Y$, and
\[
\bigl|\partial_y^2\phi(y_*(\mathbf z);\mathbf z)\bigr|\geq c_0>0.
\]
Let
$w\in C^\infty(\mathbb R_{>0}\times\mathbb R\times\mathbb R^{d_0})$
be supported in $\mathbb R_{>0}\times Y\times Z$. If, for some
integer $A>1$,
\[
\sup_{\Lambda>0,\,y,\,\mathbf z}
\bigl|\partial_\Lambda^\iota\partial_y^N
\partial_{\mathbf z}^{\boldsymbol\beta}w(\Lambda;y;\mathbf z)\bigr|
\leq C_A(1+\Lambda)^{-\iota}
\quad
(0\leq N,\iota,|\boldsymbol\beta|\leq A),
\]
then, for $\Lambda\geq1$,
\[
\int_{\mathbb R}e\!\left(\Lambda\phi(y;\mathbf z)\right)
 w(\Lambda;y;\mathbf z)\,dy
=\Lambda^{-1/2}
 e\!\left(\Lambda\phi(y_*(\mathbf z);\mathbf z)\right)
 w_1(\Lambda;\mathbf z),
\]
where
\[
\bigl|\partial_\Lambda^\iota
\partial_{\mathbf z}^{\boldsymbol\beta}w_1(\Lambda;\mathbf z)\bigr|
\lesssim_{A,c_0,C_A,\phi,y_*}\Lambda^{-\iota}
\quad(0\leq\iota,|\boldsymbol\beta|\leq A).
\]
For a family of phases, the last constant is uniform provided the
finitely many required seminorms of $\phi$ and $y_*$ are uniform.

\end{lemma}

\section{Fourier decomposition in derivative coordinates}
\renewcommand{\theequation}{\thesection.\arabic{equation}}
\label{app:derivative-fourier}

We prove Propositions~\ref{prop:n-fourier-input} and
\ref{prop:n-overlap-input} with the derivative conditions in
\eqref{eq:n-plank}. For Proposition~\ref{prop:n-fourier-input}, we first
record the pointwise bound obtained from the Fourier decomposition of
Gan--Maldague--Oh~\cite{GMO2025}. We then reduce the proof to covering
its Frenet boxes by the boxes in \eqref{eq:n-plank}. Proposition~\ref{prop:n-overlap-input} is proved by a direct argument.
We use the notation of Section~\ref{sec:fourier-decay-n}.

\subsection{Reduction to a covering of Frenet boxes}

Choose $\eta>0$ sufficiently small in terms of $\beta$ and $n$, with
$2\eta<\beta$. By a partition of unity, we may assume that $\chi$ is supported
in a sufficiently small fixed neighborhood of some $t_*$. Recall that $\widetilde\gamma(t)=(1,\gamma(t))\in\mathbb R^{n+1}$.
Nondegeneracy makes the matrix
\[
B=\bigl(\widetilde\gamma(t_*),\widetilde\gamma'(t_*),\ldots,
\widetilde\gamma^{(n)}(t_*)\bigr)
\]
invertible. Under the changes of coordinates
$\widetilde\gamma\mapsto B^{-1}\widetilde\gamma$ and
$\xi\mapsto B^T\xi$, we have
\begin{equation}\label{eq:app-affine-phase}
\begin{aligned}
B^{-1}\widetilde\gamma^{(k)}(t_*)&=\mathbf e_k
\quad(0\leq k\leq n),\\
(B^T\xi)\cdot\bigl(B^{-1}\widetilde\gamma^{(k)}(t)\bigr)
&=P_\xi^{(k)}(t).
\end{aligned}
\end{equation}
Here $\mathbf e_0,\ldots,\mathbf e_n$ denote the standard basis
of $\mathbb R^{n+1}$.

Apply Gram--Schmidt to the ordered vectors
\[
B^{-1}\widetilde\gamma(t),\quad B^{-1}\widetilde\gamma'(t),\quad\ldots,
\quad B^{-1}\widetilde\gamma^{(n)}(t)
\]
and denote the resulting orthonormal frame by $f_0(t),\ldots,f_n(t)$.
Following \cite[Section~3.3]{GMO2025}, choose scales satisfying
\[
R^{-1/2}=\Delta_1<\cdots<\Delta_{M+1}=R^{-1/n},
\qquad \Delta_{m+1}\leq R^\eta\Delta_m\quad(1\leq m\leq M),
\]
where $M$ depends only on $\eta$ and $n$.
We group the admissible scale vectors $\boldsymbol\lambda$
according to the size of $\lambda_{\mathrm{prod}}$.
As in \cite[Definition~3.15]{GMO2025}, define
\[
\begin{aligned}
\mathcal D_1&:=\{\boldsymbol\lambda\text{ admissible}:
\Delta_1\leq\lambda_{\mathrm{prod}}\leq\Delta_2\},\\
\mathcal D_m&:=\{\boldsymbol\lambda\text{ admissible}:
\Delta_m<\lambda_{\mathrm{prod}}\leq\Delta_{m+1}\}
\quad(2\leq m\leq M).
\end{aligned}
\]
Thus each admissible $\boldsymbol\lambda$ belongs to exactly one
$\mathcal D_m$, and
\[
\boldsymbol\lambda\in\mathcal D_m
\quad\Longrightarrow\quad
\Delta_m\leq\lambda_{\mathrm{prod}}\leq R^\eta\Delta_m.
\]
Let $\mathcal T_m$ denote the centers of the parameter intervals of
length a fixed multiple of $\Delta_m$ in that construction.
The signed indices are
$\boldsymbol\nu\in\{-1,0,1\}^{n-1}$ with
$\nu_i=0$ if and only if $\lambda_i=1$, as in
\cite[Definition~3.1]{GMO2025}.

Fix a sufficiently large constant $C\geq1$, independent of
$R,Q,\delta,m,a$ and $\boldsymbol\lambda$.
For $a\in\mathcal T_m$ and $\boldsymbol\lambda\in\mathcal D_m$, define
\[
\mathcal P_{a,\boldsymbol\lambda,\boldsymbol\nu}
:=\left\{\zeta\in\mathbb R^{n+1}:
\begin{aligned}
|\zeta\cdot f_0(a)|&\leq CR^{2\eta},\\
|\zeta\cdot f_1(a)|&\leq CR^\eta\Delta_m^{-1},\\
|\zeta\cdot f_k(a)|&\leq CQL_k\quad(2\leq k\leq n),\\
\nu_{n+1-k}\,\zeta\cdot f_k(a)&\geq C^{-1}QL_k
\quad(\nu_{n+1-k}\ne0)
\end{aligned}
\right\}.
\]
The last condition is imposed for $2\leq k\leq n$ with
$\nu_{n+1-k}\ne0$; when $\nu_{n+1-k}=0$, only the upper bound is imposed.
The dependence on $m$ is determined by $\boldsymbol\lambda$.
These are fixed enlargements of the signed thickened planks in
\cite[Definition~3.17]{GMO2025}, chosen to contain the cutoff supports.
Indeed, coordinate $i$ in its reversed Frenet frame is $f_{n+1-i}$ here,
and its width is
\[
R\prod_{r=1}^{i-1}\lambda_r^{i-r}=QL_{n+1-i}.
\]

The following is a consequence of
\cite[Theorem~3.22 and Lemma~3.20]{GMO2025}, written in our notation.

\begin{lemma}\label{lem:app-gmo-pointwise}
With the definitions above, for every $L>0$, $\e>0$ and $|\xi|\sim R$,
\begin{equation}\label{eq:app-gmo-pointwise}
|\widehat\mu_{\omega,\chi}(\xi)|
\lesssim_{\eta,L,\e}R^{\e}
\sum_m\Delta_m\sum_{a\in\mathcal T_m}
\sum_{\boldsymbol\lambda\in\mathcal D_m,\,\boldsymbol\nu}
\mathbf1_{\mathcal P_{a,\boldsymbol\lambda,\boldsymbol\nu}}(B^T\xi)
+R^{-L}.
\end{equation}
\end{lemma}

\begin{proof}
Write $\zeta=B^T\xi$. Let $\mu_{\om,\chi}^B$ be the measure defined as in \eqref{defmu} with $\wt\ga$ replaced by $B^{-1}\wt\ga$. 
Write the Fourier transform 
\[
\begin{aligned}
\widehat\mu_{\omega,\chi}(\xi)
&=\iint \omega(s)\chi(t)
 e\bigl(-s\,\xi\cdot\widetilde\gamma(t)\bigr)\,dt\,ds\\
&=\iint \omega(s)\chi(t)
 e\bigl(-s\,(B^T\xi)\cdot(B^{-1}\widetilde\gamma(t))\bigr)\,dt\,ds\\
 &=\wh{\mu_{\om,\chi}^B}(\zeta).
\end{aligned}
\]
The frame
$f_0,\ldots,f_n$ defined above is the Frenet frame for $B^{-1}\wt\ga(t)$. We may
therefore apply the argument in \cite[Section 3.4]{GMO2025} to obtain an estimate for $\wh{\mu_{\om,\chi}^B}(\zeta)$.
We briefly explain it below.

We refer to the argument from (3.44) to (3.47) in \cite[Section 3.4]{GMO2025}, it was proved that for any $\e>0$,
$|\wh \mu_\Ga(\xi)|$ with $|\xi|\sim R$ can be controlled by 
\[R^\e\sum_{I_m\in \bI_m}|\wh\mu_{I_m}(\xi)|\phi_{P[\vec\de,\vec\nu](s_{I_m})}(\xi)\lesssim R^\e\sum_{I_m\in \bI_m}\De_m \Id_{P[\vec\de,\vec\nu](s_{I_m})}(\xi)\]
for some choice of $m$ and $(\vec\de,\vec\nu)$. We briefly explain the proof. By \cite[Theorem 3.22]{GMO2025}, we can bound $\wh \mu_\Ga$ by
\[ \sum_{I_1<\dots<I_M} \wh\mu_{I_1,\dots,I_M}\bigg(\sum_{k}\psi_{I_k,I_{k+1}}+\phi_{I_M}\bigg). \]
Then recall \cite[Lemma 3.20 (1),(3)]{GMO2025} that $\psi_{I_kI_{k+1}}$ and $\phi_{I_M}$ can be bounded by the sum of indicator functions of $(\vec\de,\vec\nu)$-planks. We remark that there is a single $m$ and $(\vec\de,\vec\nu)$ in (3.47) of \cite{GMO2025} because a pigeonhole argument was applied. Finally, we use (3.45) in \cite{GMO2025} to see that $|\wh\mu_{I_m}|\lesssim \De_m$, and note that $\phi_{P[\vec\de,\vec\nu](s_{I_m})}$ is a bump function of $P[\vec\de,\vec\nu](s_{I_m})$.

We translate this estimate to the notation in this paper. $\mu_\Ga$ becomes $\mu_{\om,\chi}^B$; $P[\vec\delta,\vec\nu](s_{I_m})$ becomes $\mathcal P_{a,\bla,\bnu}$.
Also, we include all the $\mathcal P_{a,\bla,\bnu}$ in the upper bound.
After translating to the notation in this paper, we obtain
\eqref{eq:app-gmo-pointwise}.
\end{proof}

Set $A=\{i:\nu_i\ne0\}$. The definition above gives the required lower bounds at every order $k$ with $n+1-k\in A$, and
$i\notin A\Rightarrow\lambda_i=1$, as in \eqref{eq:n-admissibility}.
Discarding the signs requires only finitely many choices.

\begin{lemma}\label{lem:app-frenet-covering}
For each Frenet box $\mathcal P_{a,\boldsymbol\lambda,\boldsymbol\nu}$
above, set $A=\{i:\nu_i\ne0\}$. Then, for $|\xi|\sim R$,
\[
\mathbf1_{\mathcal P_{a,\boldsymbol\lambda,\boldsymbol\nu}}(B^T\xi)
\leq\sum_{\substack{\tau\in\mathcal I(\lambda_{\mathrm{prod}})\\
|t_\tau-a|\leq C_n\lambda_{\mathrm{prod}}}}
\mathbf1_{\Pi_{\tau,\boldsymbol\lambda,A}}(\xi).
\]
\end{lemma}
Assuming Lemma~\ref{lem:app-frenet-covering}, let us finish the proof of
Proposition~\ref{prop:n-fourier-input}. For a fixed
$\tau$, the separation of $\mathcal T_m$ gives
\[
\Delta_m\#\{a\in\mathcal T_m:|t_\tau-a|\leq C_n\lambda_{\mathrm{prod}}\}
\lesssim\Delta_m\left(1+\frac{\lambda_{\mathrm{prod}}}{\Delta_m}\right)
\lesssim\lambda_{\mathrm{prod}}.
\]
Substituting the covering in Lemma~\ref{lem:app-frenet-covering} into
the estimate of Lemma~\ref{lem:app-gmo-pointwise}, we obtain
\[
|\widehat\mu_{\omega,\chi}(\xi)|
\lesssim_{\beta,L,\gamma,\omega,\chi}
R^{O(\beta)}
\sum_{(\boldsymbol\lambda,A)\ \mathrm{admissible}}
\lambda_{\mathrm{prod}}\sum_{\tau\in\mathcal I(\lambda_{\mathrm{prod}})}
\mathbf1_{\Pi_{\tau,\boldsymbol\lambda,A}}(\xi)+R^{-L}.
\]
This is the desired bound. It remains to prove Lemma~\ref{lem:app-frenet-covering}.
We first compare Frenet coordinates with the derivatives of $P_\xi$
at $a$, and then pass from $a$ to a prescribed center $t_\tau$.

\bigskip

\subsection{Proof of Lemma~\ref{lem:app-frenet-covering}}

\noindent\emph{Step 1: Comparison of coordinates.}

Fix a Frenet box $\mathcal P_{a,\boldsymbol\lambda,\boldsymbol\nu}$
at level $m$ and suppose that
$B^T\xi\in\mathcal P_{a,\boldsymbol\lambda,\boldsymbol\nu}$.
We will prove the bounds in \eqref{eq:app-annular-transfer}
and \eqref{eq:app-low-transfer}, together with
\[
|P_\xi^{(k)}(a)|\lesssim QL_k\qquad(2\leq k\leq n).
\]
These estimates express the Frenet box conditions in terms of the
derivatives of $P_\xi$ at $a$. 

To prove these estimates, recall the scale identities from
Section~\ref{sec:fourier-decay-n}:
\begin{equation}\label{eq:app-scale-identities}
\begin{aligned}
&QL_0=1,\qquad QL_1=\lambda_{\mathrm{prod}}^{-1},\qquad
QL_2=\lambda_{\mathrm{prod}}^{-2},\qquad QL_n=R,\\
&R^{-1/2}\leq\lambda_{\mathrm{prod}}\leq R^{-1/n}.
\end{aligned}
\end{equation}

On a sufficiently small fixed neighborhood of $t_*$,
\begin{equation}\label{eq:app-triangular-frame}
f_k(t)=\sum_{r=0}^k c_{kr}(t)B^{-1}\widetilde\gamma^{(r)}(t),
\qquad c_{kr}(t)=\delta_{kr}+O(\varepsilon_0),
\end{equation}
where $\varepsilon_0>0$ can be chosen as small as needed.

Use the defining bounds of the Frenet box and
$\Delta_m\leq\lambda_{\mathrm{prod}}\leq R^\eta\Delta_m$.
For large $R$,
\[
\begin{aligned}
\frac{R^{2\eta}}{QL_2}
&=R^{2\eta}\lambda_{\mathrm{prod}}^2\lesssim1,\\
\frac{R^\eta\Delta_m^{-1}}{QL_2}
&=R^\eta\frac{\lambda_{\mathrm{prod}}}{\Delta_m}
\lambda_{\mathrm{prod}}
\leq R^{2\eta}\lambda_{\mathrm{prod}}\lesssim1.
\end{aligned}
\]
Also note that $L_2\leq\cdots\leq L_n$. Taking the inverse of the 
system
\[
(B^T\xi)\cdot f_k(a)=\sum_{r=0}^k c_{kr}(a)P_\xi^{(r)}(a)
\]
yields
\begin{equation}\label{system}
    P_\xi^{(k)}(a)=\sum_{r=0}^k c'_{kr}(a) (B^T\xi)\cdot f_r(a), \qquad c'_{kr}(t)=\de_{kr}+O(\varepsilon_o). 
\end{equation} 
Together with the previous bounds, it implies
\[
\bigl|P_\xi^{(k)}(a)-(B^T\xi)\cdot f_k(a)\bigr|
\lesssim\varepsilon_0 QL_k\quad(k\geq2)
\]
Thus, after choosing
$\varepsilon_0$ small,
\begin{equation}\label{eq:app-annular-transfer}
|(B^T\xi)\cdot f_k(a)|\gtrsim QL_k
\ \Longrightarrow\ |P_\xi^{(k)}(a)|\gtrsim QL_k
\quad(k\geq2).
\end{equation}
The same comparison gives the claimed upper bounds for $2\leq k\leq n$.
For $k=0,1$, the system \eqref{system} gives
\begin{equation}\label{eq:app-low-transfer}
|P_\xi(a)|\lesssim R^{2\eta},\qquad
|P_\xi'(a)|\lesssim R^{2\eta}\lambda_{\mathrm{prod}}^{-1}.
\end{equation}

\medskip
\noindent\emph{Step 2: Choice of the interval center.}

Fix a Frenet box at level $m$ and a frequency
$B^T\xi\in\mathcal P_{a,\boldsymbol\lambda,\boldsymbol\nu}$.
The preceding comparison gives
\[
\begin{aligned}
|P_\xi^{(k)}(a)|&\lesssim QL_k\quad(2\leq k\leq n),\\
|P_\xi^{(k)}(a)|&\gtrsim QL_k\quad(n+1-k\in A),\\
|P_\xi(a)|+\lambda_{\mathrm{prod}}|P_\xi'(a)|&\lesssim R^{2\eta}.
\end{aligned}
\]
We must find $\tau\in\mathcal I(\lambda_{\mathrm{prod}})$ with
$|t_\tau-a|\lesssim_n\lambda_{\mathrm{prod}}$ for which these bounds
still hold with $a$ replaced by $t_\tau$. 
The idea is similar to Lemma \ref{lemcompare}: if we perturb the parameter of the plank within $\laprod$, then the two planks are comparable. We provide detail below.

We use the scale inequalities
\begin{equation}\label{eq:app-scale-taylor}
\begin{aligned}
\lambda_{\mathrm{prod}}^rL_{k+r}&\leq L_k
\quad(0\leq k\leq k+r\leq n),\\
\frac{R\lambda_{\mathrm{prod}}^{n+1-k}}{QL_k}
&\leq\lambda_{\mathrm{prod}}\quad(0\leq k\leq n).
\end{aligned}
\end{equation}
Indeed,
\[
\lambda_{\mathrm{prod}}\frac{L_{k+1}}{L_k}
=\prod_{i=n-k+1}^{n-1}\lambda_i\leq1,
\qquad
\frac{R\lambda_{\mathrm{prod}}^{n+1-k}}{QL_k}
=\lambda_{\mathrm{prod}}\frac{\lambda_{\mathrm{prod}}^{n-k}L_n}{L_k}
\leq\lambda_{\mathrm{prod}},
\]
with empty products equal to one.

For each $k$ with $n+1-k\in A$, write Taylor's formula as
\begin{equation}\label{eq:app-recenter-polynomial}
\begin{aligned}
\frac{P_\xi^{(k)}(a+\lambda_{\mathrm{prod}}u)}{QL_k}
&=p_k(u)+O_\gamma(\lambda_{\mathrm{prod}}),\\
p_k(u)&:=\sum_{r=0}^{n-k}
\frac{\lambda_{\mathrm{prod}}^rP_\xi^{(k+r)}(a)}{r!QL_k}u^r.
\end{aligned}
\end{equation}
The remainder is uniform for bounded $u$, and
\[
\deg p_k\leq n-k,\qquad
|p_k(0)|=\frac{|P_\xi^{(k)}(a)|}{QL_k}\gtrsim1.
\]
Choose $n^2+1$ consecutive centers $t_{\tau_j}$ near $a$ and set
\[
u_j:=\frac{t_{\tau_j}-a}{\lambda_{\mathrm{prod}}},\qquad
|u_j|\lesssim_n1,\qquad |u_i-u_j|\gtrsim1\quad(i\ne j).
\]
Near an endpoint, choose the centers on the available side.
For any $n-k+1$ of these centers, Lagrange interpolation gives
\[
\begin{aligned}
p_k(0)
&=\sum_{\ell=0}^{n-k}p_k(u_{j_\ell})
\prod_{\substack{0\leq h\leq n-k\\h\ne\ell}}
\frac{-u_{j_h}}{u_{j_\ell}-u_{j_h}},\\
1&\lesssim |p_k(0)|\lesssim_n
\max_{0\leq\ell\leq n-k}|p_k(u_{j_\ell})|.
\end{aligned}
\]
Thus, for a sufficiently small fixed $c>0$,
\[
\#\{1\leq j\leq n^2+1:|p_k(u_j)|<c\}\leq n-k.
\]
Taking the union over the required derivative orders, we obtain
\[
\begin{aligned}
&\#\{1\leq j\leq n^2+1:
|p_k(u_j)|<c\text{ for some }k\text{ with }n+1-k\in A\}\\
&\qquad\leq\sum_{k=2}^n(n-k)
=\frac{(n-1)(n-2)}2<n^2+1.
\end{aligned}
\]
Hence one of the centers, denoted by $t_\tau$, satisfies the polynomial
lower bounds simultaneously. For sufficiently large $R$, \eqref{eq:app-recenter-polynomial} gives
\[
\begin{aligned}
\frac{|P_\xi^{(k)}(t_\tau)|}{QL_k}
&\geq\left|p_k\!\left(\frac{t_\tau-a}{\lambda_{\mathrm{prod}}}\right)\right|
-C_\gamma\lambda_{\mathrm{prod}}\geq c-C_\gamma\lambda_{\mathrm{prod}}\geq\frac c2
\qquad(n+1-k\in A).
\end{aligned}
\]
Taylor's formula and \eqref{eq:app-scale-taylor} also give
\[
\begin{aligned}
|P_\xi^{(k)}(t_\tau)|
&\lesssim\sum_{r=0}^{n-k}\lambda_{\mathrm{prod}}^r
|P_\xi^{(k+r)}(a)|+R\lambda_{\mathrm{prod}}^{n+1-k}\lesssim QL_k\qquad(2\leq k\leq n).
\end{aligned}
\]
For the two lowest orders, the same expansion yields
\[
\begin{aligned}
|P_\xi(t_\tau)|+\lambda_{\mathrm{prod}}|P_\xi'(t_\tau)|
&\lesssim |P_\xi(a)|+\lambda_{\mathrm{prod}}|P_\xi'(a)|+\sum_{k=2}^n\lambda_{\mathrm{prod}}^k|P_\xi^{(k)}(a)|
+R\lambda_{\mathrm{prod}}^{n+1}\\
&\lesssim R^{2\eta}+1\lesssim R^{2\eta},
\end{aligned}
\]
where we used $\lambda_{\mathrm{prod}}^kQL_k\leq QL_0=1$.
Since $2\eta<\beta$ and $QL_1=\lambda_{\mathrm{prod}}^{-1}$,
these are the bounds required in \eqref{eq:n-plank}, so
$\xi\in\Pi_{\tau,\boldsymbol\lambda,A}$.
The case of bounded $R$ is absorbed into the implicit constants.
This proves Lemma~\ref{lem:app-frenet-covering} and completes the proof of
Proposition~\ref{prop:n-fourier-input}.

\bigskip

\subsection{Proof of Proposition~\ref{prop:n-overlap-input}}

Fix $0<\beta<1/(10n)$, an admissible $(\boldsymbol\lambda,A)$, and
$\xi$ with $|\xi|\sim R$, using the notation of Section~\ref{sec:fourier-decay-n}. Since we fixed $\boldsymbol\lambda$, the number $\lambda_{\mathrm{prod}}$ is also fixed. 
Recall that $I\subset\mathbb R$ is the compact parameter interval of $\gamma$.
We claim that, if $2\leq k\leq n$ and $I_0\subset I$ is a finite set of
$\lambda_{\mathrm{prod}}$-separated points satisfying
\[
|P_\xi(t)|\lesssim R^\beta,\qquad |P_\xi^{(k)}(t)|\gtrsim \lambda_{\mathrm{prod}}^{-k}
\]
for any $t \in I_0$,
then
\begin{equation}\label{eq:app-derivative-packing}
\#I_0\lesssim_{\beta,\gamma}R^{2\beta}.
\end{equation}

Assuming this claim, let us finish the proof. Choose $k=n+1-\max A\in\{2,\ldots,n\}$.
By admissibility (see Definition \ref{def41}), $\lambda_i=1$ for $i>\max A$.
Recall that $R=Q\delta^{-1}$ and $\lambda_n=1$.
By \eqref{eq:n-derivative-scales} and \eqref{eq:n-scale-normalization},
\begin{equation}\label{eq:app-active-scale}
\begin{aligned}
QL_0
&=R\prod_{i=1}^{n}\lambda_i^{\,n+1-i}
=R\prod_{i=1}^{n-1}\lambda_i^{\,n+1-i}
=1,\\
QL_k
&=R\prod_{i=1}^{n-k}\lambda_i^{\,n+1-k-i}
=R\prod_{i=1}^{n-1}\lambda_i^{\,n+1-k-i}\\
&=\lambda_{\mathrm{prod}}^{-k}
\left(R\prod_{i=1}^{n-1}\lambda_i^{\,n+1-i}\right)
=\lambda_{\mathrm{prod}}^{-k}.
\end{aligned}
\end{equation}
In the product for $QL_k$, the added factor at $i=n+1-k=\max A$
has exponent zero, and all subsequent $\lambda_i$ equal $1$.
By \eqref{eq:n-plank}, the $\lambda_{\mathrm{prod}}$-separated set
$I_0=\{t_\tau:\xi\in\Pi_{\tau,\boldsymbol\lambda,A}\}$ satisfies
\[
|P_\xi(t_\tau)|\lesssim R^\beta,\qquad
|P_\xi^{(k)}(t_\tau)|\gtrsim QL_k=\lambda_{\mathrm{prod}}^{-k}.
\]
Applying \eqref{eq:app-derivative-packing} gives
\[
\sum_{\tau\in\mathcal I(\lambda_{\mathrm{prod}})}\mathbf1_{\Pi_{\tau,\boldsymbol\lambda,A}}(\xi)
=\#I_0\lesssim_{\beta,\gamma}R^{2\beta},
\]
which proves Proposition~\ref{prop:n-overlap-input}.

\begin{proof}[Proof of the claim]
The idea of the proof is to approximate the function $P_{\xi}$ by a polynomial of a fixed degree, and use the fundamental theorem of algebra. To do so, choose an integer $M\geq n+\lceil2/\beta\rceil$ and partition $I$ into
$O_{\beta,\gamma}(R^\beta)$ intervals of length at most $cR^{-\beta}$,
where $0<c\leq1$ is a sufficiently small fixed constant. On each such
interval $J$, let $p$ be the Taylor polynomial of order $M$ for $P_\xi$ at the
left endpoint. Since $|\xi|\sim R$ and $\gamma$ is smooth, Taylor's formula gives
\[
\max_{0\leq j\leq n}\|P_\xi^{(j)}-p^{(j)}\|_{L^\infty(J)}
\lesssim_{\beta,\gamma}R(cR^{-\beta})^{M+1-n}
\lesssim_{\beta,\gamma}R^{-1}.
\]
Since $0<\lambda_i\leq1$, we have $\lambda_{\mathrm{prod}}\leq1$.
Taking $c$ sufficiently small, we therefore have
\[
|p(t)|\lesssim R^\beta,\qquad |p^{(k)}(t)|\gtrsim \lambda_{\mathrm{prod}}^{-k}
\quad(t\in I_0\cap J).
\]
If $I_0\cap J$ is nonempty, the preceding lower bound implies that $p$ is nonconstant.
By the fundamental theorem of algebra, we may write
\[
p(t)=a\prod_{\ell=1}^{\deg p}(t-z_\ell),\qquad \deg p\leq M,
\]
where $z_1,\ldots,z_{\deg p}\in\mathbb C$ are the roots, counted with multiplicity.
For $p(t)\neq0$, the product rule gives
\[
|p^{(k)}(t)|\lesssim_M
|p(t)|\bigl(\min_{1\leq\ell\leq\deg p}|t-z_\ell|\bigr)^{-k}.
\]
Consequently every $t\in I_0\cap J$, including any zero of $p$, satisfies
\[
\min_{1\leq\ell\leq\deg p}|t-z_\ell|
\lesssim_{\beta,\gamma}R^{\beta/k}\lambda_{\mathrm{prod}}
\leq R^\beta \lambda_{\mathrm{prod}}.
\]
Thus $I_0\cap J$ is covered by at most $M$ real intervals centered at
$\operatorname{Re}z_\ell$, each of length $\lesssim_{\beta,\gamma}R^\beta \lambda_{\mathrm{prod}}$.
Since $I_0$ is $\lambda_{\mathrm{prod}}$-separated,
\[
\#(I_0\cap J)\lesssim_{\beta,\gamma}1+R^\beta
\lesssim_{\beta,\gamma}R^\beta.
\]
Summing over the $O_{\beta,\gamma}(R^\beta)$ intervals $J$ proves
\eqref{eq:app-derivative-packing}.
\end{proof}

\bibliographystyle{alpha}
\bibliography{newreference.bib}

\end{document}